\documentclass[11pt]{article}
\usepackage{paper,asb}

\usepackage{amsmath}
\usepackage{amsfonts}
\usepackage{amssymb}
\usepackage{amsthm}
\usepackage{hyperref}
\usepackage{subcaption}
\usepackage{cite}

\newtheorem{theorem}{Theorem}[section]

\newtheorem{lemma}[theorem]{Lemma}
\newtheorem{assumption}[theorem]{Assumption}
\newtheorem{remark}[theorem]{Remark}

\numberwithin{equation}{section}

\usepackage{booktabs}
\usepackage{caption}
\usepackage{float}
\usepackage{titlesec}
\usepackage{capt-of}
\usepackage{arydshln}

\newcommand{\RomanNumeralCaps}[1]
    {\MakeUppercase{\romannumeral #1}}
    
\author{Albert S. Berahas \and Raghu Bollapragada \and  Jiahao Shi}
\date{}

\newcommand{\papertitle}{Modified Line Search Sequential Quadratic Methods for Equality-Constrained Optimization with Unified Global and Local Convergence Guarantees}
\newcommand{\paperauthora}{Albert S. Berahas}
\newcommand{\paperauthoraaffiliation}{Department of Industrial and Operations Engineering, University of Michigan}
\newcommand{\paperauthoraemail}{aberahas@umich.edu}
\newcommand{\paperauthorb}{Raghu Bollapragada}
\newcommand{\paperauthorbaffiliation}{Operations Research Program, University of Texas at Austin}
\newcommand{\paperauthorbemail}{raghu.bollapragada@utexas.edu}
\newcommand{\paperauthorc}{Jiahao Shi}

\newcommand{\paperauthorcemail}{jiahaos@umich.edu}

\title{\papertitle}

\author{\paperauthora\footnotemark[2] \footnotemark[1]
   \and \paperauthorb\footnotemark[3]
   \and \paperauthorc\footnotemark[2]\ }

\begin{document}

\maketitle

\renewcommand{\thefootnote}{\fnsymbol{footnote}}
\footnotetext[2]{\paperauthoraaffiliation. (\url{\paperauthoraemail}, \url{\paperauthorcemail})}
\footnotetext[3]{\paperauthorbaffiliation. (\url{\paperauthorbemail})}
\footnotetext[1]{Corresponding author.}
\renewcommand{\thefootnote}{\arabic{footnote}}

\begin{abstract}{
In this paper, we propose a method that has foundations in the line search sequential quadratic programming paradigm for solving general nonlinear equality constrained optimization problems. The method employs a carefully designed modified line search strategy that utilizes second-order information of both the objective and constraint functions, as required, to mitigate the Maratos effect. Contrary to classical line search sequential quadratic programming methods, our proposed method is endowed with global convergence and local superlinear convergence guarantees. Moreover, we extend the method and analysis to the setting in which the constraint functions are deterministic but the objective function is stochastic or can be represented as a finite-sum. We also design and implement a practical inexact matrix-free variant of the method. Finally, numerical results illustrate the efficiency and efficacy of the method.}
\end{abstract}

\maketitle







\section{Introduction}

In this paper, we propose a 
sequential quadratic programming (SQP) algorithm for solving equality constrained optimization problems of the form
\begin{align}
    \min_{x \in \mathbb{R}^n} \;f(x) \quad \text{ s.t. } c(x) = 0,
\label{problem.deterministic}
\end{align}
where $f: \mathbb{R}^n \rightarrow \mathbb{R}$ and $c: \mathbb{R}^n \rightarrow \mathbb{R}^m$ are twice continuously differentiable nonlinear functions. Problems of this structure arise in a plethora of science and engineering applications; see \cite{rao2019engineering,deb2012optimization} and the references therein. Our proposed algorithm has foundations in the line search SQP paradigm, utilizes second-order information of 
the objective and constraint functions, and employs a carefully designed modified line search strategy. 
The 
method is endowed with both global convergence and local fast (superlinear) convergence guarantees. This is contrary to classical line search (and other) SQP methods that assume proximity of the initial iterate to the solution and a unit step size to prove analogous results; see e.g., ~\cite{han1976superlinearly,powell1978convergence,palomares1976superlinearly}.

Moreover, we extend the modified line search SQP methodology to the equality constrained stochastic setting, i.e., \begin{align}
    \min_{x \in \mathbb{R}^n} \; f(x) = \mathbb{E}[F(x,\omega)] \quad \text{ s.t. }  c(x) = 0,
\label{problem.expected_risk}
\end{align}
where $\omega$ is a random variable with probability space $(\Omega, \mathcal{F}, P)$ that represents the stochasticity in the problem, $F: \mathbb{R}^n \times \Omega \rightarrow \mathbb{R}$, and $\mathbb{E}[\cdot]$ is the expectation taken with respect to the distribution of $\omega$. In this problem setting, it is assumed that one has access to exact values of the constraint function and its derivatives, but only stochastic estimates of the objective function and its derivatives. Such problems arise in numerous 
application areas including machine learning \cite{marquez2017imposing, roy2018geometry}, optimal power flow \cite{ wood2013power, vrakopoulou2014stochastic}, statistical estimation \cite{geyer1991constrained, chatterjee2016constrained}, portfolio optimization \cite{uryasev2013stochastic, ziemba2014stochastic}, PDE-constrained optimization \cite{rees2010optimal,biegler2003large}, and simulation-based optimization \cite{andradottir1998simulation, fu1994optimization}. As a special case, and since the distribution of $\omega$ is often unknown in practice, we also consider the empirical risk variant \cite{vapnik1991principles} of \eqref{problem.expected_risk}, i.e., 
\begin{align}
    \min_{x \in \mathbb{R}^n}  \; f(x) =  \frac{1}{N} \sum_{i=1}^N f_i(x) = \frac{1}{N} \sum_{i=1}^N F(x,\omega_i) \quad \text{ s.t. } c(x) = 0,
\label{problem.empirical_risk}
\end{align}
where 
$N$ denotes the number of samples, and $f_i: \mathbb{R}^n \rightarrow \mathbb{R}$ (for $i \in \{1,2,\dots,N\}$) denote the component functions that are twice continuously differentiable, $\omega_i$ is a random realization generated from the distribution $P$, and $F: \mathbb{R}^n \times \Omega \rightarrow \mathbb{R}$. 
The proposed approaches for solving~\eqref{problem.expected_risk} and~\eqref{problem.empirical_risk} employ inexact objective function 
information, and as such the 
strategies need to be modified appropriately. Under reasonable assumptions on the stochasticity, 
we show local fast (superlinear) convergence rates analogous to those in the deterministic setting for these 
problem classes.

\subsection{Background and Related Work}

As a disclaimer, we note that this is not an extensive review of methods for solving problems of the form~\eqref{problem.deterministic}, \eqref{problem.expected_risk} and~\eqref{problem.empirical_risk}. Rather, it is a concise summary of the works most closely related to our proposed approach. We begin with a detour to the unconstrained (deterministic and stochastic) settings and the fast local convergence of Newton's method, and then discuss SQP methods for equality constrained deterministic and stochastic optimization.

In the  unconstrained  setting, it is well established that Newton's method, under reasonable assumptions, is globally convergent and has fast local convergence~\cite{polak1971computational,NoceWrig06}. This requires a globalization mechanism (e.g., line search) that controls the length of the step when the iterates are far from stationarity, and results in an adaptive algorithm with eventual fast local quadratic convergence. In the unconstrained stochastic setting, several second-order methods have been proposed; see e.g., \cite{berahas2020investigation,roosta2019sub,bollapragada2019exact,erdogdu2015convergence,pilanci2017newton}. We focus on \cite{bollapragada2019exact,roosta2019sub}. In \cite{bollapragada2019exact}, a subsampled Newton method is analyzed, and under appropriate sampling (gradient and Hessian) conditions, global linear and local superlinear convergence guarantees are established in expectation. In \cite{roosta2019sub}, the same method is explored and high probability convergence guarantees are derived. We note that for both approaches (and all other stochastic Newton methods listed above) the fast local results are derived under a proximity to the optimal solution assumption and under the assumption that the unit step size is employed. Our goal is to derive analogous results for the equality constrained setting without the last two assumptions.

Classical deterministic second-order SQP methods have proven robust and efficient in practice across many applications~\cite{schittkowski2012nonlinear,gill2005snopt}, and are endowed with local superlinear convergence guarantees; see e.g., \cite{wilson1963simplicial,han1976superlinearly,powell1978convergence,palomares1976superlinearly}. To establish such fast local guarantees, the results critically rely on a proximity assumption between the starting point and the optimal solution and on the assumption that the unit step size is employed. The former assumption depends on problem specific, often unknown, quantities, and employing the unit step size far from the solution can lead to undesirable behavior. To alleviate these issues and to globalize SQP methods, several techniques have been proposed; see e.g., \cite{berahas2021sequential,han1977globally,powell2006fast,vardi1985trust,byrd1987trust}. That said, unlike the unconstrained setting for which (as mentioned above) globalized versions of Newton's method have provable ``two-phase'' behavior (global convergence and local fast convergence) \cite[Theorem 6.2.31]{polak1971computational}, in the equality constrained setting such results cannot be proven due to the Maratos effect \cite[Section 15.5]{NoceWrig06}. In fact, there exist simple counter-examples for which the conditions required for local fast convergence are never guaranteed to be satisfied for classical globalized SQP methods~\cite[Chapter \RomanNumeralCaps{3}.5]{maratos1978exact}. 
Several remedies have been proposed to mitigate these issues; see e.g.,~\cite{schittkowski1982nonlinear,chamberlain1982watchdog,fukushima1986successive} 
 and  Section~\ref{sec.gap} for more details. On a high level, these approaches either change the algorithm, modify the SQP subproblem, 
 allow for nonmonotonicity or restrict the problem class considered. In this work, 
we propose a second-order SQP method with a 
modified line search condition that allows for both global and fast local convergence guarantees without the unit step size and proximity assumptions.

Several methods have been proposed for solving constrained stochastic optimization problems of the form \eqref{problem.expected_risk} and \eqref{problem.empirical_risk}; see e.g., \cite{berahas2021sequential, berahas2022adaptive,berahas2022accelerating,curtis2021inexact,curtis2023sequential,curtis2023worst,berahas2021stochastic,na2023adaptive,na2023inequality,oztoprak2023constrained,fang2022fully,berahas2023sequential}. While stochastic penalty methods have been proposed for this problem class, we focus here on SQP methods due to their demonstrated empirical superiority in both the deterministic \cite{han1977globally,han1979exact,powell2006fast} and stochastic \cite{berahas2021sequential,na2023adaptive} settings. The proposed SQP-type methods are adaptive (line search, trust region or other), and utilize, insofar as possible, well-established techniques from the deterministic setting. That said, with the exception of~\cite{na2022asymptotic,na2023adaptive}, all methods are essentially first-order methods. We should note that the presentation of most of the algorithms is general and allows for second-order information to be utilized, however, the analyses (even in \cite{na2022asymptotic, na2023adaptive}) do not explicitly account for second-order information and no local fast results are established, and numerical experiments consider only first-order variants. Our goal is to develop a practical stochastic SQP method endowed with local superlinear convergence rates.

\subsection{Summary of Contributions}
\begin{itemize}[leftmargin=0.5cm]
    \item We develop, analyze, and implement a second-order line search SQP method with a novel modified line search condition. The method is designed to solve problems of the form \eqref{problem.deterministic}, and is extended to tackle stochastic problems. Our algorithm leverages the global convergence properties of classical line search SQP methods, and employs a carefully modified line search condition, that incorporates curvature information of both objective and constraint functions, as appropriate and required to mitigate the Maratos effect. The condition that triggers the modified line search is adaptive and does not require knowledge of unknown quantities.
    \item To set the stage for our main theoretical results, we derive local convergence guarantees for the SQP method under the proximity and unit step size assumptions. We include this section for completeness and to compare with existing results in the literature, but most importantly to discuss the sampling requirements. Specifically, we utilize adaptive sampling strategies within the SQP methodology (applicable to both problems \eqref{problem.expected_risk} and \eqref{problem.empirical_risk}) and, under reasonable assumptions, show that locally the primal-dual iterates converge to a first-order stationary point at a superlinear rate. These results are natural extensions of the local convergence results of existing deterministic SQP methodologies.
    \item Our main theoretical contributions are as follows. For problems \eqref{problem.deterministic} and \eqref{problem.empirical_risk}, we prove that the iterates generated by the proposed modified line search SQP method, from any starting point, converge to a first-order stationary point (global convergence) and, moreover, eventually, the unit step size is always accepted and the algorithm achieves a superlinear convergence rate asymptotically (local convergence). Our methodology is simple and adaptive, makes adjustments as needed and detected over the course of the optimization, and has provable fast local convergence guarantees, without requiring knowledge of problem specific parameters or the proximity and unit step size assumptions.
    \item To tackle large-scale problems, 
    we propose a practical inexact matrix-free variant of the modified line search SQP method that employs the minimum residual (MINRES) method \cite{choi2011minres} as a subroutine to compute a step. 
    The method reduces the computational and storage costs and allows for inexact subproblem 
    solutions. Under reasonable conditions and  a sufficient number of MINRES iterations, the method enjoys an asymptotic linear rate of convergence with constant one half.
    \item Finally, we empirically demonstrate that our proposed method is competitive with a classical SQP method and other adaptations on the CUTEst test problems~\cite{bongartz1995cute}. In the stochastic setting, we compare the performance of our proposed method to a stochastic SQP methodology with an adaptive step size selection scheme~\cite{curtis2023sequential} on constrained binary classification logistic regression problems using the LIBSVM datasets~\cite{chang2011libsvm}. We illustrate the efficiency and robustness of the practical variant.
\end{itemize}

\subsection{Paper Organization}
We conclude this section by introducing some notation. 
In Section \ref{sec.gap}, we present a discrepancy in global and local convergence guarantees of line search SQP methods due to the Maratos effect, and examine various remedies. General assumptions and preliminaries are given in Section~\ref{sec.assum_notation}. Our proposed modified line search SQP method is presented in Section~\ref{sec.ourmethod}. In Section~\ref{sec.superlinear}, we present local convergence guarantees. The main theoretical results of the proposed methodology are given in Section~\ref{sec.analysis}. In Section~\ref{sec.inexact}, we introduce and analyze a practical inexact variant for the large scale setting. Numerical results (deterministic and stochastic settings) are presented in Section~\ref{sec.numerical}, and concluding remarks are given in~Section~\ref{sec.remark}.

\subsection{Notation}\label{sec.notation}

Let $\mathbb{R}$ denote the set of real numbers, $\mathbb{R}^n$ denote the set of $n$-dimensional real vectors, $\mathbb{R}^{m \times n}$ denote the set of $m$-by-$n$-dimensional real matrices, and $\mathbb{N}:=\{0,1,2, \ldots\}$ denote the set of natural numbers. Unless specified, all norms $\| \cdot \|$ are $2$-norms. 
Let $f_k := f (x_k)$, $g_k := \nabla f (x_k)$, and $H_k := \nabla^2 f(x_k)$ denote the value, gradient and Hessian of the objective function, respectively, evaluated at $x_k$. Let $c_k := c(x_k)$, and 
$J_k: = \nabla c(x_k)^T$ denote the Jacobian matrix, evaluated at $x_k$.  
Let $y \in \mathbb{R}^m$ denote the dual variables, and $w = [x^T \;\ y^T ]^T \in \mathbb{R}^{n+m}$ 
a pair of primal-dual variables. 
Let $\mathcal{L}(x,y):= f(x) + y^T c(x)$ denote the Lagrangian function and $W(x,y) \in \mathbb{R}^{n \times n}$ denote the Hessian of the Lagrangian function, where $W_k := W(x_k,y_k) :=  H(x_k) + \sum_{i=1}^m  y_{k, i} \nabla^2 c_i(x_k)$.  
Throughout the paper we use overbars to denote approximations. 
We use the sets $S_k^f, S_k^g, S_k^H \subseteq  \{\omega_1,\omega_2,\dots\}$ to denote the samples 
used to compute function, gradient, and Hessian estimates, respectively,
\begin{align}
 \bar f_k =  \tfrac{1}{\left|S_{k}^f\right|} \sum_{i \in S_{k}^f} f_{i}\left(x_{k}\right),  \quad 
\bar g_k = \tfrac{1}{\left|S_{k}^g\right|} \sum_{i \in S_{k}^g} \nabla f_{i}\left(x_{k}\right),  \quad  \bar H_k = \tfrac{1}{\left|S_{k}^H\right|} \sum_{i \in S_{k}^H} \nabla^{2} f_{i}\left(x_{k}\right), \label{eq.Hessian.estimate}
\end{align}
where $f_i(x_k) = F(x_k, \omega_i)$, $\nabla f_i(x_k) = \nabla F(x_k, \omega_i)$, and $\nabla^2f_i(x_k) = \nabla^2F(x_k, \omega_i)$. These samples $(\omega_i)$ are drawn at random from the distribution $P$ or uniformly from the set $\{\omega_1,\omega_2,\dots , \omega_N\}$ in the context of problems \eqref{problem.expected_risk} or \eqref{problem.empirical_risk}, respectively.


\section{Analysis of Deterministic SQP Methods}
\label{sec.gap}

In this section, we present the line search SQP methodology and its associated global and local convergence guarantees. 
The goal is to highlight an existing 
gap in the analysis of line search (and other) SQP methods for solving deterministic equality constrained problems of the form~\eqref{problem.deterministic} and to discuss several remedies. By gap in this setting we refer to the need for proximity and unit step size assumptions or modifications to classical and practical methodologies in order to prove strong global and fast local convergence guarantees. We make this explicit below. 

We begin by introducing the foundation of the classical 
SQP methodology\footnote{By \emph{classical SQP methodology} we refer to line search SQP methods; e.g., \cite{powell2006fast}, compute a step by solving~\eqref{problem.deterministic}, that utilize the $\ell_1$-merit function, update the merit parameter via dual variables estimates (or other mechanisms), and that impose a sufficient decrease condition. See e.g., \cite[Chapter 18, Algorithm 18.3]{NoceWrig06}} for equality constrained optimization~\cite[Chapter 18]{NoceWrig06}. 
SQP methods are iterative and solve a sequence of subproblems at each iteration in order to compute a step and update the iterate. Specifically, at each iteration $k \in \mathbb{N}$, 
\begin{align}
    x_{k+1} \gets x_k + \alpha_k d_k,
\label{eq.x_update_global}
\end{align}
where $\alpha_k \in \mathbb{R}_{>0}$ is the step size and $d_k \in \mathbb{R}^n$ is the search direction defined as the (possibly inexact) solution of the following SQP suproblem, 
\begin{align}
    \min _{d \in \mathbb{R}^n} \;f(x_k)+  \nabla f(x_k)^T d+\tfrac{1}{2} d^T W_k d \quad \text { s.t. }   c(x_k)+\nabla c(x_k)^{T} d=0, 
\label{eq.SQP_determistic}
\end{align}
where $W_k \in \mathbb{R}^{n \times n}$ is a matrix that captures curvature information of the objective and constraint functions. To the best of our knowledge, the first SQP method was proposed in~\cite{wilson1963simplicial}, with $W_k = \nabla^2 \mathcal{L}(x_k,y_k)$, where $\mathcal{L}(x_k,y_k) = f(x_k) + y_k^Tc(x_k)$ is the Lagrangian function and $y_k \in \mathbb{R}^m$ are the Lagrange multipliers. This method, as well as variants for which $W_k$ is some economical and sufficiently accurate approximation of the Hessian of the Lagrangian, e.g., \cite{han1976superlinearly,powell1978convergence,palomares1976superlinearly}, are endowed with 
local superlinear convergence guarantees under the assumptions that the unit step size is employed, i.e., $\alpha_k =1$ in \eqref{eq.x_update_global}, and that the starting point $(x_0,y_0)$ is in the neighborhood of a first-order stationary point $(x^*,y^*)$, i.e., points for which 
\begin{align}\label{eq.first_order_stat}
    \nabla_{x}  \mathcal{L}(x^*,y^*) = 0 \quad \text{and} \quad c(x^*) = 0. 
\end{align} 
Unfortunately, the neighborhood is defined by problem-specific unknown parameters, and employing a unit step size far from the solution may lead to undesirable behavior. 

To overcome the two limitations and ensure global convergence, i.e., convergence to a first-order stationary point from any initial starting point, two main algorithmic paradigms have prevailed, line search \cite{han1977globally,powell2006fast} and trust region \cite{vardi1985trust,byrd1987trust}. While these approaches ensure convergence from any starting point, local fast convergence can only be proven under the same proximity and unit step size conditions mentioned above.  Consider a detour to the unconstrained setting. The iterate update form in this setting, \eqref{eq.x_update_global}--\eqref{eq.SQP_determistic}, results in Newton steps. If the classical Armijo sufficient decrease line search condition with the initial trial step size set to unity is employed, it is well established that under reasonable assumptions after sufficiently many iterations, the line search condition will be satisfied with the unit step size~\cite{polak1971computational}. This is known as the ``two-phase'' property of  Newton's method. Perhaps surprisingly, this property does not extend to the general nonlinear equality constrained setting, and as such, there exists a gap between the   globalizing phase and the fast local phase for classical SQP methods. That is, for some problems, the unit step size ($\alpha_k = 1$, in \eqref{eq.x_update_global}) does not satisfy the sufficient decrease condition imposed by classical line search SQP methods irrespective of the proximity of the primal-dual iterate to a first-order stationary point. An example of such a problem is given in \cite{maratos1978exact},
\begin{equation}
 \min \; z_1^2 + z_2^2 \quad
    \text{ s.t. }  \left(z_1+1\right)^2+z_2^2-4 = 0,
\label{eq.counterex}
\end{equation} 
for  $(z_1, z_2) \in \{ \left(z_1+1\right)^2+z_2^2= 4, z_1 \in (-1,1) \}$. An 
intuitive explanation for the discrepancy 
is that unlike the unconstrained setting it is not always possible to simultaneously reduce both the constraint violation and the objective function value along a given search direction, regardless of its magnitude.  
As a result, there are problems for which even steps that make good progress towards a solution and keep the constraint violation reasonably small can be rejected by the line search SQP algorithm. This unfavorable phenomenon is known as the Maratos effect \cite[Section 15.5]{NoceWrig06}.  

To mitigate this issue, a sufficient condition is derived in~\cite{maratos1978exact}, related to the curvature in the objective and constraint functions and a measure of progress, that allows for classical line search SQP methods to have provable ``two-phase'' guarantees. 
Rather than a remedy, the condition restricts the general problem class and highlights the gap in the global/local analysis of classical SQP methods. Our modified line search strategy draws inspiration from this condition.
Below, we discuss some algorithmic approaches that mitigate the issue and ensure, after a sufficiently large number of iterations, the unit step length is accepted. 
This, in conjunction with the established global convergence, yields the asymptotic superlinear convergence rate.

\paragraph{\textbf{Modified Merit Function.}} One remedy is to modify the classical line search SQP method that utilizes the $\ell_1$ merit function and instead use the augmented Lagrangian function as the the merit function \cite{schittkowski1982nonlinear}. By modifying the merit function in this way, the constrained problem can be formulated as an unconstrained problem with $n+m$ variables and the analysis for the line search Newton's method, e.g., \cite[Theorem 6.2.31]{polak1971computational}, can be generalized to this setting. Several works have used similar strategies, e.g., Fletcher's exact penalty function \cite{powell1986recursive} and the (augmented) Lagrangian function \cite{chamberlain1982watchdog,gill2017stabilized}. 
Modifying the merit function is a reasonable approach to fill the gap between local and global convergence analysis. However, the issue of these merit functions is their sensitivity to the merit parameter updating scheme. 

\paragraph{\textbf{Modified Subproblem and Update Rule.}} 
One can also modify the constraint approximation in the SQP subproblem \eqref{eq.SQP_determistic} 
as a correction
; see e.g., \cite{coleman1982nonlinear,fukushima1986successive,mayne1982surperlinearly}. Specifically, after computing the solution to \eqref{eq.SQP_determistic}, an auxiliary subproblem is solved for which the affine approximation to the constraints 
is replaced by the 
quadratic approximation, $c_i(x_k)+\nabla c_i(x_k)^{T} d + \frac12 d^T \nabla^2 c_i(x_k) d =0$, for $i=\{1,\dots,m\}$. The algorithm then searches along an arc defined by a combination of the solution of the original and auxilliary subproblems. 
Both line search  \cite{coleman1982nonlinear,fukushima1986successive,mayne1982surperlinearly} and trust region  \cite{gabay1982reduced, fletcher2006second} SQP variants have been developed that utilize this strategy. 
While this reformulation guarantees that after sufficiently many iterations the unit step size is accepted, this does not come for free as the cost of solving the auxiliary subproblem (even approximately) can be higher than that of solving \eqref{eq.SQP_determistic}. 

\paragraph{\textbf{Watchdog and Other Methods.}} The watchdog method \cite{chamberlain1982watchdog}, a nonmonotone method that do not require or enforce sufficient decrease at every iteration and allows for occasional aggressive (unit step size) steps, presents another alternative. A drawback of this method is the possibly expensive (in terms of function evaluations) restart procedure that is required when 
aggressive steps fail. 
Another drawback is the sensitivity to algorithmic parameters such as number of consecutive aggressive steps. 
Other approaches have either used nonmonotone line search, trust-region or filter techniques to address the Maratos effect \cite{ulbrich2004superlinear,gu2011secant,chen2020penalty,chen2022line}.

\vspace{0.1cm}

Our approach and remedy, presented in Section~\ref{sec.ourmethod}, is algorithmic. We maintain the same subproblem, merit function and merit parameter update, and carefully modify the sufficient decrease condition as required, and derive asymptotic local superlinear and global convergence guarantees,  
i.e., we design and analyze a 
line search SQP method for general equality constraints endowed with the ``two-phase'' behavior of Newton's method. In contrast to preceding approaches, 
our 
approach strives to preserve the classical line search SQP methodology to the extent possible and 
only deviates from the classical method when the Maratos effect is detected. In 
such cases, 
a modified line search condition 
is employed to produce more suitable step sizes and guide the iterates away from undesired regions, while maintaining global convergence guarantees and adding fast local convergence. 

We conclude this section with some empirical results on two handpicked examples that showcase the differences between the methods. We compare the classical line search SQP method (\textbf{SQP-L1}) \cite{han1977globally}, the \textbf{2nd-corr}, \textbf{Watchdog}, \textbf{SQP-AugLag} alternatives (discussed above), and our proposed method (\textbf{Our method}). Figure \ref{fig.counterexample} illustrates the performance of the methods on the counter-example \eqref{eq.counterex}, starting from $(z_1, z_2) = (\sqrt{2}-1,\sqrt{2})$, a point that is unfavorable for the \textbf{SQP-L1} method. As expected, \textbf{SQP-L1} produces small step sizes from this starting point, and takes several iterations to reach a point where the unit step size is eventually accepted. All other remedies, including our proposed method, despite traversing different trajectories, accept the unit step size from the second iteration onwards and converge much faster. Figure \ref{fig.ex2} illustrates the performance of the methods on a constrained version of the Rosenbrock problem ($c(z) = (z_1 + 2)^2 +(z_2 - 1)^2 - 9 $), with starting point $(z_1, z_2) = (-1.1,1)$. Notably, the \textbf{SQP-L1} and \textbf{SQP-AugLag} struggle on this problem, and fail to make progress. On the other hand, \textbf{2nd-corr}, \textbf{Watchdog} and \textbf{Our method} accept the unit step size for the majority of the iterations, traverse different trajectories and exhibit fast convergence. The \textbf{2nd-corr} method performs particularly well on this problem, this is not surprising given the constraint is quadratic. Overall, our approach appears to be competitive with contemporary approaches.

\begin{figure}[H]
     \centering
     \begin{subfigure}[b]{1\textwidth}
         \centering
         \includegraphics[width=0.23\textwidth]{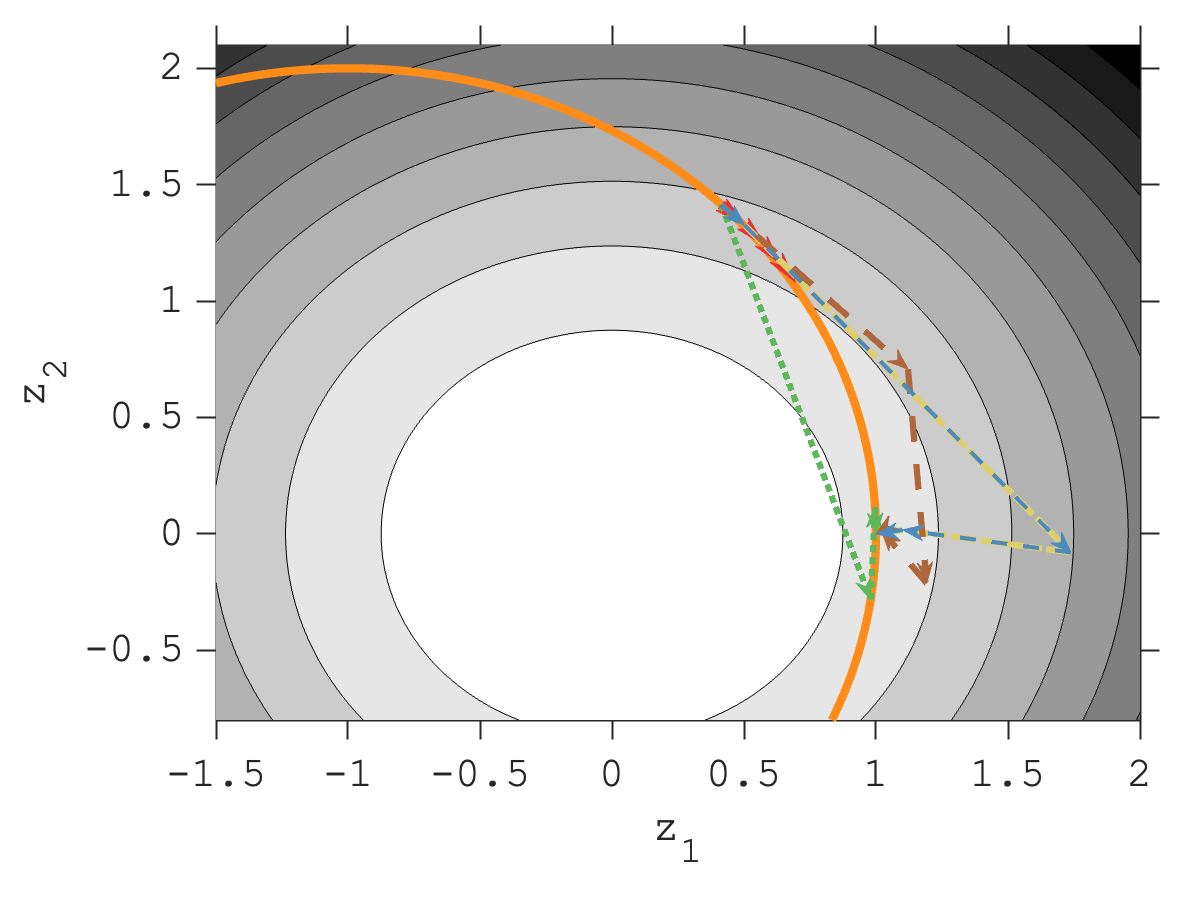}
    \includegraphics[width=0.23\textwidth,clip=true,trim=10 180 50 200]{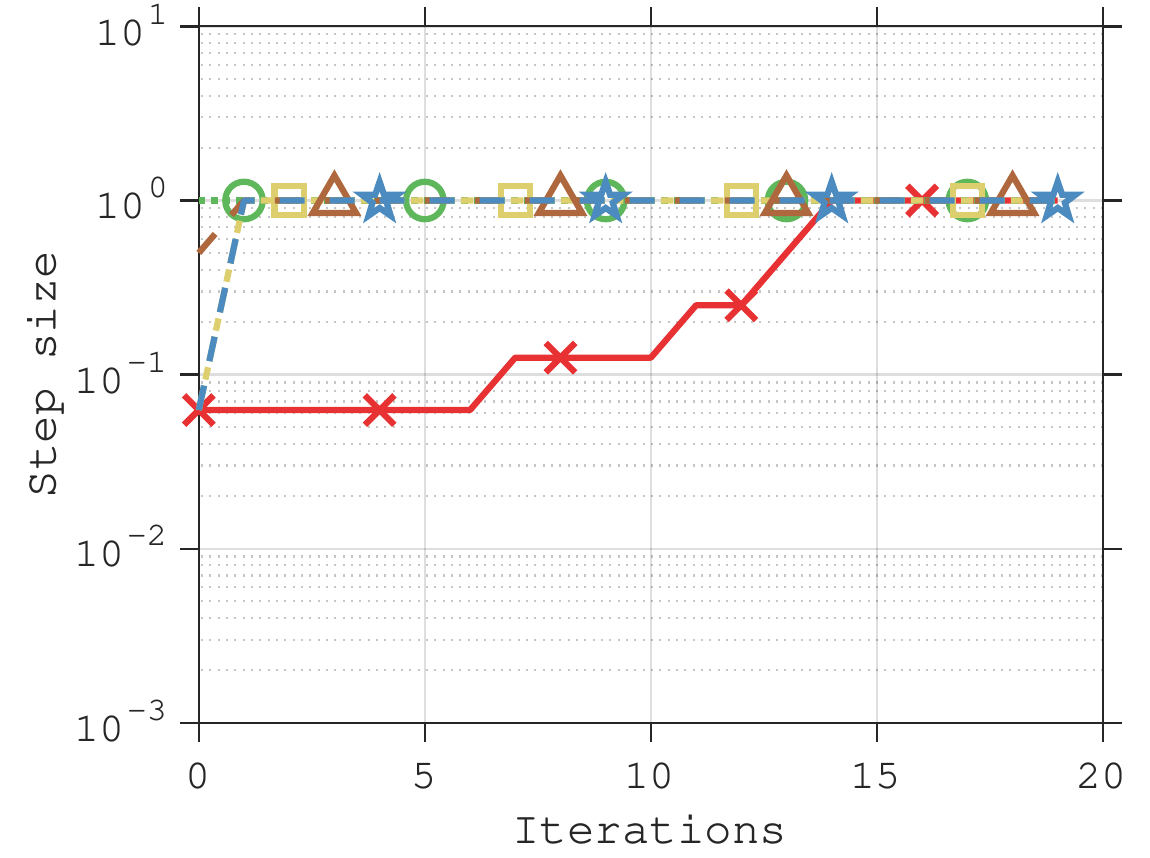}
    \includegraphics[width=0.23\textwidth,clip=true,trim=10 180 50 200]{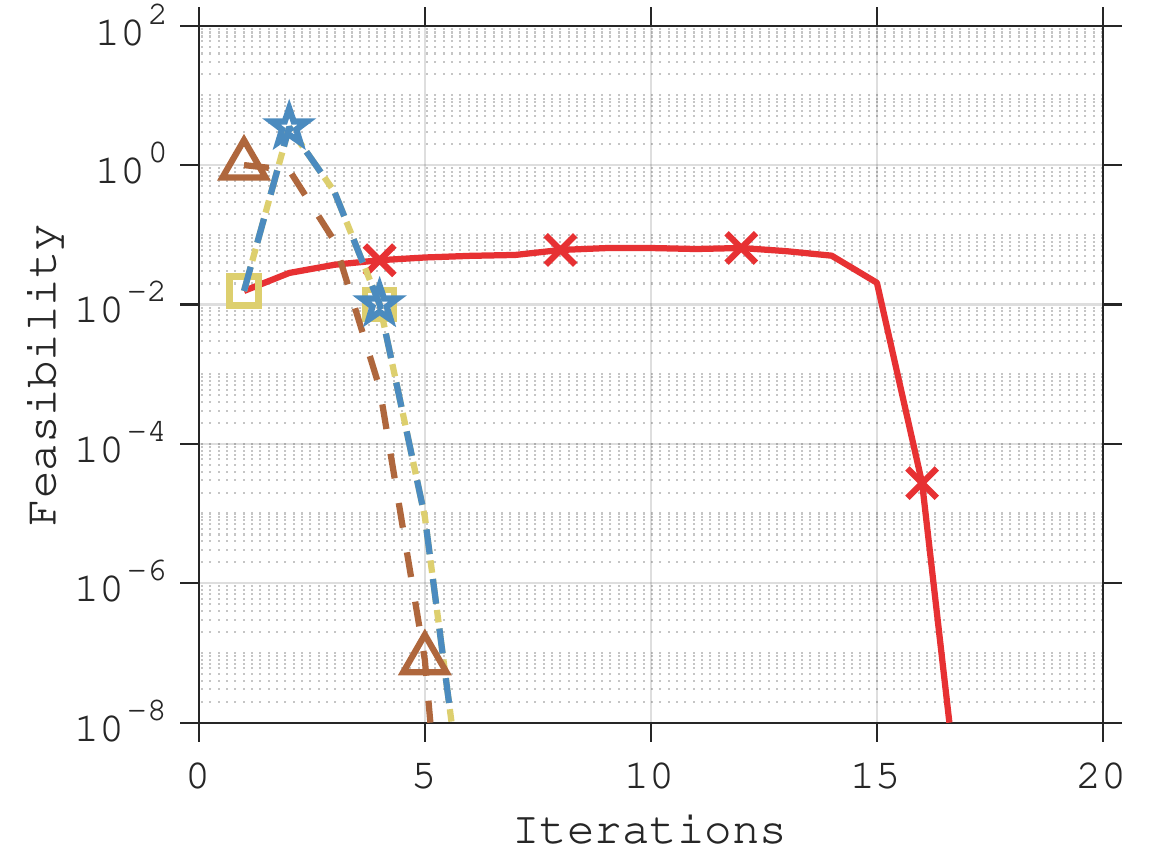} 
    \includegraphics[width=0.23\textwidth,clip=true,trim=10 180 50 200]{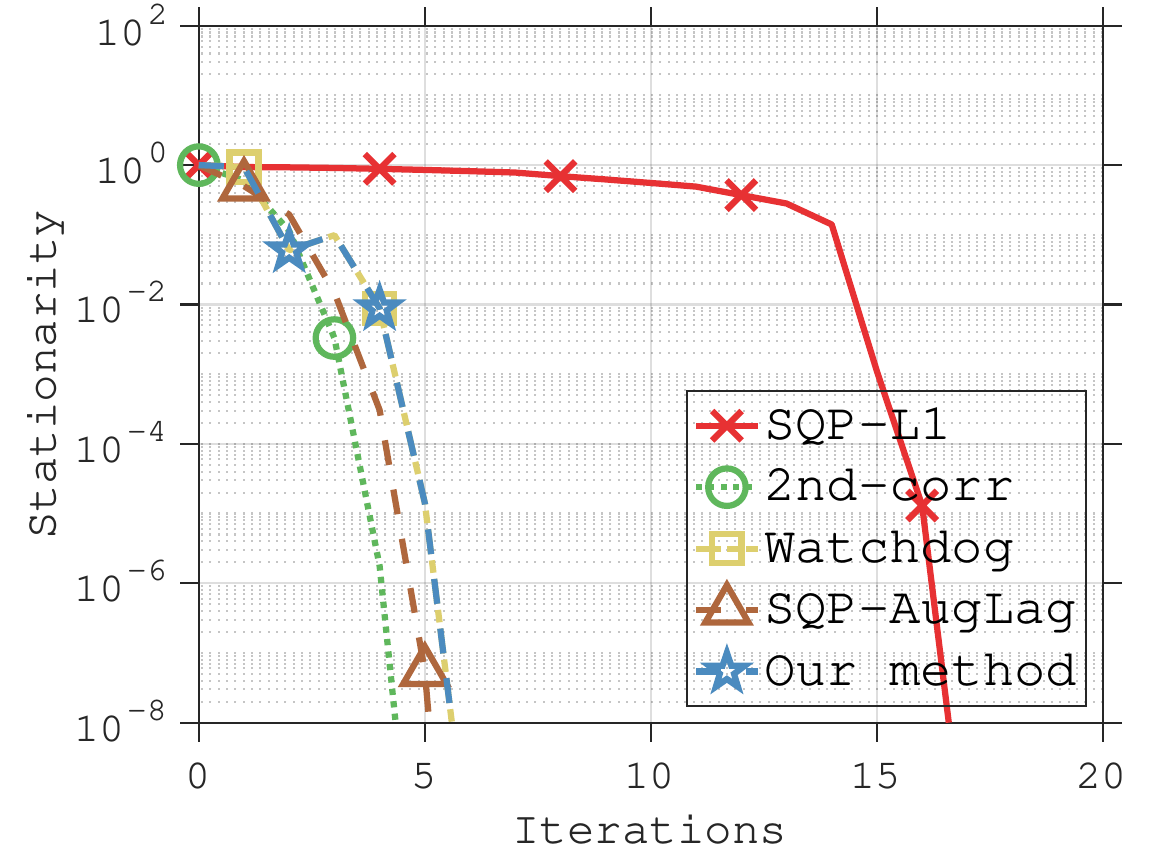}
         \caption{\textbf{Example 1:} Counter-example~\eqref{eq.counterex}, starting point $(\sqrt{2}-1,\sqrt{2})$.}
         \label{fig.counterexample}
     \end{subfigure}
     \hfill
     \begin{subfigure}[b]{1\textwidth}
         \centering
         \includegraphics[width=0.23\textwidth]{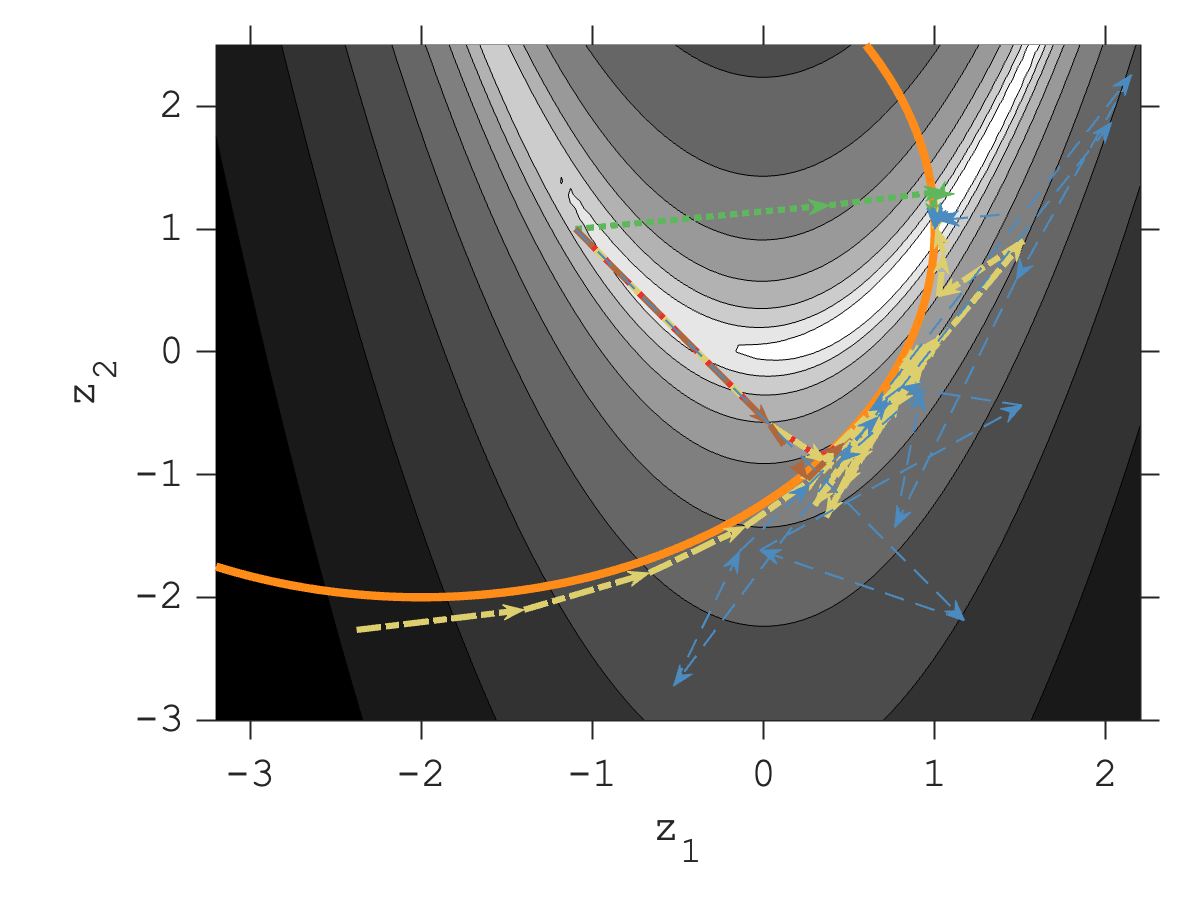}
        \includegraphics[width=0.23\textwidth,clip=true,trim=10 180 50 200]{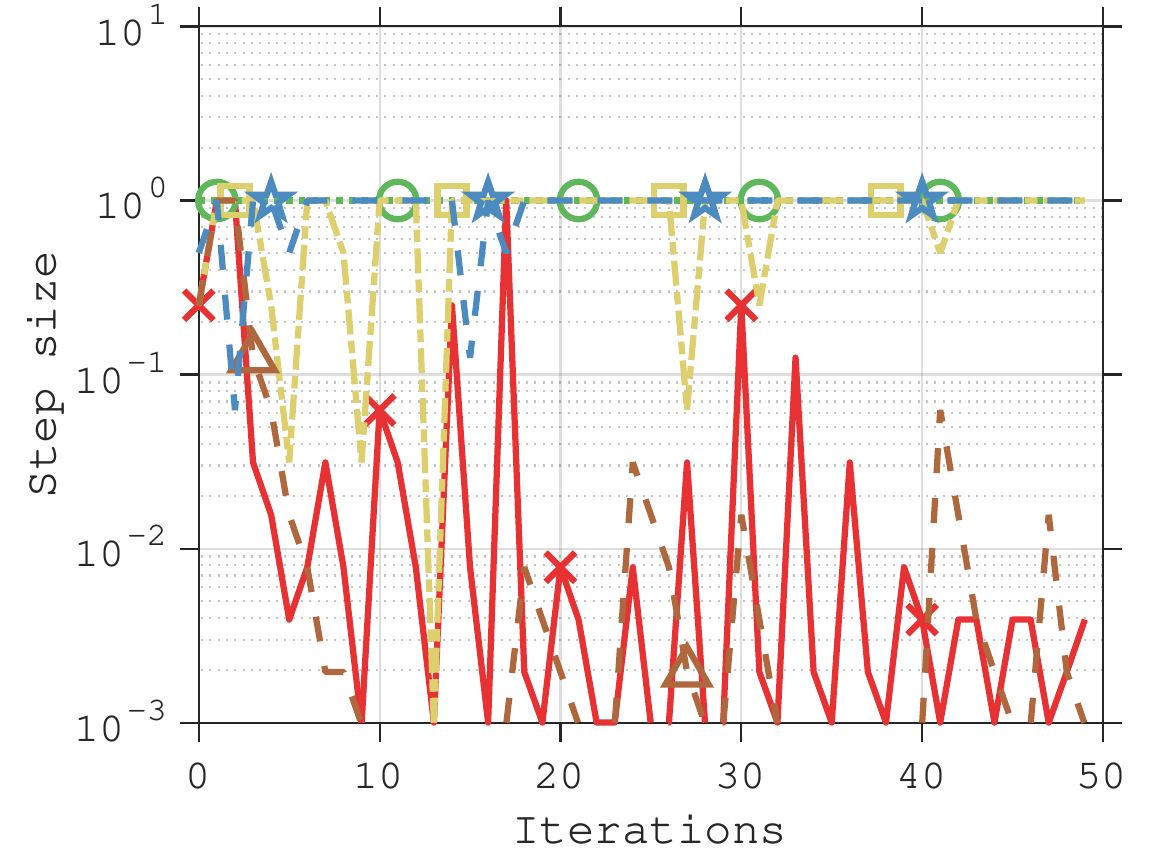}
        \includegraphics[width=0.23\textwidth,clip=true,trim=10 180 50 200]{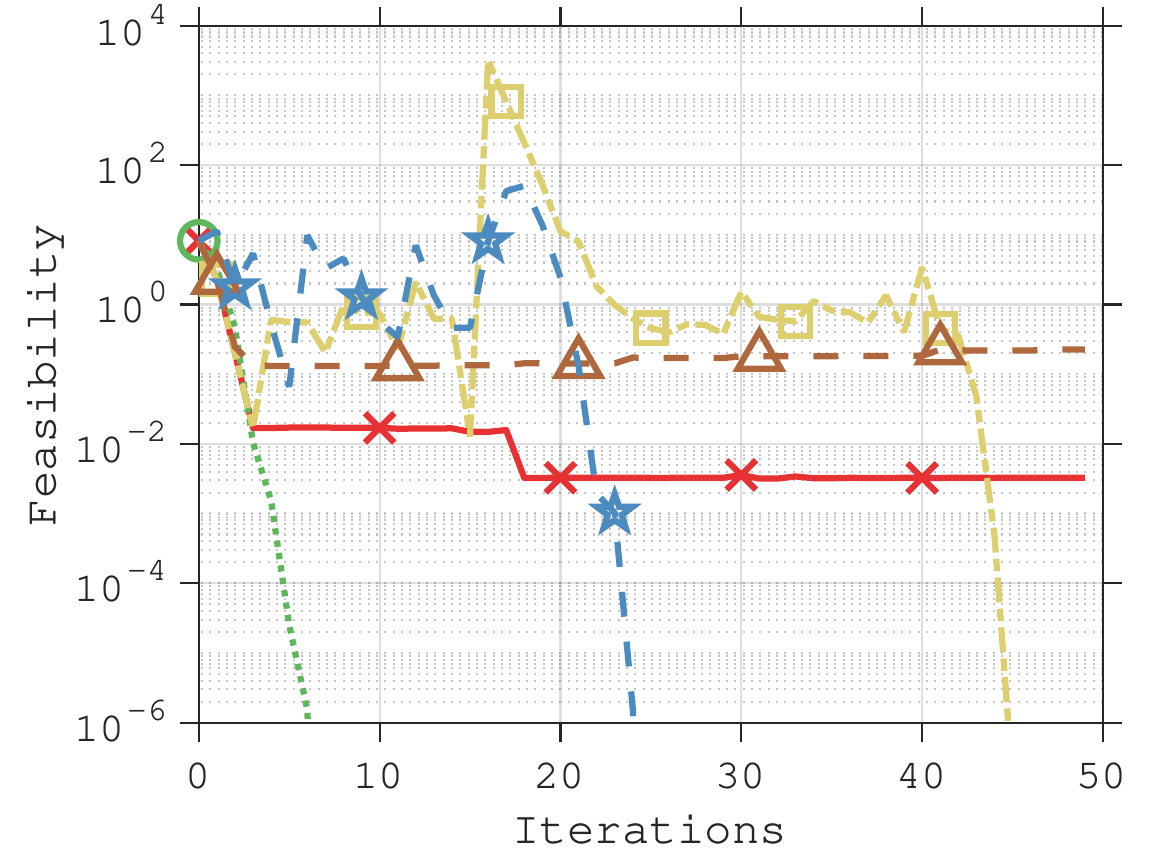}
        \includegraphics[width=0.23\textwidth,clip=true,trim=10 180 50 200]{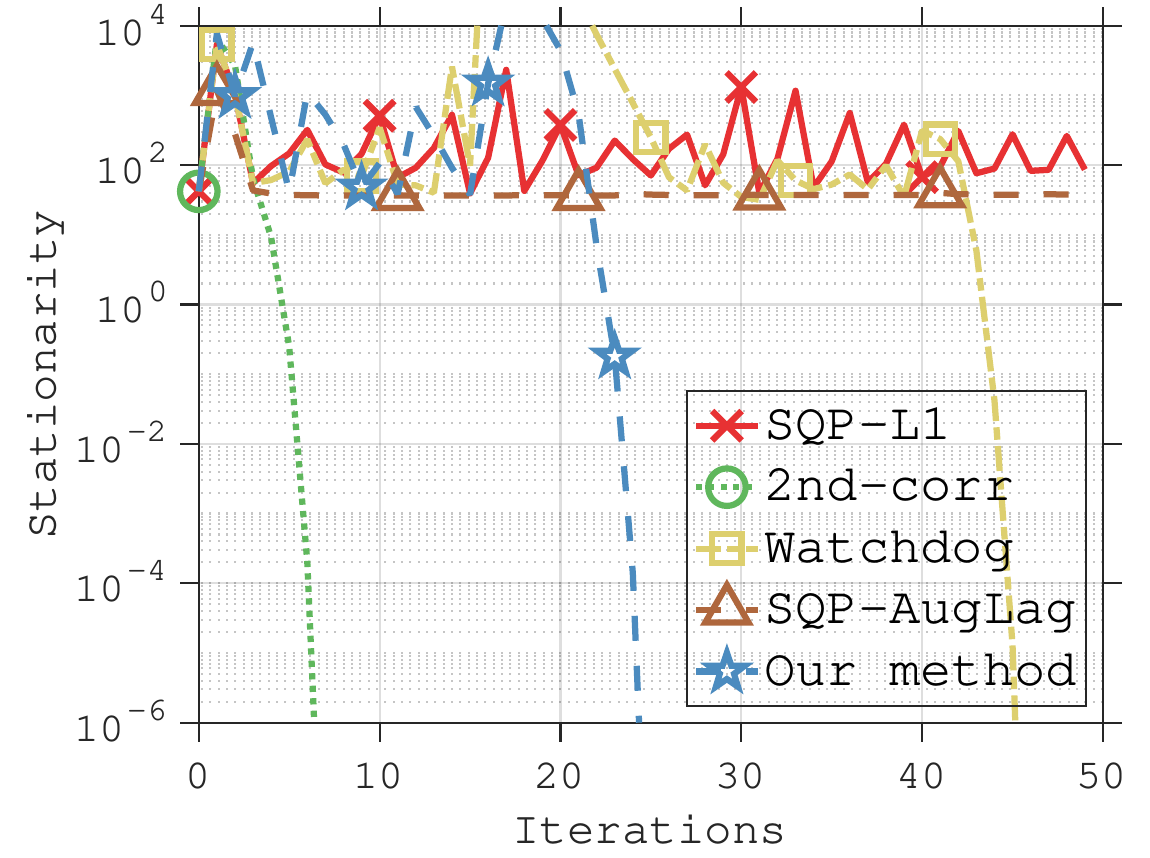}
         \caption{\textbf{Example 2:} $   \min_{z_1, z_2} \;   
(1 - z_1)^2 + 100(z_2 - z_1^2)^2 \text{ s.t. } (z_1 + 2)^2 + (z_2-1)^2 = 9 $, starting point starting point $(-1.1,1)$.  }
         \label{fig.ex2}
     \end{subfigure}
     \vspace{-0.8cm}
        \caption{Comparison of classical line search SQP method (\textbf{SQP-L1}), remedies to address Maratos effect (\textbf{2nd-corr}, \textbf{SQP-AugLag}, \textbf{Watchdog}) and our proposed SQP method (\textbf{Our method}). Trajectory plots (contours in grayscale, constraints in orange), step size, feasibility and stationarity. (Trajectories outside 
        region 
        excluded.)}
        \label{fig.examples}
\end{figure}

\vspace{-0.75cm}

\section{General Assumptions and Preliminaries}\label{sec.assum_notation}

In this section, we first present the general assumptions that are applicable to all the problems considered in this paper and then state and discuss some problem-specific assumptions. 
Additional technical assumptions are stated, as required, in Sections~\ref{sec.superlinear} and \ref{sec.analysis}.
Following each set of assumptions, we introduce  lemmas 
that are used in Sections~\ref{sec.superlinear} and \ref{sec.analysis}. 

We make the following assumption 
about problems~\eqref{problem.deterministic},~\eqref{problem.expected_risk}, and~\eqref{problem.empirical_risk}.

\begin{assumption}
  Let $\mathcal{X} \subseteq \mathbb{R}^{n}$ be an open convex set containing the iterates $\{x_k\}$ and first trial points 
   $\{x_k + d_k\}$. 
   The objective function $f: \mathbb{R}^{n} \to \mathbb{R}$ is twice continuously differentiable and bounded below over $\mathcal{X}$ with $|f(x)| \le \kappa_f$, the gradients $g: \mathbb{R}^{n} \to \mathbb{R}^{n}$ are Lipschitz continuous with constant $L_1$ and bounded over $\mathcal{X}$ with $\|g(x)\|_2 \le \kappa_g$, and the Hessians $H: \mathbb{R}^{n} \to \mathbb{R}^{n\times n}$ are Lipschitz continuous with constant $L_2$ and bounded over $\mathcal{X}$ with $\|H(x)\|_2 \le \kappa_H$.
  The constraint function $c : \mathbb{R}^{n} \to \mathbb{R}^{m}$ (where $m \leq n$) is twice continuously differentiable and bounded over $\mathcal{X}$ with  $\|c(x)\|_2 \le \kappa_c$ and the associated 
  Jacobian matrix $\nabla c^T : \mathbb{R}^{n} \to \mathbb{R}^{m \times n}$ is Lipschitz continuous with constant $\Gamma_1$ and bounded over $\mathcal{X}$. The Hessian of $c(\cdot)$ is Lipschitz continuous with constant $\Gamma_2$. 
Finally, we assume the singular values of $\nabla c^T$ are bounded away from zero over $\mathcal{X}$ such that $\| (J_k J_k^T)^{-1} J_k  \|_2 \le \kappa_{J^{\dag}}$. 
\label{ass.function}
\end{assumption}

We make a few remarks about Assumption~\ref{ass.function}. 
Most of the statements in Assumption \ref{ass.function} constitute standard smoothness assumptions for deterministic equality constrained nonlinear optimization problems of the form~\eqref{problem.deterministic}; see e.g., \cite{byrd1987trust, wachter2005line}. With regards to \eqref{problem.expected_risk} and \eqref{problem.empirical_risk}, smoothness assumptions similar to those in Assumption~\ref{ass.function} are common in the unconstrained stochastic setting; see e.g., \cite{bottou2018optimization,sra2012optimization,shi2021sqp}. When considering problems 
\eqref{problem.expected_risk} and \eqref{problem.empirical_risk}, the iterates generated by the proposed algorithms are stochastic, and admittedly, it is not ideal to assume that the gradient and Hessian of the objective function are bounded over the set of iterates and trial iterates. That said, we believe this assumption is reasonable due to the fact that any well designed algorithm should be driven towards the deterministic feasible region. Moreover, we note that similar assumptions are common in the constrained stochastic optimization literature; see e.g., \cite[Assumption 3.2]{berahas2021sequential}, \cite[Assumption 1]{na2023adaptive}. Finally, the assumption on the boundedness of the singular values of the Jacobian of the constraints is satisfied by the common linear independence constraint 
qualification (LICQ) condition \cite{NoceWrig06}.

The next assumption pertains to the Lipschitz continuity of the gradient and Hessian of the Lagrangian function near a first-order primal-dual stationary point.

\begin{assumption}
For some neighborhood of $w^* = ({x^*}^T, {y^*}^T )^T$ with radius $r\in \mathbb{R}_{>0}$,  where $x^*, y^*$ is a  first-order primal-dual stationary point, i.e., points satisfying \eqref{eq.first_order_stat}, 
the gradient and Hessian of the  Lagrangian function, 
$\mathcal{L}(x,y) = f(x) + y^T c(x) $, are  Lipschitz continuous with constants $L_{\nabla \mathcal{L}}$ and $L_W$, respectively.
 \label{ass.M.Lipschitz.det1}
\end{assumption}

Assumption \ref{ass.M.Lipschitz.det1} is a consequence of Assumption \ref{ass.function}, and can be regarded as an analogue to the Lipschitz continuity of the Hessian of the objective function assumption in the context of unconstrained optimization and Newton's method \cite[Theorem 3.5]{NoceWrig06}. Assumptions  \ref{ass.function}  and \ref{ass.M.Lipschitz.det1} are required for all problems \eqref{problem.deterministic},  \eqref{problem.expected_risk} and \eqref{problem.empirical_risk}.

We require an additional assumption common in the stochastic optimization literature in order to establish theoretical results for the fully stochastic regime and problem \eqref{problem.expected_risk}; see e.g., \cite{bottou2018optimization,sra2012optimization}. Specifically, unbiased gradient estimates with bounded variance, and Hessians estimates with bounded variance. 
\begin{assumption}\label{ass.g.var.sto}
There exists $\sigma_g, \sigma_H  \in \mathbb{R}_{>0}$ such that for all $x \in \mathcal{X}$ and $i = 1,2,\dots$, $\mathbb{E}[\nabla f_i(x)] = \nabla f(x)$, $\operatorname{Var} (\nabla f_i(x)) \le \sigma_g^2$ 
and $\|\mathbb{E}[(\nabla^2 f_i(x) - \nabla^2 f(x))^2 ]\|_2  \le \sigma_H^2$.
\end{assumption}

Assumption~\ref{ass.g.var.sto} pertains only to the derivatives of the objective function of~\eqref{problem.expected_risk} and not the function values. To this end, we make the following disclaimer. In Section~\ref{sec.superlinear}, we present local convergence results (under proximity and unit step size assumptions) for all problem settings, however, our main theoretical results (Section~\ref{sec.analysis}) are only for 
problems~\eqref{problem.deterministic} and \eqref{problem.empirical_risk}. That said, at the expense of additional complexity in the algorithm and analysis, one could extend the results to the fully stochastic setting \eqref{problem.expected_risk} by utilizing approaches for adaptive step sizes developed for the unconstrained stochastic setting \cite{cartis2018global,berahas2021global,paquette2020stochastic}. This is not in the scope of this paper.

In the context of \eqref{problem.empirical_risk}, we replace Assumption~\ref{ass.g.var.sto} with the following assumption.
\begin{assumption}
\label{ass.individual.f.subsampled}
There exists $\kappa
_0, \mu_{0},\kappa_1, \mu_{1}, \kappa_2, \mu_{2} \in \mathbb{R}_{>0}$ such that
for all $x\in \mathcal{X}$ and $i = 1,\dots,N$, $| f_i(x) |  \le  \kappa_0   + \mu_{0} | f(x) |$, $\|\nabla f_i(x) \|_2  \le  \kappa_1   + \mu_{1} \| \nabla f(x) \|_2$, and $\|\nabla^2 f_i(x) \|_2  \le  \kappa_2    + \mu_{2}  \| \nabla^2 f(x) \|_2$.
\end{assumption}

The next lemma characterizes the errors in the approximations. 

{\allowdisplaybreaks
\begin{lemma}
\label{lemma.diff.g.g.bar}
    Suppose  Assumptions~\ref{ass.function} and \ref{ass.individual.f.subsampled} hold. Then, for $k \in \mathbb{N}$, 
\begin{align} 
    |f_k - \bar f_k|  &\le  2 \left(\tfrac{N-\left|S_{k}^f\right|}{N}\right) (\kappa_0 + \mu_{0} \kappa_f):= \epsilon_{f_k}, \label{eq.diff.f.f.bar} \\
    \left\|g_k -\bar g_k\right\|_2 &\le  2 \left(\tfrac{N-\left|S_{k}^g\right|}{N}\right)(\kappa_1 + \mu_{1} \kappa_g):= \epsilon_{g_k}, \label{eq.diff.g.g.bar}  \\   
    \left\|H_k -\bar H_k\right\|_2  &\leq 2 \left(\tfrac{N-\left|S_{k}^H\right|}{N}\right)(\kappa_2 + \mu_{2} \kappa_H):= \epsilon_{H_k}.  \label{eq.diff.H.H.bar}  
\end{align}
\end{lemma}
\begin{proof} 
For brevity, we defer the proof to Appendix~\ref{app.lemma_2}. 
\end{proof}}

Lemma~\ref{lemma.diff.g.g.bar} shows 
an inverse relationship between the approximation errors and the sample sizes. 
In this setting, the sample sizes are capped at $N$, and when $\left|S_{k}^f\right| =\left|S_{k}^g\right| = \left|S_{k}^H\right| = N$, it follows that 
$\epsilon_{f_k} = \epsilon_{g_k} = \epsilon_{H_k} =0$.


\section{A Modified Line Search SQP Method}\label{sec.ourmethod}

In this section, we present in detail our proposed second-order modified line search SQP method. The goal is to develop a practical algorithm that exhibits the ``two-phase'' behavior of Newton's method. That is, global convergence to first-order critical points from any starting point, and fast superlinear local convergence after sufficiently many iterations without imposing proximity or unit step size conditions. For simplicity, we present a deterministic variant of the method for solving \eqref{problem.deterministic}, and discuss the necessary algorithmic changes required in order to tackle problems \eqref{problem.expected_risk} and \eqref{problem.empirical_risk}.

Our proposed modified line search SQP method is iterative \eqref{eq.x_update_global}. The search direction is computed by solving an SQP subproblem~\eqref{eq.SQP_determistic}, and then a step size is selected guided by a merit function and a sufficient decrease condition. At every iteration $k \in \mathbb{N}$, a primal-dual search direction $(d_k,\delta_k) \in \mathbb{R}^n \times \mathbb{R}^m$ is computed by solving (possibly inexactly) the subproblem given in~\eqref{eq.SQP_determistic}. Under Assumption~\ref{ass.function}, this is equivalent to solving the following linear system of equations 
\begin{equation}
M_k
\begin{bmatrix}
       {d}_k \\   \delta_k
\end{bmatrix} =   
\begin{bmatrix}
   W_k   & J_k^T \\ J_k & 0
\end{bmatrix}
\begin{bmatrix}
      {d}_k \\  \delta_k
\end{bmatrix} = - \begin{bmatrix}
     {g}_k + J_k^T  y_k \\ c_k
\end{bmatrix},
\label{eq.SQP.det}
\end{equation}
where $W_k \in \mathbb{R}^{n \times n}$ is defined as the Hessian of the Lagrangian. We make the following standard assumptions about the matrices $M_k$ and $W_k$ that appear in~\eqref{eq.SQP.det}. 
\begin{assumption}
    The sequence $\{ W_k\}$ is bounded in norm by $\kappa_W \in \mathbb{R}_{>0}$ over $x \in \mathcal{X}$.  In addition, there exists a constant $\zeta \in \mathbb{R}_{>0}$ such that, for all $k \in \mathbb{N}$, 
     $u^T  W_k u \geq \zeta \|u\|_2^2$ for all $u \in \mathbb{R}^{n}$ such that $J_k u = 0$. 
\label{ass.H}
\end{assumption}
\begin{assumption}
There are $\Lambda, \kappa_M \in \mathbb{R}_{>0}$ such that
$\| M_k^{-1} \|\le \Lambda$, $\| M_k \|\le \kappa_M$. 
\label{ass.M.inverse.bound}
\end{assumption}
Under Assumptions~\ref{ass.function} and \ref{ass.H}, the linear system \eqref{eq.SQP.det} has a unique solution. Moreover, in the deterministic setting and under these assumptions, if $d_k = 0$, then $(x_k,y_k)$ satisfy \eqref{eq.first_order_stat}. Assumption~\ref{ass.H} is common in the equality constrained optimization literature \cite{palomares1976superlinearly,powell1978convergence}. 
The existence of $\Lambda \in \mathbb{R}_{>0}$ in Assumption \ref{ass.M.inverse.bound} is guaranteed by Assumptions \ref{ass.function} (singular values of  Jacobian bounded from zero) and \ref{ass.H}; see  \cite[Lemma 16.1]{wright1999numerical}. The existence of $\kappa_M \in \mathbb{R}_{>0}$ is guaranteed since both $ W_k$ and $J_k$ are bounded. 

Given a search direction $d_k \in \mathbb{R}^n$, the method proceeds to select a positive step size 
to update the iterate \eqref{eq.x_update_global}. This selection is guided by a merit function. As is common with popular line search SQP methods, we consider the $\ell_1$-merit function
\begin{align}
\phi(x, \tau) :=\tau f(x)+\|c(x)\|_1, 
    \label{eq.def.phi}
\end{align}
where $\phi: \mathbb{R}^n \times \mathbb{R}_{>0} \rightarrow \mathbb{R}$ and $\tau \in \mathbb{R}_{>0}$ is the merit parameter that is chosen adapatively as the optimization progresses and whose goal is to balance the two possibly competing goals of minimizing the objective function and satisfying the constraints. We make use of a local model of the merit function $l : \mathbb{R}^{n} \times \mathbb{R}_{>0} \times \mathbb{R}^{n} \times \mathbb{R}^{n} \to \mathbb{R}$, defined as 
\begin{equation}
  l(x,\tau,\nabla f(x),d) := \tau (f(x) + \nabla f(x)^T d) + \|c(x) + \nabla c(x)^Td\|_1.
\label{def.localmodel}
\end{equation}
In particular, a critical quantity in our algorithm is the reduction (across a step for which $c(x) + \nabla c(x)^Td = 0$) in the local model of the merit function, $\Delta l : \mathbb{R}^{n} \times \mathbb{R}_{>0} \times \mathbb{R}^{n} \times \mathbb{R}^{n} \to \mathbb{R}$, defined as
\begin{equation}
\begin{aligned}
        \Delta l(x,\tau,\nabla f(x),d) := l(x,\tau,\nabla f(x),0) - l(x,\tau,\nabla f(x),d) = -\tau \nabla f(x)^Td + \|c(x)\|_1.
\label{def.merit_model_reduction}
\end{aligned}
\end{equation}

After computing a step and before updating the iterate (taking the step) the merit parameter $\tau_k$ is updated. 
We follow standard approaches \cite{berahas2021sequential,byrd2008inexact}. For some $\sigma, \epsilon_{\tau} \in (0,1)$, the merit parameter is updated as follows,
\begin{equation}
\label{eq.merittrial.1a}
  \tau_k^{trial} \gets \begin{cases} \infty & \text{if } g_k^Td_k + \max\{d_k^T  W_k d_k,0\} \leq 0 \\ \tfrac{(1 - \sigma)\| c_k\|_1}{g_k^Td_k + \max\{d_k^T  W_k  {d}_k,0\}} & \text{otherwise},
\end{cases}
\end{equation}
followed by
\begin{equation}
    \tau_k \gets 
    \begin{cases} \tau_{k-1} & \text{if $\bar\tau_{k-1} \leq  \tau_k^{trial}$} \\ (1-\epsilon_{\tau}) \tau_k^{trial} & \text{otherwise}.
\end{cases}
\label{eq.meritupdate.1a}
\end{equation} 
In both cases, $\tau_k \le  \tau_k^{trial}$ and 
$\Delta l\left(x_k, \tau_k, {g}_k, d_k\right) \geq \tau_k \max \left\{d_k^T  W_k  d_k, 0\right\}+\sigma\left\|c_k\right\|_1$. This  
inequality is encouraging  as it guarantees $\Delta l\left(x_k, \tau_k, {g}_k, d_k\right)$ is always nonnegative, and is zero only when $\|c_k\|_1 = \|d_k\|  = 0$, points that satisfy first-order stationarity \eqref{eq.first_order_stat}.

Our proposed algorithm computes a step size by ensuring sufficient decrease on the merit function akin to that used in classical line search SQP algorithms, e.g.,  \cite{han1977globally}, 
\begin{align}
     \phi(x_{k} + \alpha_k d_k,\tau_{k}) \le  \phi(x_{k},\tau_{k}) - \eta \alpha_k \Delta l(x_k,\tau_k,g_k,d_k), \label{eq.line_search_cond}
\end{align} 
where $\eta \in (0,1)$, 
until a certain condition is satisfied. The distinctive feature of our 
approach is a carefully crafted modified line search 
condition that is employed as required over the course of the optimization. 
The intuition of the condition and modification is to detect and correct for the Maratos effect.

As revealed by classical SQP analysis, under the proximity (iterates to the optimal solution) and unit step size assumptions, the method enjoys a superlinear rate of convergence; see e.g., \cite{han1976superlinearly,powell1978convergence,palomares1976superlinearly}. In Section \ref{sec.superlinear}, we prove a slightly different result and show that there exists a threshold $\gamma_d$ such that if we run the classical line search SQP method until $\|d_j\| \le \gamma_d$ for some iteration $j \in \mathbb{N}$ and simply set $ \alpha_k = 1$ for $k\ge j$, superlinear convergence is retained. However, $\gamma_d$ depends on unknown parameters (see Lemma \ref{theorem.superlinear.convergence.deterministic}), and it is therefore impractical to implement such rule. We propose two modifications to this naive and impractical approach. First, we consider an adaptive threshold parameter ($\gamma_k$ instead of $\gamma_d$), that is 
updated (reduced), as necessary, over the course of the optimization. Second, since it is not known if $\gamma_k$ is sufficiently small, and employing a unit step size for all iterates for which $\|d_k\| \le \gamma_k$ may be too aggressive, we propose a carefully crafted sufficient decrease condition that incorporates second-order information of the objective and constraint functions, i.e.,
\begin{equation}
\begin{aligned} \label{eq.line_search_cond_modified}
    \phi(x_{k} + \alpha_k d_k,\tau_{k}) &\le  \phi(x_{k},\tau_{k}) - \eta \alpha_k \Delta l(x_k,\tau_k,g_k,d_k)  \\  
    & \quad  +  \frac12 \alpha_{k}^2 \tau_{k}  d_k^T  H_k  d_k + \frac{1}{2} \alpha_k^2  \sum_{i=1}^m \left| d_k^T \nabla^2 c_i(x)  d_k  \right|. 
\end{aligned}
\end{equation}

The final component of the algorithm  is the fact that \eqref{eq.line_search_cond_modified} is satisfied with $\alpha_k = 1$ when $\|d_k\|$ is sufficiently small (Lemma~\ref{lemma.dist.small.imply.alpha.1}). 
As such, our proposed adaptive mechanism reduces $\gamma_k$ when \eqref{eq.line_search_cond_modified} is triggered 
to ensure that the threshold parameter is eventually, and as needed, sufficiently small. We note that this novel step size condition only needs to be to be triggered when, for small search directions, the classical line search fails to accept the unit step size. Overall, our proposed approach only deviates from the classical line search approach when the occurrence of the Maratos effect is detected.  The full algorithm is given in Algorithm~\ref{alg.SubsampledSQP.practical}.

\begin{algorithm}[ht]
  \caption{Modified Line Search SQP}
  \label{alg.SubsampledSQP.practical}
  \begin{algorithmic}[1]
    \Require $x_0 \in \mathbb{R}^{n}$; $y_0 \in \mathbb{R}^{m}$; $\tau_{-1} \in \mathbb{R}_{>0}$; $\gamma_0   \in \mathbb{R}_{>0}$; $ \nu_{\alpha}, \nu_{\gamma}, \epsilon_{\tau}, \sigma \in (0,1) $. 
    \For{\textbf{all} $k \in \mathbb{N}$}
    \State Compute 
    ${f}_k$, ${g}_k$, ${H}_k$.
    \label{line.obtain.estimate}
	  \State Compute $({d}_k,{\delta}_k)$ as the solution of \eqref{eq.SQP.sto}.
      \State Set $\tau_k^{trial}$ and $\tau_k$ by \eqref{eq.merittrial.1a} and \eqref{eq.meritupdate.1a}. \label{line.merit.parameter}
      	\If{$\| d_k\|> \gamma_k $} \label{line.if5}  
      \State \textbf{until} \eqref{eq.line_search_cond} is achieved, set $\alpha_k = \nu_{\alpha} \alpha_k$. \label{line: d_large.classical}     
      \Else  \label{line.else}
   \If{ \eqref{eq.line_search_cond} is satisfied with $\alpha_k = 1$} set  $\alpha_k = 1$.
    \label{line.alpha.1}
      \Else \label{line.else10}
      \State \textbf{until} \eqref{eq.line_search_cond_modified} is achieved, set $\alpha_k = \nu_{\alpha} \alpha_k$. \label{line: d_small.modified}
    \State Set $\gamma_{k+1}  = \nu_{\gamma} \gamma_k$. \label{line: d_small.sigma.reduce} 
      \EndIf    
      \EndIf \label{line.end}
      \State Set $x_{k+1} \gets x_k + \alpha_k {d}_k$ and $ y_{k+1} \gets y_k +  \delta_k $. 
    \EndFor
  \end{algorithmic}
\end{algorithm}

\begin{remark}
    We make a few remarks about Algorithm~\ref{alg.SubsampledSQP.practical}. 
    \begin{itemize}[leftmargin=0.5cm]
        \item \textbf{Adaptive Condition (Line~\ref{line.if5}) and Behavior of $\gamma_k$:} The purpose of the condition in Line \ref{line.if5} is to ensure that $\|d_k\|$ is sufficiently small before any modifications are employed. The else condition is guaranteed to eventually be satisfied (see Lemma \ref{lemma.sigma.exist}) since $\gamma_k$ remains fixed while $\| d_k\| > \gamma_k$. When $\| d_k\| \le \gamma_k$, the method attempts to detect if the Maratos effect is present by checking whether $\alpha_k  = 1$ satisfies \eqref{eq.line_search_cond}. 
        If $\alpha_k  = 1$ satisfies \eqref{eq.line_search_cond}, the classical line search condition is good enough and the step size is set to one. Otherwise, we either have Marato's effect or we are not yet sufficiently close to the optimal solution ($\gamma_k$ not sufficiently small). In this case, we consider the modified line search condition \eqref{eq.line_search_cond_modified}, and reduce $\gamma_k$.
        The sequence $\gamma_k$, which is non-increasing, is fundamental to our algorithm. This sequence reduces by a factor of $\nu_{\gamma}$ upon each decrement. Importantly, while $\gamma_k$ experiences a decrease at a linear rate at most, we prove that $\|{d}_k\|$ undergoes a decrease at a superlinear rate after a sufficient number of iterations (see Lemma~\ref{lemma.sigma.exist}), leading to the consistent engagement of Line \ref{line.else}. If $\alpha_k = 1$ does not satisfy \eqref{eq.line_search_cond} infinitely often, the algorithm ends up with $\gamma_k \to 0$.
        \item \textbf{Marato's Condition:} In \cite[Proposition 8]{maratos1978exact} it is mentioned that if 
\begin{align}
        \label{eq.Maratos.cond}
        \tau_{k} d_k^T H_k d_k + \sum_{i=1}^m \left| d_k^T \nabla^2 c_i(x) d_k \right| \le 2 \sigma_m \eta  \Delta l\left(x_k, \tau_k, {g}_k, d_k\right)  
\end{align}
        is satisfied for some $\sigma_m \in \mathbb{R}_{>0}$ when $\|d_k\|$ is sufficiently small, then the Maratos effect is not present. Thus, one option is to only modify the algorithm when \eqref{eq.Maratos.cond} is not satisfied. We deviate from this strategy in two ways. First, we directly use an adaptive condition on the search direction for simplicity. Second, we only apply the modification when $ \alpha_k = 1$ does not satisfy condition \eqref{eq.line_search_cond} to reduce the frequency of employment of the modified line search condition. 
        \item \textbf{Comparison to Classical Sufficient Decrease Condition:} Compared to classical sufficient decrease (on the merit function) conditions employed by SQP methods \eqref{eq.line_search_cond}, condition \eqref{eq.line_search_cond_modified} has two additional terms related to second-order information of the objective and constraint functions. As we show (see Lemma \ref{lemma.dist.small.imply.alpha.1}), $ \alpha_k = 1$ satisfies \eqref{eq.line_search_cond_modified} when $\| d_k\|$ is sufficiently small. Hence, we reduce $\gamma_k$ by a factor of $\nu_{\gamma}$ to ensure asymptotic unit step size acceptance.
        \item  \textbf{Comparison to Existing Remedies:} Similar to the Watchdog method \cite{chamberlain1982watchdog}, our algorithm adopts a nonmonotone line search 
        condition when $ \tau_{k} d_k^T H_k d_k + \sum_{i=1}^m \left| d_k^T \nabla^2 c_i(x) d_k \right|$ $\gg 0$. That said, the nonmonotonicity is controlled in our approach, and the step size is selected by condition \eqref{eq.line_search_cond_modified}, and thus is more conservative than the unit step size that is employed by the Watchdog method \cite{chamberlain1982watchdog}. 
        \item \textbf{Comparison to Adaptive Step Size Scheme:} Alternative step size schemes that do not require function evaluations have been proposed for the constrained stochastic setting \cite{berahas2021sequential,berahas2022accelerating}. Such schemes are particularly effective in the stochastic setting, however, we refrain from utilizing such approaches and retain a traditional line search approach for two main reasons. First, for these adaptive schemes, asymptotic unit step sizes can only be guaranteed for well-conditioned problems. Second, the necessity of the Lipschitz constants (or an estimate) in these algorithm poses challenges in the implementation since these constants are usually unknown.
        \item \textbf{Parameters:} Algorithm \ref{alg.SubsampledSQP.practical} has several parameters. Most importantly, $\gamma_0$, the initial value of the threshold parameter. We recommend and set $\gamma_0 \gets \bar \gamma \| d_0\| $ for some $\bar \gamma  \in (0,1)$ to make sure that Line \ref{line.if5} is triggered in the first iteration. The other parameters ($ \tau_{-1}$, $\eta$ and $\nu_{\alpha}$) are universal to all methods considered, with the additional parameter $\nu_{\gamma}$, the reduction factor in the threshold parameter.
    \end{itemize} 
\end{remark}

\subsection{The Stochastic Setting}

The development of adaptive line search SQP methods for the stochastic and finite sum settings is not as straightforward. Several modifications are necessary to ensure convergence (global and fast local) as well as practicality. We focus on the finite sum setting \eqref{problem.empirical_risk} in this section. In this context, we follow the same algorithmic structure as Algorithm~\ref{alg.SubsampledSQP.practical} but replace the exact function (and derivatives) with their stochastic counterparts.


At every iteration $k \in \mathbb{N}$, we first select subsets of all the data  $S_k^f, S_k^g, S_k^H \subseteq  \{\omega_1,\omega_2,\dots , \omega_N\}$ 
(specific rules given in Section~\ref{sec.analysis}), compute the estimates of function value, gradient, and Hessian via \eqref{eq.Hessian.estimate} as  $\bar{f}_k$, $\bar{g}_k$, $\bar{H}_k$ (replace $f_k$, $g_k$, $H_k$ Line \ref{line.obtain.estimate} of Algorithm~\ref{alg.SubsampledSQP.practical}), and set $\bar W_k =\bar{H}_k + \sum_{i=1}^m {y}_{k,i} \nabla^2 c_i(x_k) $. Then, we obtain the  search direction $ (\bar d_k, \bar \delta_k) \in \mathbb{R}^n \times \mathbb{R}^m$ by solving a linear system of the form
\begin{equation}
\bar M_k
\begin{bmatrix}
     \bar  {d}_k \\  \bar \delta_k
\end{bmatrix} =   
\begin{bmatrix}
  \bar W_k   & J_k^T \\ J_k & 0
\end{bmatrix}
\begin{bmatrix}
     \bar {d}_k \\ \bar \delta_k
\end{bmatrix} = - \begin{bmatrix}
    \bar {g}_k + J_k^T  \bar{y}_k \\ c_k
\end{bmatrix}.
\label{eq.SQP.sto}
\end{equation}
To ensure the uniqueness of solution to \eqref{eq.SQP.sto} and other theoretical properties (similar to those discussed above for the deterministic), we make the following assumption.

\begin{assumption}
Assumptions \ref{ass.H} and  \ref{ass.M.inverse.bound} hold (with the same constants) for the matrices $\bar W_k$  and $\bar M_k$.  
\label{ass.M.inverse.bound.subsampled}
\end{assumption}

Further, the update rule for the merit parameter ($\bar \tau_k$) in the finite sum setting is similar to \eqref{eq.merittrial.1a}--\eqref{eq.meritupdate.1a}, with approximations replacing all deterministic quantities. 
The classical and modified line search 
conditions \eqref{eq.line_search_cond} and  \eqref{eq.line_search_cond_modified} are relaxed as follows 
\begin{align}
    \bar \phi(x_{k} + \bar\alpha_k \bar d_k,\bar\tau_{k}) &\le  \bar \phi(x_{k},\bar\tau_{k}) - \eta \bar\alpha_k \Delta l(x_k,\bar\tau_k,\bar g_k, \bar d_k) + \epsilon_{A_k}, \label{eq.line_search_cond_stoch} \\
    \bar\phi(x_{k} + \bar \alpha_k \bar d_k,\bar \tau_{k}) &\le  \bar\phi(x_{k},\bar \tau_{k}) - \eta \bar\alpha_k \Delta l(x_k,\bar\tau_k,\bar g_k,\bar d_k) \notag \\  & \qquad  +  \tfrac12 \bar\alpha_{k}^2 \bar\tau_{k} \bar d_k^T  H_k \bar d_k + \tfrac{1}{2} \bar\alpha_k^2  \sum_{i=1}^m \left|\bar d_k^T \nabla^2 c_i(x) \bar d_k  \right| + \epsilon_{A_k}, \label{eq.line_search_cond_modified_stoch}
\end{align}
respectively, where $\bar \phi(x_{k},\bar\tau_{k}) = \bar \tau_k 
 \bar f_k + \|c_k\|_1$ and $\epsilon_{A_k} \in \mathbb{R}_{\geq 0}$ is a relaxation term added to account for the noise. The relaxation term is of paramount importance to both the theoretical analysis 
 as well as the practical implementation. This term is proportional to the accuracy of the function and gradient approximations.
The explicit expression for $\epsilon_{A_k}$ is given in Section \ref{sec.analysis} (Lemma~\ref{lemma.sigma.exist}). Finally, the iterates are updated via
\begin{equation}
    x_{k+1} = x_{k} + \bar{\alpha}_k \bar d_k, \quad \text{and} \quad y_{k+1} = y_k +\bar  {\delta}_k.
    \label{eq.subsample.unit.stepsize.update}
\end{equation}


\section{Preliminary Local Convergence Guarantees}
\label{sec.superlinear}

Before we present the theoretical convergence guarantees for Algorithm~\ref{alg.SubsampledSQP.practical} (Section~\ref{sec.analysis}), in this section, we set the stage by formalizing the superlinear convergence results for a general SQP method under the proximity and unit step size assumptions. We consider both the deterministic and stochastic settings, 
derive neighborhood conditions, and provide fast local convergence guarantees. In contrast to the conventional neighborhood proximity condition, which assumes $\|w_0 - w^*\|$ is sufficiently small \cite{palomares1976superlinearly,han1976superlinearly,powell1978convergence}, we propose an alternative explicit  condition based on the norm of the search direction $\|d_k\|$ (or $\|\bar d_k\|$). This modification is motivated by the fact that it is computable within an algorithm. 
We show that under reasonable assumptions, when $\|d_k\|$ (or $\|\bar d_k\|$) is below a certain threshold (depends on unknown problem-specific parameters) the conventional neighborhood condition is satisfied. In Section~\ref{sec.analysis}, we analyze the adaptive algorithm presented in Section~\ref{sec.ourmethod} that does not require knowledge of the unknown parameters.

We make the following assumption 
throughout this section. 
\begin{assumption} 
\label{ass.d.bound.imply.difference.det}
There exists $\kappa_d, \mu_M \in \mathbb{R}_{>0}$, such that 
for all $k \in \{j \in \mathbb{N}| \|d_j\| \le \kappa_d\}$, 
it follows that $\tfrac{\mu_M}{2}  \|w_k - w^*\| \le \left\|\begin{bmatrix}
     g_k + J_k^T y_k \\ c_k
\end{bmatrix}\right\|$, where $d_k$ is the solution of 
\eqref{eq.SQP.det} and $w_k = (x_k^T,y_k^T)^T$ is generated by Algorithm \ref{alg.SubsampledSQP.practical}. 
\end{assumption}
\begin{remark}
   When $(x_k,y_k)$ satisfies $\|\nabla \mathcal{L}(x_k,y_k)\|$ $\le \Lambda^{-1} \kappa_d$, it follows that $\|d_k\|$  $\le \Lambda \|\nabla \mathcal{L}(x_k,y_k)\| \le \kappa_d$. Thus, Assumption \ref{ass.d.bound.imply.difference.det} is satisfied if $\mathcal{L}(x,y)$ is  $\tfrac{\mu_M}{2}$-strongly convex with respect to $w$. 
   Consequently, in this scenario, the iterates converge to a second-order stationary point.  We note that this assumption is not needed for all iterates; instead, it is required only 
   after running the algorithm for a sufficiently large number of iterations, such that the iterates enter a locally strongly convex regime. 
 \end{remark}

\subsection{Deterministic Problem~\eqref{problem.deterministic}}

We present local analysis for the SQP method for~\eqref{problem.deterministic} first. 
\begin{theorem}
\label{theorem.quadratic.det}
Suppose Assumptions~\ref{ass.function}, \ref{ass.M.Lipschitz.det1}, \ref{ass.H}, \ref{ass.M.inverse.bound}, and~\ref{ass.d.bound.imply.difference.det}  hold. Let $w_k = [x_k^T, y_k^T]^T$, 
$x_{k+1} = x_{k} + d_k$ and $y_{k+1} = y_k + {\delta}_k$, where $(d_k,\delta_k)$ are computed via~\eqref{eq.SQP.det}. Then, for $k \in \mathbb{N}$, $\| w_{k+1} - w^* \| \le \tfrac{\Lambda L_W}{2}  \| w_{k} - w^* \|^2$. 
If 
the starting point satisfies  $\|d_0\| \le \min\left \{ \tfrac{\mu_M}{2 \kappa_M (1+   2\kappa_{J^{\dagger}} \kappa_W ) \Lambda L_W} 
, \kappa_d \right \}$, 
then, $\|w_{k} - w^* \| \to  0$ at a Q-quadratic rate. 
\end{theorem}
\begin{proof}
   Since a more general result is proven in Theorem \ref{theorem.superlinear.convergence.stochastic}, we defer the proof of this theorem to Appendix~\ref{appendix.quadratic.theorem}. 
\end{proof}

Note that similar local quadratic convergence results for the deterministic SQP method are established in \cite[Theorem 15.2.1]{conn2000trust} and~\cite{NoceWrig06,wilson1963simplicial}. 

\subsection{Stochastic Problem~\eqref{problem.expected_risk}}
\label{sec.superlinear.stochastic}

At each iteration  $k \in \mathbb{N}$,
given $S_k^g, S_k^H \subseteq  \{1,2,\dots\}$, 
we compute $(\bar d_k, \bar \delta_k)$ via \eqref{eq.SQP.sto} and update the iterates via \eqref{eq.subsample.unit.stepsize.update} with $\bar{\alpha}_k = 1$. 
Since the step size is predefined (unit), function estimates are not required. Lemma \ref{lemma.d.bound.imply.difference} ensures that when $\|\bar d_k\|$ is sufficiently small, the neighborhood condition is satisfied, and is important in establishing the superlinear convergence result.




\begin{lemma}
\label{lemma.d.bound.imply.difference}
Suppose Assumptions 
\ref{ass.function}, \ref{ass.M.Lipschitz.det1},  \ref{ass.individual.f.subsampled}, 
\ref{ass.M.inverse.bound.subsampled}, and \ref{ass.d.bound.imply.difference.det} 
hold. 
   For any  $\kappa_w \in \mathbb{R}_{>0}$ and $k \in \mathbb{N}$, suppose that $ \|g_k - \bar g_k\| \le \min \left\{ \frac{\mu_M \kappa_w}{4}, \frac{\kappa_d}{2\Lambda} \right\}$. 
    When  $\|\bar d_k\|\le \min \left\{ \frac{\mu_M\kappa_w }{4\kappa_M (1+   2\kappa_{J^{\dagger}} \kappa_W )}, \frac{\kappa_d}{2} \right\}$,  it follows that $\|w_{k}-w^{*}\|\le \kappa_w$.
\end{lemma}
{\allowdisplaybreaks
\begin{proof}
  We use the orthogonal decomposition of the search direction $\bar  d_k \in \mathbb{R}^n$, i.e., 
\begin{align}
    \label{eq.orthogonal.sto}
    \bar  d_k = \bar u_k + v_k, \quad \bar u_k \in \operatorname{Null}\left(J_k\right), \quad v_k \in \operatorname{Range}\left(J_k^T\right) \quad \text{for all } k \in \mathbb{N}.
\end{align} 
Let $Z_k $ be an orthogonal basis for
the null space of $J_k$ which under Assumption \ref{ass.function} is a matrix in $\mathbb{R}^{n \times (n-m)}$. 
By \eqref{eq.SQP.sto}, it follows that
\begin{align}
    v_k &= -J_k^T (J_k J_k^T)^{-1} c_k, \qquad 
    \bar u_k  = -Z_k (Z_k^T  \bar W_k Z_k )^{-1} Z_k^T( \bar g_k + J_k^T y_k +  \bar W_k  v_k ), \label{eq.uk.sto} \\
  \bar \delta_k &= -(J_k J_k^T)^{-1} J_k ( \bar g_k + J_k^T y_k+ \bar W_k  \bar d_k). \label{eq.deltak.sto}
\end{align}
By \eqref{eq.orthogonal.sto}, \eqref{eq.uk.sto}, \eqref{eq.deltak.sto}, the Cauchy-Schwartz inequality, and Assumption \ref{ass.function}, 
{\allowdisplaybreaks
\begin{align}
    \|\bar  \delta_k \| &\le \| (J_k J_k^T)^{-1} J_k  \| \| \bar g_k + J_k^T y_k+\bar W_k  v_k+\bar W_k \bar  u_k\| \notag \\ 
    &\le \kappa_{J^{\dagger}} \left( \| \bar  g_k + J_k^T y_k+ \bar W_k v_k\|+\|\bar W_k \bar  u_k\| \right) \notag \\ 
    &\le  \kappa_{J^{\dagger}}  \left( \| Z_k (Z_k^T \bar  W_k Z_k )  Z_k^T \bar u_k \|  + \|\bar  W_k \bar u_k \| \right)  \notag \\
&\le   2\kappa_{J^{\dagger}} \kappa_W \|\bar d_k \| \label{eq.delta.ub.d}. 
\end{align}}

When $\|\bar d_k\| \le \frac{\kappa_d}{2}$, 
the condition in Assumption \ref{ass.d.bound.imply.difference.det} is satisfied since it follows by Assumptions~\ref{ass.M.inverse.bound.subsampled} that 
\begin{align*}
    \| d_k\| &\le \|\bar d_k\|  + \|d_k - \bar d_k\| \le \|\bar d_k\|  +  \bar M_k^{-1} \begin{bmatrix}
    g_k  -  \bar  g_k\\ 0
\end{bmatrix}  \le \|\bar d_k\|  + \Lambda \epsilon_{g_k}  \le \kappa_d.
\end{align*}
By Assumptions \ref{ass.d.bound.imply.difference.det} and the gradient error condition, 
\begin{align*}
    \|w_k - w^*\|& \le  2\mu_M^{-1} \left\|\begin{bmatrix}
     g_k + J_k^T y_k \\ c_k
\end{bmatrix}\right\| \\& \le  2\mu_M^{-1} \left\|\begin{bmatrix}
   \bar   g_k + J_k^T y_k \\ c_k
\end{bmatrix}+\begin{bmatrix}
   g_k - \bar   g_k \\0
\end{bmatrix}\right\| \\ &\le 2\mu_M^{-1} \left\| \bar M_k \bar M_k^{-1}\begin{bmatrix}
  \bar   g_k + J_k^T y_k \\ c_k
\end{bmatrix}  \right\| + 2\mu_M^{-1} \epsilon_{g_k} 
\\ & \le  2\mu_M^{-1}  \kappa_M \left\|\begin{bmatrix}
 \bar  d_k  \\\bar  \delta_k
\end{bmatrix}\right\|  + \tfrac{\kappa_w}{2}  
\\ & \le  2\mu_M^{-1}  \kappa_M (1+   2\kappa_{J^{\dagger}} \kappa_W ) \| \bar d_k\| + \tfrac{\kappa_w}{2} \le \kappa_w. 
\end{align*}
\end{proof}}

Our goal is to prove that in expectation $\|w_k-w^*\| \to 0$ at a superlinear rate. 
If there is a deterministic uniform bound for $\|w_k-w^*\|$, the result follows directly. 
However, in the general stochastic setting, this assumption does not hold. 
Instead, we introduce an 
assumption on the second moment of the distance of the iterates to $w^*$. This assumption is also made in the stochastic unconstrained setting \cite{bollapragada2019exact}.

\begin{assumption}
There exists a constant $\iota \in \mathbb{R}_{>0}$ such that for all $k\in \mathbb{N}$, 
$\mathbb{E}[\|w_k - w^* \|^2] \le \iota ( \mathbb{E}[\|w_k - w^* \|] )^2$.
\label{ass.second.moment.bound}
\end{assumption}

The local superlinear convergence guarantees, under the proximity and unit step size assumptions, of an adaptive sampling SQP method for \eqref{problem.expected_risk} is given below. 

\begin{theorem}\label{theorem.superlinear.convergence.stochastic}
Suppose Assumptions  
\ref{ass.function}, \ref{ass.M.Lipschitz.det1},  \ref{ass.g.var.sto}, 
\ref{ass.M.inverse.bound.subsampled},  \ref{ass.d.bound.imply.difference.det}, 
\ref{ass.second.moment.bound} hold. Further, suppose that for all $k \in \mathbb{N}$: 
\begin{itemize}
    \item[($\romannumeral1$)]  $ \|g_k - \bar g_k\| \le \min \left\{ \frac{\mu_M }{12\iota L_W \Lambda}, \frac{\kappa_d}{2\Lambda} \right\}$;  
    \item[($\romannumeral2$)]  $\left|S_{k}^g\right| \geq\left|S_{0}^g\right| \beta_{k}^{k}$, with $\left|S_{0}^g\right| \geq\left(6 \sigma_g \iota L_W \Lambda^2 \right)^{2}, \beta_{k}>\beta_{k-1}, \beta_{k} \rightarrow \infty$, and $\beta_{1}>1$; 
    \item[($\romannumeral3$)] $\left|S_{k}^H\right|>\left|S_{k-1}^H\right|$, with $\lim _{k \rightarrow \infty}\left|S_{k}^H\right|=\infty$, and $\left|S_{0}^H\right| \geq\left(4 \sigma_H\Lambda\right)^{2}$.
\end{itemize}
Then, if the starting point satisfies $\left\|\bar d_0\right\| \leq \min \left\{ \tfrac{\mu_M}{ 12 \iota \Lambda L_W\kappa_M (1+   2\kappa_{J^{\dagger}} \kappa_W )}, \tfrac{\kappa_d}{2} \right\}$,
it follows that $\mathbb{E}\left[\left\|w_{k}-w^{*}\right\|\right] \rightarrow 0$ at an R-superlinear rate, i.e., there exists a positive sequence $\left\{\xi_{k}\right\}$ such that $\mathbb{E}\left[\left\|w_{k}-w^{*}\right\|\right] \leq \xi_{k}$ and $\tfrac{\xi_{k+1}}{\xi_{k}} \rightarrow 0$.
\end{theorem}
\begin{proof}
It follows by Lemma~\ref{lemma.d.bound.imply.difference} and condition ($\romannumeral1$) that $\|w_0 -  w^{*}\|\le \frac{1}{3 \iota L_W \Lambda}$. The remainder of the proof is similar to \cite[Theorem 2.4]{bollapragada2019exact} if we view the Lagrangian function as the objective function. Thus, for brevity we omit the rest of the proof. 
\end{proof}
\begin{remark} Theorem \ref{theorem.superlinear.convergence.stochastic} is a stochastic analogue of Theorem \ref{theorem.quadratic.det}. The sample growth rates for the gradient and Hessian estimates match those derived for local superlinear guarantees 
in the unconstrained setting \cite{bollapragada2019exact}, i.e., the gradient sample size is required to grow at a rate faster than geometric and the Hessian sample size is required to grow monotonically. Condition ($\romannumeral1$) is required due to the usage of the $\|\bar d_k\|$ as neighborhood proximity condition. We acknowledge that this condition is relatively strong and is one of the main reasons we focus on problem \eqref{problem.empirical_risk} in Section~\ref{sec.analysis}. 
\end{remark}

\subsection{Finite-sum Problem~\eqref{problem.empirical_risk}}
\label{sec.superlinear.finite.sum}

In this subsection, we present local analysis for the SQP method (iterates updated via~\eqref{eq.subsample.unit.stepsize.update}) for solving~\eqref{problem.empirical_risk}. We start with a preliminary lemma that bounds the optimality gap between subsequent iterates. 
\begin{lemma}
\label{lemma.superlinear.lemma.subsampled}
Suppose that Assumptions \ref{ass.function},  \ref{ass.M.Lipschitz.det1}, \ref{ass.individual.f.subsampled}, and~\ref{ass.M.inverse.bound.subsampled} hold. Then, for $k \in \mathbb{N}$, 
\begin{align*}
    \| w_{k+1} - w^* \| \le \tfrac{\Lambda L_W}{2}  \| w_{k} - w^* \|^2  +  \Lambda \epsilon_{H_k} \| w_{k} - w^* \| + \Lambda \epsilon_{g_k}.
\end{align*}
\end{lemma}
\begin{proof}
 The proof is similar to \cite[Lemma 2.2]{bollapragada2019exact} if we remove the expectation,
 view the Lagrangian function as the objective function, and view the primal-dual variables as the decision variables. 
 Thus, for brevity we omit the proof.
\end{proof}

We now present the local fast convergence result (analogous to Theorem~\ref{theorem.superlinear.convergence.stochastic}).

\begin{theorem}
\label{theorem.superlinear.convergence.deterministic}
Suppose Assumptions 
\ref{ass.function},
\ref{ass.M.Lipschitz.det1}, \ref{ass.individual.f.subsampled}, 
\ref{ass.M.inverse.bound.subsampled}, and \ref{ass.d.bound.imply.difference.det} hold.  
Further, suppose that for all $k \in \mathbb{N}$:
\begin{itemize}
    \item[($\romannumeral1$)] $ |S_k^g| \ge \left(  1-  \tfrac{ \min \left\{ \mu_M \Lambda, 6 \Lambda L_W \kappa_d, 2\right\}\beta_k^{k/2}}{24 \Lambda^2 L_W (\kappa_1 + \mu_{1} \kappa_g )}\right)  N$ 
    with $\beta_k < \beta_{k-1}$, $\beta_k\to 0$, $\beta_0 \le 1$; 


    \item[($\romannumeral2$)] $\left|S_k^H\right|>\left|S_{k-1}^H\right|$, $\lim _{k \rightarrow \infty}\left|S_k^H\right|=N$, and $\left|S_0^H\right| \geq \max \left \{ \tfrac{8\Lambda ( \kappa_2 + \mu_{2} \kappa_H) -1}{8\Lambda ( \kappa_2 + \mu_{2} \kappa_H)}N, 1 \right\}$.
\end{itemize}
If the 
starting point satisfies 
\begin{equation}
\label{eq.superlinear.starting.finite}
    \left\|\bar d_0\right\| \leq \min \left\{ \tfrac{\mu_M}{ 12 \Lambda L_W\kappa_M (1+   2\kappa_{J^{\dagger}} \kappa_W )}, \tfrac{\kappa_d}{2} \right\}  ,
\end{equation}
then, 
$\left\|w_k-w^*\right\| \rightarrow 0$ at an $R$-superlinear rate, i.e., there is a positive sequence $\left\{\xi_k\right\}$ such that $\left\|w_k-w^*\right\| \leq \xi_k$ and $\tfrac{\xi_{k+1}}{\xi_k} \rightarrow 0$. 
In addition, for any $k \in \mathbb{N}$, $\left\|w_k-w^*\right\| \le \tfrac{1}{3 \Lambda L_W}$.
\end{theorem}
\begin{proof}
For brevity, we defer the proof to Appendix~\ref{app.superlinear.theorem}.
\end{proof}

\begin{remark}
The requirements for the rate of increase in the samples sizes of the gradient and Hessian estimates are similar to those in Theorem \ref{theorem.superlinear.convergence.stochastic} with the exception that the sample sizes are capped at $N$.  
If 
$\left|S_k^H\right| = \left|S_k^g\right| = N$ for all $k \in \mathbb{N}$, then this theorem recovers Theorem \ref{theorem.quadratic.det}.  
\end{remark}



\section{Convergence Analysis of Modified Line Search SQP Method}\label{sec.analysis}

In this section, we show that the proposed algorithm (Algorithm \ref{alg.SubsampledSQP.practical}) has an asymptotic superlinear convergence rate 
without the proximity and unit step size assumptions. Due to the challenges that arise with stochastic problems \eqref{problem.expected_risk} (discussed in Sections~\ref{sec.ourmethod} and \ref{sec.superlinear.stochastic}), 
we focus on solving finite-sum problems \eqref{problem.empirical_risk}. As a special case, we recover results for the deterministic problem \eqref{problem.deterministic}. Specifically, we show that 
after sufficiently many iterations, the unit step size is always accepted, and that the iterates approach a first-order critical point~\eqref{eq.first_order_stat}. The need for the proximity assumption is circumvented by the adaptivity of the parameter $\gamma_k$ (proximity check) in our proposed algorithm.

We start by presenting two fundamental results that are used in the analysis of SQP methods. The first lemma pertains to the relationship between the reduction in the model of the merit function and the search direction. 

\begin{lemma}
\label{lemma.d.upper.bound.by.Delta}
Suppose Assumptions \ref{ass.function} and \ref{ass.M.inverse.bound.subsampled} hold. 
There exists a constant $\kappa_l \in \mathbb{R}_{>0}$ such that  for all $k \in \mathbb{N}$, 
$\Delta l(x_{k}, \bar{\tau}_{k}, \bar{g}_{k}, \bar{d}_{k}) \geq \kappa_{l} \bar{\tau}_{k} \left\|\bar{d}_{k}\right\|^{2}$.
\end{lemma}
\begin{proof}
    The result is independent of the step size scheme. Similar results are proven in \cite[Lemma 3.4]{berahas2021sequential} 
    and, thus, for brevity we omit the proof.
\end{proof}

The next lemma shows that the merit parameter sequence is bounded below.

\begin{lemma}\label{lemm.tau}
Suppose Assumptions \ref{ass.function}, \ref{ass.individual.f.subsampled}, 
and \ref{ass.M.inverse.bound.subsampled} hold.
There exists $\bar \tau_{\min} \in \mathbb{R}_{>0}$ such that 
 $\bar \tau_k \ge \bar \tau_{\min}$ for all $k \in \mathbb{N}$. 
\label{lemma.tau.bound}
\end{lemma}
\begin{proof}
By Lemma \ref{lemma.diff.g.g.bar}, 
$  \| \bar g_k - g_k  \|   \le 2 \left(\tfrac{N-1}{N}\right)(\kappa_1 + \mu_{1} \kappa_g )   \le 2 (\kappa_1 + \mu_{1} \kappa_g )$.
The desired result then follows by \cite[Proposition 3.18]{berahas2021sequential}. \end{proof}

It is well established that, under reasonable assumptions, the SQP method with the classical line search condition \eqref{eq.line_search_cond}
is globally convergent~\cite{han1977globally, powell2006fast}. The next lemma implicitly proves such a result for our proposed method for the finite-sum setting \eqref{problem.empirical_risk}.  Specifically, there does not exist a constant $\gamma_{\tilde{k}}\in \mathbb{R}_{>0}$ such that $\|\bar d_k\| > \gamma_{\tilde{k}}$ for all $k > \tilde{k}$, 
which implicitly proves global convergence. The result ensures that Line \ref{line.else} of Algorithm~\ref{alg.SubsampledSQP.practical} is triggered infinitely often.

\begin{lemma}
\label{lemma.sigma.exist}
Suppose Assumptions \ref{ass.function}, \ref{ass.individual.f.subsampled}, and \ref{ass.M.inverse.bound.subsampled} hold. 
Let the relaxation parameter in \eqref{eq.line_search_cond_stoch}
be defined as
\begin{equation}
\label{eq.epsilon.A_k.def}
        \epsilon_{A_k} =  \bar{\tau}_{k} \left( \kappa_{J^{\dag}} \kappa_c  + \zeta^{-1} ( \kappa_g + \epsilon_{g_k} + \kappa_W \kappa_{J^{\dag}} \kappa_c)\right)  \epsilon_{g_k}   + 2\bar \tau_{k} \epsilon_{f_k}, 
\end{equation}
where $\epsilon_{f_k}$ and $\epsilon_{g_k}$ are given in Lemma~\ref{lemma.diff.g.g.bar}. 
Further, suppose that the sampling conditions in line \ref{line.obtain.estimate} of Algorithm \ref{alg.SubsampledSQP.practical} satisfy:
  \begin{itemize}
     \item[($\romannumeral1$)] $\sum_{k = 0}^{\infty}  \frac{N-\left|S_{k}^f\right|}{N}  < \infty$;  $\quad$ \textit{($\romannumeral2$)} $\sum_{k = 0}^{\infty}  \frac{N-\left|S_{k}^g\right|}{N}  < \infty$.
 \end{itemize}
For any iteration $\tilde k \in \mathbb{N}$, if $\|\bar d_{\tilde k}\| > \gamma_{\tilde k}$, then there must exist $k > {\tilde k}$ such that 
$\|\bar d_k\| \le \gamma_{\tilde k}$. 
\end{lemma}
{\allowdisplaybreaks
\begin{proof}
    We derive the desired result  by contradiction. Suppose that $\|\bar d_k\| > \gamma_{\tilde k} $ for all $k > {\tilde k} $, the algorithm will always select a step size using the classical line search condition \eqref{eq.line_search_cond_stoch} for all $k > {\tilde k} $. 
    By Assumption~\ref{ass.function}, \eqref{eq.orthogonal.sto}, \eqref{eq.uk.sto}, and \eqref{eq.epsilon.A_k.def}, and the fact that $\bar \alpha_k \le 1$,
\begin{align*}
   &\ \phi(x_{k}+\bar{\alpha}_{k} \bar{d}_{k}, \bar{\tau}_{k})-\phi(x_{k}, \bar{\tau}_{k})\\ 
   \leq &\  \bar{\alpha}_{k} \bar{\tau}_{k} g_{k}^{T} \bar{d}_{k}- \bar{\alpha}_{k} \left\|c_{k}\right\|_{1}+\tfrac{1}{2}\left(\bar{\tau}_{k} L_1 +\Gamma_1 \right) \bar{\alpha}_{k}^{2}\left\|\bar{d}_{k}\right\|^{2}
    \\
    = &\  - \bar{\alpha}_{k}  \Delta l(x_k,\bar\tau_k,\bar{g}_k,\bar{d}_k) +  \bar{\alpha}_{k} \bar{\tau}_{k} \bar{d}_{k}^T (g_{k} - \bar g_{k}) + \tfrac{1}{2}\left(\bar \tau_k L_1 + \Gamma_1\right) \bar{\alpha}_{k}^{2}\left\|\bar{d}_{k}\right\|^{2} \\ \leq &\   - \bar{\alpha}_{k} \left (1- \tfrac{\bar \tau_k L_1 + \Gamma_1 }{ 2 \bar \tau_k \kappa_l } \bar{\alpha}_{k} \right) \Delta l(x_k,\bar\tau_k,\bar{g}_k,\bar{d}_k) +  \bar{\alpha}_{k} \bar{\tau}_{k} \bar{d}_{k}^T (  g_{k} - \bar g_{k}) \\ 
    \leq &\   - \bar{\alpha}_{k} \left (1- \tfrac{\bar \tau_k L_1 + \Gamma_1 }{ 2 \bar \tau_k \kappa_l } \bar{\alpha}_{k} \right) \Delta l(x_k,\bar\tau_k,\bar{g}_k,\bar{d}_k) + \epsilon_{A_k} -  2\bar \tau_{k} \epsilon_{f_k}.
\end{align*}
When $\bar {\alpha}_k \le \min \left \{1, \tfrac{2\bar{\tau}_{k} \kappa_l (1-\eta)}{\bar{\tau}_{k} L_1 +\Gamma_1 } \right \}$, by Lemma~\ref{lemma.diff.g.g.bar}, condition~\eqref{eq.line_search_cond_stoch} is satisfied. Let  $\tilde\alpha_{\min} := \left\{1, \nu_{\alpha}\tfrac{2\bar{\tau}_{\min} \kappa_l (1-\eta)}{\bar{\tau}_{\min} L_1+\Gamma_1}\right \}$. 
Since $\tfrac{2\bar{\tau}_{k} \kappa_l (1-\eta)}{\bar{\tau}_{k} L_1+\Gamma_{1}} \ge  \tfrac{2\bar{\tau}_{\min} \kappa_l (1-\eta)}{\bar{\tau}_{\min} L_1+\Gamma_{1}} $, 
\eqref{eq.line_search_cond_stoch} is satisfied with $\bar \alpha_k = \tilde \alpha_{\min}$. Hence, for all $k \ge {\tilde k} $, 
\begin{align*}
   &\  \bar\phi(x_{k}+ \bar\alpha_k \bar {d}_{k}, \bar{\tau}_{k+1})-  \bar\phi(x_{k}, \bar{\tau}_{k}) \\  
   = &\  \bar  \phi(x_{k}+ \bar\alpha_k \bar{d}_{k},\bar {\tau}_{k+1}) - \bar  \phi(x_{k}+ \bar\alpha_k \bar{d}_{k}, \bar{\tau}_{k}) +    \bar\phi(x_{k}+ \bar\alpha_k \bar{d}_{k},\bar {\tau}_{k})- \bar \phi(x_{k},\bar {\tau}_{k})  \\  \leq &\     ( \bar\tau_{k+1} -  \bar\tau_{k}) f_{\inf} - \eta \tilde \alpha_{\min} \Delta l(x_k,\bar\tau_k,\bar{g}_k,\bar{d}_k)  + \epsilon_{A_k}. 
\end{align*}
For any $K> {\tilde k}$, summing over 
$k =  \tilde k, \cdots, K$, by Lemma~\ref{lemm.tau}
\begin{equation}\label{eq.final}
\begin{aligned}
   &\ \sum_{k= \tilde k}^{K}  \bar \phi(x_{k}+\bar {\alpha}_{k} \bar{d}_{k}, \bar{\tau}_{k+1}) - \bar  \phi(x_{k},\bar {\tau}_{k})   \\  \leq &\   -  \eta \tilde \alpha_{\min} \sum_{k=\tilde k}^{K} \Delta l(x_k,\bar\tau_k,\bar{g}_k,\bar{d}_k) +  \sum_{k=\tilde k}^{K} ( \bar\tau_{k+1} -  \bar\tau_{k}) f_{\inf} 
 + \sum_{k=\tilde k}^{K}  \epsilon_{A_k}    \\  \leq &\    -  \eta \tilde \alpha_{\min} \kappa_l \bar \tau_{\min} \sum_{k=\tilde k}^{K} \| \bar d_k\|^2+  \bar  \tau_{-1}   |f_{\inf} |  + \sum_{k=\tilde k}^{K}  \epsilon_{A_k}.
\end{aligned}
\end{equation}
By Assumption \ref{ass.function} the merit function is bounded below, and by conditions ($\romannumeral1$) and ($\romannumeral2$) 
the relaxation term $\epsilon_{A_k}$ is summable. Thus, by \eqref{eq.final}, $\sum_{k={\tilde k}}^{K} \|\bar  d_k\|^2  <  \infty$, 
which leads to a contradiction, and the desired result follows.  
\end{proof}}


The sampling condition ($\romannumeral1$) (and ($\romannumeral2$) with $|S_{k}^g|$) can be replaced by a sufficient condition following  d'Alembert's ratio test \cite{whittaker1920course}, $\lim_{k \to \infty}\sfrac{(N - |S_{k+1}^f|)}{(N - |S_{k}^f|)} < 1$. Another surrogate condition is to ensure that $|S_{k}^f| = N$ after sufficiently many iterations. 
Lemma \ref{lemma.sigma.exist} shows that Algorithm~\ref{alg.SubsampledSQP.practical} is guaranteed to eventually trigger the \emph{else} condition (Line~\ref{line.else}). Since \eqref{eq.line_search_cond_stoch} is not guaranteed to be satisfied with the unit step size (Line~\ref{line.alpha.1} of Algorithm~\ref{alg.SubsampledSQP.practical}), we need to examine the step sizes returned by the modified line search condition \eqref{eq.line_search_cond_modified_stoch}. Specifically, we show that there exists a lower bound for $\bar \alpha_k$ returned by 
\eqref{eq.line_search_cond_modified_stoch}, and  when $\|\bar d_k\|$ is sufficiently small, $\bar \alpha_k = 1$ satisfies \eqref{eq.line_search_cond_modified_stoch}.






\begin{lemma} 
Suppose the assumptions and conditions in Lemma~\ref{lemma.sigma.exist} hold.  
Additionally, let $|S_k^H| \ge \left( 1- \tfrac{\bar \epsilon_H}{2(\kappa_2 + \mu_{2} \kappa_H)} \right) N $ where $\bar \epsilon_H \in \left(0,  2(1-\eta)\kappa_l\right)$. 
For all 
$k\in \mathbb{N}$, \eqref{eq.line_search_cond_modified_stoch} is satisfied with 
$\bar\alpha_{k}  \ge \min \left \{1, \nu_{\alpha} \sqrt{\tfrac{6 \bar{\tau}_{\min} \kappa_l(1-\eta) - 3 \bar{\tau}_{\min} \bar \epsilon_H}{ (\bar{\tau}_{\min} L_2 +\Gamma_2 )\gamma_0 }} \right \}:= \bar \alpha_{\min}$. 
Moreover, if $\|\bar d_k\| \le \frac{ 6 \bar{\tau}_{\min} \kappa_l(1-\eta) - 3 \bar{\tau}_{\min} \bar \epsilon_H}{\bar{\tau}_{\min} L_2 +\Gamma_2}$, \eqref{eq.line_search_cond_modified_stoch} is satisfied with $\bar \alpha_k = 1$.
\label{lemma.dist.small.imply.alpha.1}  
\end{lemma}
\begin{proof}
When $|S_k^H| \ge \left( 1- \frac{\bar \epsilon_H}{2(\kappa_2 + \mu_{2} \kappa_H)} \right) N $, by Lemma~\ref{lemma.diff.g.g.bar}, 
it follows that $\epsilon_{H_k} = 2\tfrac{N-|S_k^H|}{N} (\kappa_2 + \mu_{2} \kappa_H)   \le \bar \epsilon_{H} < 2(1-\eta)\kappa_l$, 
for all $k \in \mathbb{N}$. 
Thus, it follows 
that
{\allowdisplaybreaks
\begin{align*}
      &\    \bar \phi(x_{k}+\bar \alpha_k \bar{d}_{k}, \bar{\tau}_{k})- \bar \phi(x_{k}, \bar{\tau}_{k}) \\
    \leq &\  \bar{\alpha}_{k} \bar{\tau}_{k} g_{k}^{T} \bar{d}_{k}- \bar{\alpha}_{k}\left\| c_{k}\right\|_{1}+\tfrac{1}{2}  \bar{\alpha}_{k}^2 \bar \tau_{k} \bar d_k^T  \nabla^2 f(x_k) \bar d_k  \\ 
    &\  +  \tfrac{1}{2}  \bar{\alpha}_{k}^2  \sum_{i=1}^m \left|\bar d_k^T \nabla^2 c_i(x) \bar d_k  \right| + \tfrac16  \bar{\alpha}_{k}^3 \left( \bar{\tau}_{k} L_2 +\Gamma_2 \right)  \|\bar d_k \|^3  + 2 \bar \tau_{k} \epsilon_{f_k}  \\ 
    = &\ \bar{\alpha}_{k} \left(  \bar{\tau}_{k} \bar g_{k}^{T} \bar{d}_{k} -   \left\|c_{k}\right\|_{1} \right) + 
    \tfrac{1}{2}\bar{\alpha}_{k}^2 \bar{\tau}_{k} \bar d_k^T \bar H_k \bar d_k  + \tfrac{1}{2}\bar{\alpha}_{k}^2  \sum_{i=1}^m \left|\bar d_k^T \nabla^2 c_i(x)\bar d_k  \right| 
    + \bar{\alpha}_{k} \bar{\tau}_{k} (g_k - \bar g_k )^T \bar d_k   \\ 
    &\  + \tfrac12 \bar{\alpha}_{k}^2 \bar{\tau}_{k}  \bar d_k^T ( \nabla^2 f(x_k) - \bar H_k )\bar d_k   +  \tfrac16 \bar{\alpha}_{k}^3  \left( \bar{\tau}_{k} L_2 +\Gamma_2 \right)  \|\bar d_k \|^3 + 2 \bar \tau_{k} \epsilon_{f_k} \\
    \leq &\ -  \bar{\alpha}_{k}  \Delta l(x_k,\bar\tau_k,\bar{g}_k, \bar{d}_k) + \tfrac12 \bar{\alpha}_{k}^2 \bar{\tau}_{k}  \bar \epsilon_H \|\bar d_k\|^2      +  \tfrac16 \bar{\alpha}_{k}^3  \left( \bar{\tau}_{k} L_2 +\Gamma_2 \right)  \|\bar d_k \|^3 \\ &\  +  \tfrac12 \bar{\alpha}_{k}^2 \bar{\tau}_{k} \bar d_k^T \bar H_k \bar d_k  + \tfrac{1}{2}  \bar{\alpha}_{k}^2 \sum_{i=1}^m \left|\bar d_k^T \nabla^2 c_i(x)\bar d_k  \right| + \epsilon_{A_k} \\
    \leq &\  -\bar \alpha_k \left(1 - \tfrac{  \bar \epsilon_H}{2\kappa_l} \bar \alpha_k  -\tfrac{\bar \tau_k L_2 + \Gamma_2 }{ 6 \bar \tau_k \kappa_l } \bar{\alpha}_{k}^2 \|\bar d_k\|  \right)  \Delta l(x_k,\bar\tau_k,\bar{g}_k, \bar{d}_k)   \\ 
    &\  +  \tfrac12 \bar{\alpha}_{k}^2 \bar{\tau}_{k} \bar d_k^T \bar H_k \bar d_k  + \tfrac{1}{2}  \bar{\alpha}_{k}^2 \sum_{i=1}^m \left|\bar d_k^T \nabla^2 c_i(x)\bar d_k  \right| + \epsilon_{A_k} \\
    \leq &\ -\bar \alpha_k \left(1 - \tfrac{  \bar \epsilon_H}{2\kappa_l}  -\tfrac{\bar \tau_k L_2 + \Gamma_2 }{ 6 \bar \tau_k \kappa_l } \bar{\alpha}_{k}^2\gamma_0 \right)  \Delta l(x_k,\bar\tau_k,\bar{g}_k, \bar{d}_k)  \\ &\  +  \tfrac12 \bar{\alpha}_{k}^2 \bar{\tau}_{k} \bar d_k^T \bar H_k \bar d_k  + \tfrac{1}{2}  \bar{\alpha}_{k}^2 \sum_{i=1}^m \left|\bar d_k^T \nabla^2 c_i(x)\bar d_k  \right|  + \epsilon_{A_k}, 
\end{align*}
where the first inequality follows due to Assumption~\ref{ass.function} 
and Lemma~\ref{lemma.diff.g.g.bar}, the second inequality follows by Lemma \ref{lemma.diff.g.g.bar} and the definitions of $ \Delta l(x_k,\bar\tau_k,\bar{g}_k, \bar{d}_k) $ and $ \epsilon_{A_k} $, the third inequality follows by Lemma \ref{lemma.d.upper.bound.by.Delta}, and the last inequality follows by the fact that $\bar \alpha_k \le 1$ and $\|\bar d_k\| \le \gamma_k \le \gamma_0$. When $\bar {\alpha}_k \le \min \left\{1, \sqrt{\tfrac{6 \bar{\tau}_{k} \kappa_l(1-\eta) - 3 \bar{\tau}_{k} \bar \epsilon_H}{ (\bar{\tau}_{k} L_2 +\Gamma_2 )\gamma_0}}  \right \}$, the relaxed line search condition \eqref{eq.line_search_cond_modified_stoch} 
is satisfied. 
Since 
$ \sqrt{\tfrac{6 \bar{\tau}_{k} \kappa_l(1-\eta) - 3 \bar{\tau}_{k} \bar \epsilon_H}{ (\bar{\tau}_{k} L_2 +\Gamma_2 )\gamma_0}}  \ge \sqrt{\tfrac{6 \bar{\tau}_{\min} \kappa_l(1-\eta) - 3 \bar{\tau}_{\min} \bar \epsilon_H}{ (\bar{\tau}_{\min} L_2 +\Gamma_2 )\gamma_0}}$, the desired conclusion 
holds with the line search procedure $\bar \alpha_k =\nu_{\alpha} \bar \alpha_k$. }


Plugging $\bar \alpha_k = 1$ into the third inequality above, 
{\allowdisplaybreaks
\begin{align}
\label{eq.inequality.model.reduction.showing.stepsize.dist}
        \bar \phi(x_{k}+ \bar{d}_{k}, \bar{\tau}_{k})- \bar \phi(x_{k}, \bar{\tau}_{k})  
&\  \leq  - \left(1 - \tfrac{  \bar \epsilon_H}{2\kappa_l}   -\tfrac{\bar \tau_k L_2 + \Gamma_2 }{ 6 \bar \tau_k \kappa_l } \|\bar d_k\| \right)  \Delta l(x_k,\bar\tau_k,\bar{g}_k, \bar{d}_k)  \nonumber  \\ 
&\ \qquad  +  \tfrac12  \bar{\tau}_{k} \bar d_k^T \bar H_k \bar d_k  + \tfrac{1}{2}   \sum_{i=1}^m \left|\bar d_k^T \nabla^2 c_i(x)\bar d_k  \right| + \epsilon_{A_k}.
\end{align}
When $\|\bar d_k\| \le \tfrac{ 6 \bar{\tau}_{\min} \kappa_l(1-\eta) - 3 \bar{\tau}_{\min} \bar \epsilon_H}{\bar{\tau}_{\min} L_2 +\Gamma_2}$, 
 it follows that $1 - \frac{  \bar \epsilon_H}{2\kappa_l}   -\frac{\bar \tau_k L_2 + \Gamma_2 }{ 6 \bar \tau_k \kappa_l } \|\bar d_k\| \ge \eta$. Thus, the modified line search condition \eqref{eq.line_search_cond_modified_stoch} 
 is satisfied. }
\end{proof}




We now present 
our main convergence result for Algorithm \ref{alg.SubsampledSQP.practical}. 
We establish asymptotic local  superlinear convergence results under appropriate sampling conditions.

\begin{theorem}
 Let $\{w_k\}_{k \in \mathbb{N}}$ be the primal-dual iterates generated by Algorithm \ref{alg.SubsampledSQP.practical}.  Suppose Assumptions 
\ref{ass.function},
\ref{ass.M.Lipschitz.det1}, \ref{ass.individual.f.subsampled}, 
\ref{ass.M.inverse.bound.subsampled}, and \ref{ass.d.bound.imply.difference.det} hold.
Further, suppose that the following conditions are satisfied for all $k \in \mathbb{N}$:
 \begin{itemize}
     \item[($\romannumeral1$)] $\sum_{k = 0}^{\infty}  \frac{N-\left|S_{k}^f\right|}{N}  < \infty$; $\quad$ \textit{($\romannumeral2$)} $\sum_{k = 0}^{\infty}  \frac{N-\left|S_{k}^g\right|}{N}  < \infty$; 
     \item[($\romannumeral3$)] $|S_k^H| \ge \left( 1- \frac{\bar \epsilon_H}{2(\kappa_2 + \mu_{2} \kappa_H)} \right) N $, where $\bar \epsilon_H < 2(1-\eta)\kappa_l$.
 \end{itemize}
 Moreover, suppose that for sufficiently large $\hat{k} \in \mathbb{N}$, the following sampling conditions are satisfied for all $k \ge \hat{k}$, where $t = k-\hat{k}$:
\begin{itemize}
   \item[($\romannumeral4$)]  $|S_k^g| \ge \left( 1- \frac{ 
\kappa_w \rho_0 \beta_{t}^{t/2}}{4\Lambda (\kappa_1 + \mu_{1} \kappa_g)}\right) N $ for some $\kappa_w \in \mathbb{R}_{>0}$, with $\beta_{t} < \beta_{t-1}$, $\beta_{t}\to 0$, 
$\beta_{1} \le \rho_0^4$ and $\rho_0: =\min \left\{\nu_{\gamma},  \tfrac{\nu_{\gamma}^2 \mu_M  }{4 \kappa_M (1 + 2 \kappa_{J^{\dag}} 
\kappa_W ) \max \{ 2 \Lambda  L_{\nabla \mathcal{L}}, \nu_{\gamma} \} }, \tfrac{\mu_M \Lambda}{2} \right\}$;
         \item[($\romannumeral5$)]  $\left|S_{k+1}^H\right|>\left|S_{k}^H\right|$, $\lim _{k \rightarrow \infty}\left|S_k^H\right|=N$, and $\left|S_{\hat{k}}^H\right| \geq \max \left \{ \frac{8\Lambda ( \kappa_2 + \mu_{2} \kappa_H) - \rho_0}{8\Lambda ( \kappa_2 + \mu_{2} \kappa_H)}N, 1 \right\}$.
\end{itemize}
 Then, 
$\left\|w_k-w^*\right\| \rightarrow 0$ at an $R$-superlinear rate.
\label{theorem.global.to.local.practical}
\end{theorem} 
{\allowdisplaybreaks
\begin{proof}
Conditions ($\romannumeral4$) and ($\romannumeral5$) imply, $\epsilon_{g_k} = 2 \tfrac{N- |S_k^g|}{N} (\kappa_1 + \mu_{1} \kappa_g)  \le \tfrac{ \kappa_w \rho_0 \beta_{t}^{t/2}}{2\Lambda}$ and $\epsilon_{H_{\hat{k}}} = 2 \tfrac{N- |S_{\hat{k}}^H|}{N} (\kappa_2 + \mu_{2} \kappa_H)  \le \tfrac{\rho_0}{4\Lambda}$. 
It follows by Lemma \ref{lemma.sigma.exist} that $\|\bar d_k\| \le \gamma_k$ infinitely often. We consider 
two cases:
\begin{itemize}
    \item[$(1)$] Line \ref{line.else10} of Algorithm \ref{alg.SubsampledSQP.practical} is triggered infinitely often; 
     \item[$(2)$] Line \ref{line.else10} of Algorithm \ref{alg.SubsampledSQP.practical} is not triggered infinitely often. 
\end{itemize}

Let us first consider case (1). Define 
\begin{equation*}
\begin{split}
 \kappa_w  &:= \min \ \left \{  \tfrac{\nu_{\gamma}}{3\Lambda L_W } , \tfrac{\kappa_d}{\rho_0} ,  
 \tfrac{\nu_{\gamma}^2 \mu_M  }{8 \kappa_M (1 + 2 \kappa_{J^{\dag}} \kappa_W ) \Lambda^2 L_{\nabla \mathcal{L}}  L_W  },
 \tfrac{\nu_{\gamma}^2 \min \left \{ \tfrac{\kappa_d}{2}, \tfrac{ 6 \bar{\tau}_{\min} \kappa_l(1-\eta) - 3 \bar{\tau}_{\min} \bar \epsilon_H}{\bar{\tau}_{\min} L_2 +\Gamma_2}  \right \}  }{   \max \{ 2 \Lambda  L_{\nabla \mathcal{L}}, \nu_{\gamma}  \} \rho_0 } \right \}, \\
\gamma_d &: =  \min \left \{  \tfrac{\mu_M \kappa_w }{4 \kappa_M (1 + 2 \kappa_{J^{\dag}} 
\kappa_W ) } , \tfrac{\kappa_d}{2}, \tfrac{ 6 \bar{\tau}_{\min} \kappa_l(1-\eta) - 3 \bar{\tau}_{\min} \bar \epsilon_H}{\bar{\tau}_{\min} L_2 +\Gamma_2}  \right \}.
\end{split}
\end{equation*}
Algorithm \ref{alg.SubsampledSQP.practical} sets $\gamma_{k+1} =\nu_{\gamma}\gamma_k$ when it enters Line \ref{line.else10}. Since $\gamma_{k+1} =\nu_{\gamma}\gamma_k$ happens infinitely often, there must exist some $\tilde k > 0$ such that the algorithm sets $\gamma_{\tilde k + 1} =\nu_{\gamma}\gamma_{\tilde k} \in [\nu_{\gamma}\gamma_d,\gamma_d]$. Suppose that $\left \|\bar d_{\tilde k+ 1} \right\| > \gamma_{\tilde k+ 1} $, it follows by Lemma \ref{lemma.sigma.exist} that  
there must exists $\hat{k} > \tilde k + 1$ such that  $\|\bar d_{\hat{k}} \| \le  \gamma_{\tilde k+ 1}$ for the first time, $\gamma_k$ is therefore not updated and $\gamma_{\hat{k}} = \gamma_{\tilde k+ 1} \in [\nu_{\gamma}\gamma_d,\gamma_d]$. If $\left \|\bar d_{\tilde k+ 1} \right\| \le \gamma_{\tilde k+ 1}$, let $\hat{k} = \tilde k+ 1$ and $t = k - \hat{k}$. We define sequences $\{ \xi_t \}_{t=0}^{\infty}$ and $\{ \rho_t \}_{t=1}^{\infty}$ as follows:
 \begin{align*}
\xi_{t+1}=\max \left\{\xi_t \rho_t, \beta_{t+1}^{(t+1) / 4}\right\}, \quad \xi_0=1, \quad
\rho_t=\tfrac{\Lambda L_W \kappa_w \xi_t}{2}+\Lambda \epsilon_{H_{\hat{k}}} \tfrac{N-|S_k^H|}{N-|S_{\hat{k}}^H|}+\tfrac{\rho_0 \beta_{t}^{t/4}}{2 }.  
 \end{align*}
 We 
 prove the following results for all $k \ge \hat{k}$ ($t = k-\hat{k}$, $t \ge 0$):
\begin{itemize} 
    \item[\textbf{(a)}]  \textbf{(Unit step size)}: 
    $\bar  \alpha_k =1$;
    \item[\textbf{(b)}] \textbf{(Superlinear convergence)}: $\|w_{k} - w^*\| \le \kappa_w \xi_{t}$;
  \item[\textbf{(c)}] \textbf{(Controlled sequences)}: $ \xi_{t} \le 1, \rho_t < 1$; 
    \item[\textbf{(d)}]  \textbf{(}\boldmath{$\|\bar  d_{k+1} \| \le \gamma_{k+1}$}\textbf{)}:\unboldmath{
 $ 2 \Lambda  L_{\nabla \mathcal{L}}  \kappa_w \xi_{t+1} \le \gamma_{\hat{k}} \nu_{\gamma}^{t + 1}$}. 
\end{itemize} 
Note that \textbf{(a)}, \textbf{(b)}, \textbf{(c)} are required in Theorem \ref{theorem.superlinear.convergence.deterministic}, we also need \textbf{(d)} to enforce the algorithm to always enter line \ref{line.else}. Otherwise, our algorithm may consider to employing the line search condition \eqref{eq.line_search_cond_stoch} and the resulting step size could be smaller than 1. 

Let us first give two general results. Assume that  \textbf{(a)} and \textbf{(b)} hold  for iteration $k$, and consider iteration $k+1$. By Lemma \ref{lemma.superlinear.lemma.subsampled}, conditions ($\romannumeral4$) and ($\romannumeral5$), it follows that
 \begin{align}
    \left\|w_{k+1}-w^*\right\| 
    & \leq  \tfrac{\Lambda L_W }{2}\left\|w_k-w^*\right\|^2 + \Lambda \epsilon_{H_{k}} \left\|w_{k}-w^*\right\| + \tfrac{\kappa_w  \rho_0 \beta_{t}^{t/2}}{2} \notag \\
    & \leq \tfrac{\Lambda L_W }{2} \kappa_w^2 \xi_t^2  +  \Lambda \epsilon_{H_{\hat{k}}} \tfrac{N-|S_k^H|}{N-|S_{\hat{k}}^H|} \kappa_w \xi_t  + \tfrac{\kappa_w  \rho_0 \beta_{t}^{t/2}}{2} \notag\\
    & \leq \kappa_w \xi_t \left(\tfrac{\Lambda L_W \kappa_w \xi_t}{2}+\Lambda \epsilon_{H_{\hat{k}}} \tfrac{N-|S_k^H|}{N-|S_{\hat{k}}^H|} +\tfrac{\rho_0 \beta_{t}^{t/4}}{2 }\right). \label{eq.w_k+1.w_k}
\end{align}
Moreover, 
when $\left\|w_{k}-w^*\right\|  \le \kappa_w$, 
by 
\eqref{eq.SQP.sto}, Assumptions \ref{ass.M.Lipschitz.det1} and condition ($\romannumeral4$) 
that 
\begin{align}
        \|\bar d_{k}\| 
\le \| \bar M_{k}^{-1}\| \left\|\begin{bmatrix}
   \bar  g_{k} + J_{k}^T y_{k} \\ c_{k}
\end{bmatrix}\right\| 
&\le \| \bar M_{k}^{-1}\| \left\|\begin{bmatrix}
    g_{k} + J_{k}^T y_{k} \\ c_{k}
\end{bmatrix} + \begin{bmatrix}
   \bar  g_{k} - g_{k}  \notag \\ 0
\end{bmatrix}  \right\|  \notag \\ &\le \Lambda L_{\nabla \mathcal{L}} 
\left\|w_{k}-w^*\right\| + \tfrac{ \kappa_w \rho_0 \beta_{t}^{t/2}}{2} .\label{eq.d.ub.by.dist1} 
\end{align}

We then use induction to prove \textbf{(a)--
(d)}. For the base case $k =\hat{k}$, since $\|\bar d_{\hat{k}} \| \le \gamma_d$, \textbf{(a)} is satisfied by Lemma \ref{lemma.dist.small.imply.alpha.1}. Since $\rho_0 \le \tfrac{\mu_M \Lambda}{2}$ and $\kappa_w \le \tfrac{\kappa_d}{\rho_0}$, we have
$ \epsilon_{g_{\hat{k}}} \le \tfrac{\kappa_w \rho_0}{2\Lambda} \le \min \left\{ \tfrac{\mu_M \kappa_w}{4}, \frac{\kappa_d}{2\Lambda} \right\}$. 
It then follows by Lemma \ref{lemma.d.bound.imply.difference} that $\|w_{\hat{k}} - w^*\| \le  \kappa_w$  by the definitions of $\gamma_d$ and $\kappa_w$,  \textbf{(b)} is then satisfied. Since $\nu_{\gamma} < 1$,  \textbf{(c)} holds trivially by the definitions of $\xi_0$  and $\rho_0$. By \eqref{eq.w_k+1.w_k}, 
$\kappa_w \le   \tfrac{\nu_{\gamma}^2 \mu_M  }{8 \kappa_M (1 + 2 \kappa_{J^{\dag}} 
\kappa_W ) \Lambda^2 L_{\nabla \mathcal{L}} L_W }$, 
$\xi_0 = 1$, the definition of $\rho_0$, and conditions ($\romannumeral4$) and ($\romannumeral5$) it follows that 
\begin{align*}
    \left\|w_{\hat{k}+1}-w^*\right\| 
    \le  \kappa_w \xi_0 \left(\tfrac{\Lambda L_W \kappa_w}{2}+\Lambda \epsilon_{H_{\hat{k}}} +\tfrac{\rho_0}{2 }\right)  \le \kappa_w \rho_0.
\end{align*}
Since $\beta_{1}^{1/4} \le \rho_0$, we have $\xi_1 = \rho_0$, and it follows from the definition of $\rho_0$ that  
\begin{align}
\label{eq.c.intermediate1}
\max \{ 2\Lambda  L_{\nabla \mathcal{L}} \kappa_w    \rho_0, \nu_{\gamma} \kappa_w    \rho_0  \}  \le    \tfrac{\nu_{\gamma}^2 \mu_M \kappa_w }{4 \kappa_M (1 + 2 \kappa_{J^{\dag}} \kappa_W ) } .  
\end{align} 
Furthermore, since $\kappa_w  \le  \tfrac{\nu_{\gamma}^2}{ \max \{ 2 \Lambda  L_{\nabla \mathcal{L}}, \nu_{\gamma}  \} \rho_0} \min \left \{ \tfrac{\kappa_d}{2}, \tfrac{ 6 \bar{\tau}_{\min} \kappa_l(1-\eta) - 3 \bar{\tau}_{\min} \bar \epsilon_H}{\bar{\tau}_{\min} L_2 +\Gamma_2}  \right \}  $,
\begin{align}
\label{eq.c.intermediate2}
\max \{ 2\Lambda  L_{\nabla \mathcal{L}} \kappa_w   \rho_0, \nu_{\gamma} \kappa_w    \rho_0  \}   \le \nu_{\gamma}^2  \min \left \{  \tfrac{\kappa_d}{2}, \tfrac{ 6 \bar{\tau}_{\min} \kappa_l(1-\eta) - 3 \bar{\tau}_{\min} \bar \epsilon_H}{\bar{\tau}_{\min} L_2 +\Gamma_2}  \right \}.  
\end{align}
It then follows from \eqref{eq.d.ub.by.dist1}, \eqref{eq.c.intermediate1},  and \eqref{eq.c.intermediate2}  
that 
\begin{align*}
\|\bar d_{\hat{k} + 1}\| \le \Lambda L_{\nabla \mathcal{L}} \kappa_w   \xi_1  + \tfrac{ \kappa_w \rho_0 \nu_{\gamma}}{2}  \le \nu_{\gamma}^2  \gamma_d \le\nu_{\gamma}\gamma_{\hat{k}} \le\gamma_{\hat{k}+1},  
\end{align*}
and \textbf{(d)} is satisfied for the base case. 

Let us then assume that \textbf{(a)} and \textbf{(b)} hold  for iteration $k$.  At iteration $k+1$, \textbf{(a)} is satisfied since $\|\bar d_{k+1}\| \le \gamma_{k+1} \le \gamma_d$. \textbf{(b)} is satisfied since $\left\|w_{k+1}-w^*\right\| \leq \kappa_w \xi_t \rho_t \le \kappa_w \xi_{t+1}$.
As for \textbf{(c)}, it follows that 
\begin{align*}
\xi_{t+1} \le \max \left\{ \rho_t, \beta_{1}^{(t+1) / 4}\right\} < 1, \;\ \rho_{t+1} \le  \tfrac{\Lambda L_W \kappa_w}{2}+\Lambda \epsilon_{H_{\hat{k}}} +\tfrac{\rho_0 }{2 } = \rho_0 \le\nu_{\gamma}< 1.
\end{align*}
Since \textbf{(a)} and \textbf{(b)} hold  for iteration $k+1$, we can apply \eqref{eq.w_k+1.w_k} again at iteration $k+1$ and obtain $\left\|w_{k+2}-w^*\right\| \leq \kappa_w \xi_{t+1} \rho_{t+1}$ by the definition of $\rho_{t+1}$. 
It then follows from \eqref{eq.d.ub.by.dist1}, \eqref{eq.c.intermediate1}, and \eqref{eq.c.intermediate2} that   
\begin{align*}
    \|\bar d_{k+2}\| &\le \Lambda L_{\nabla \mathcal{L}}  \|w_{k+2} - w^* \| + \tfrac{ \kappa_w \rho_0 \beta_{t+2}^{(t+2)/2}}{2}   \le \Lambda L_{\nabla \mathcal{L}}  \kappa_w \xi_{t+1} \rho_{t+1} + \tfrac{ \kappa_w \rho_0 \beta_{1}^{(t+2)/4}}{2}    \\ &\le \tfrac{ \gamma_{\hat{k}} \nu_{\gamma}^{t+1}   \rho_{t+1}}{2}  + \tfrac{ \kappa_w \rho_0 \nu_{\gamma}^{t+2}}{2}
    \le  \tfrac{\gamma_{\hat{k}} \nu_{\gamma}^{t+2}}{2} + \tfrac{\nu_{\gamma} \gamma_d  \nu_{\gamma}^{t+2}}{2}  = \gamma_{\hat{k}} \nu_{\gamma}^{t+2}.    
\end{align*}
\textbf{(d)} is therefore satisfied. We then conclude that $\{\xi_t\} \to 0$ by the fact that $\rho_t \le\nu_{\gamma}$ and $\beta_{t} \to 0$.  
In addition, ($\romannumeral4$) and ($\romannumeral5$) implies that  $\{\rho_t\} \to 0$. We can conclude that  
 \begin{align*}
\lim_{t \rightarrow \infty} \tfrac{\xi_{t+1}}{\xi_t} 
=\lim _{t \rightarrow \infty} \max \left\{\rho_t, \tfrac{\beta_{t+1}^{(t+1) / 4}}{\xi_t}\right\} \leq \lim _{t \rightarrow \infty} \max \left\{\rho_t,\left(\tfrac{\beta_{t+1}}{\beta_{t}}\right)^{t / 4} \beta_{t+1}^{1 / 4}\right\} 
=0.
\end{align*}
Hence, $\|w_k - w^*\| \to 0$ at an R-superlinear rate. 

As for case when Line \ref{line.else10} of Algorithm \ref{alg.SubsampledSQP.practical} is not triggered infinitely often, 
Algorithm \ref{alg.SubsampledSQP.practical} either enters Line \ref{line: d_large.classical} or Line \ref{line.alpha.1} after some iteration $k_c$. In this scenario, line search condition \eqref{eq.line_search_cond_stoch} is always considered and $\gamma_k$ remains fixed at some $\bar \gamma$ after iteration $k_c$. 
It then follows by the proof of Lemma \ref{lemma.sigma.exist} that there exists $k_d \in \mathbb{R}_{>0}$ such that $\|\bar d_k\| \le \min \{\bar \gamma,\bar \gamma_d\} $  for all $k \ge k_d$ ($\bar \gamma_d$ is defined in Theorem~\ref{theorem.superlinear.convergence.deterministic}), which implies  Algorithm \ref{alg.SubsampledSQP.practical} enters Line \ref{line.alpha.1} for all $k \ge \max\{k_c, k_d\} $. Let $k_s = \max \{ k_c, k_d, \hat k\}$, excluding first $k_s$ iterations, and assuming (without loss of generality) that  $w_{k_s}$ as the starting point, it follows by the definition of $\kappa_w$ and $\rho_0$ that $\kappa_w \rho_0 \le  \tfrac{1}{3\Lambda L_W} \min \left\{ \tfrac{\mu_M \Lambda}{2}, 3 \Lambda L_W \kappa_d, 1 \right\}$. In addition, $\beta_{k}<\beta_{k-1}$ and $\rho_0 < 1$ hold, and as a result, conditions ($\romannumeral1$) and ($\romannumeral2$) in Theorem~\ref{theorem.superlinear.convergence.deterministic} are satisfied. Thus, considering $w_{k_s}$ as the starting point and applying the result from Theorem~\ref{theorem.superlinear.convergence.deterministic} completes the proof. \end{proof}}


\begin{remark}
Theorem \ref{theorem.global.to.local.practical} provides an asymptotic  superlinear convergence rate with respect to $\|w_k - w^*\|$. 
It unifies the  global convergence and the local superlinear convergence results of classical SQP methods. The sampling conditions ($\romannumeral1$), ($\romannumeral2$) are required for the global convergence result. 
Condition  ($\romannumeral3$) is needed for proving that the modified line search condition  \eqref{eq.line_search_cond_modified_stoch} produces unit step size when $\|\bar d_k\|$ is sufficiently small, and conditions ($\romannumeral4$) and ($\romannumeral5$) are  required for showing the superlinear convergence rate of $\|w_k - w^*\|$ and need to be satisfied after $\hat{k}$ iterations.  
 The iteration $\hat{k}$ determines when $\|\bar d_{\hat{k}}\|$ is sufficiently small (similar to the conditions required in Theorem \ref{theorem.superlinear.convergence.deterministic}). Although $\hat{k}$ depends on unknown problem-specific constants, the initial conditions of ($\romannumeral4$) and ($\romannumeral5$) are guaranteed to be satisfied after a sufficient number of iterations as long as   ($\romannumeral2$) and $\lim_{k \to \infty} \left|S_k^H\right| = N$ hold.
\end{remark}

\section{Practical Inexact SQP}
\label{sec.inexact}


In Algorithm \ref{alg.SubsampledSQP.practical}, the exact solution of the linear system \eqref{eq.SQP.sto} is required 
to compute a step. 
This poses challenges in large-scale ($n$ and $m$) settings. 
In this section, we present an inexact matrix-free variant of Algorithm \ref{alg.SubsampledSQP.practical} 
that utilizes the Minimum Residual (MINRES) method \cite{choi2011minres} to solve \eqref{eq.SQP.sto} at every iteration. 
The rationale behind choosing MINRES over other widely used approaches lies in the facts that the matrix $\bar M_k$ is not assumed to be positive definite, that the method can be implemented matrix-free, and due to the robust theoretical guarantees and empirical performance~\cite{choi2011minres,fong2012cg,liu2022newton}.

Let $(\bar d_k^{(t)} ,\bar \delta_k^{(t)} )$ denote the approximate solution obtained at the $t$-th iteration of the  MINRES method (at the $k$-th iteration of Algorithm \ref{alg.SubsampledSQP.practical}) such that 
\begin{equation}  
\begin{bmatrix}
    \rho_k^{(t)} \\  r_k^{(t)}
\end{bmatrix} =
\begin{bmatrix}
   \bar W_k  & J_k^T \\ J_k & 0
\end{bmatrix}
\begin{bmatrix}
     \bar {d}_k^{(t)} \\ \bar \delta_k^{(t)}
\end{bmatrix}  + \begin{bmatrix}
  \bar  {g}_k + J_k^T  y_k \\ c_k
\end{bmatrix},
\label{eq.SQP.inexact.res}
\end{equation}
where 
$\rho_k^{(t)} \in \mathbb{R}^{n}$ and $ r_k^{(t)} \in \mathbb{R}^{m}$ denote the vectors of residuals. 
In the practical version of our algorithm, the residual vectors are nonzero and we utilize the inexact solutions $\bar{d}_k^{(t)}$, $\bar{\delta}_k^{(t)}$ in lieu of the exact solutions $\bar{d}_k, \bar{\delta}_k$ of subproblem \eqref{eq.SQP.sto} in Algorithm~\ref{alg.SubsampledSQP.practical}. Restrictions are imposed on the residual vectors which in turn provide lower bounds on the number of MINRES iterations. 

For ease of presentation and analysis, we assume the true gradient of the objective function is available and used ($\bar g_k = g_k = \nabla f_k$) and only approximate the Hessian. 
One can consider the setting with inexact gradients at the cost of weaker theoretical results and more complicated analysis. For brevity we do not provide 
a comprehensive convergence analysis of the practical version of Algorithm \ref{alg.SubsampledSQP.practical}.  Instead, we explore the 
relationship between the accuracy in  \eqref{eq.SQP.inexact.res}, the Hessian sample size, 
and the rate of convergence of the inexact method. Similar to Section \ref{sec.superlinear} we make the proximity and unit step size assumptions.  

The following lemma characterizes the accuracy  of the approximate solutions obtained with respect to the number of MINRES iterations.  

\begin{lemma}
\label{lemma.MINRES.res.convergence}
Suppose Assumptions \ref{ass.function} and \ref{ass.M.inverse.bound.subsampled} hold. There exists $\theta \in (0,1)$ such that for all $k,t\in \mathbb{N}$,
$\left\|  \begin{bmatrix}
    \rho_k^{(t)} \\  r_k^{(t)}
\end{bmatrix}  \right\| \le 2 \theta^t \left\| \begin{bmatrix}
     {g}_k + J_k^T  y_k \\ c_k
\end{bmatrix}   \right\|$. 
\end{lemma}
\begin{proof}
The proof can be found in Appendix \ref{app.minres.res.prove}. 
\end{proof}

We make two comments about the above result. First, the constant $\theta$ has dependence on the condition number of the matrix $\bar M_k$ \eqref{eq.SQP.sto}. Second, if $t = m + n$ MINRES iterations are performed, then the exact solution of the linear system is guaranteed to be obtained under the given assumptions \cite{choi2011minres}.

The next lemma is an inexact (in terms of the solution of the linear system) counterpart of Lemma~\ref{lemma.superlinear.lemma.subsampled}.


\begin{lemma}
\label{lemma.inexact.linear.intermediate}
  Suppose Assumptions \ref{ass.function}, \ref{ass.M.Lipschitz.det1}, 
  and \ref{ass.M.inverse.bound.subsampled} hold. For all $k \in \mathbb{N}$, let $|S_k^g| = N$ and $|S_k^H| 
  \leq N$. Then, when $\| w_{k} - w^* \| \le r$ for some $r \in \mathbb{R}_{>0}$ (defined in Assumption~\ref{ass.M.Lipschitz.det1}), 
\begin{align*}
  \| w_{k+1} - w^*  \|  \le \tfrac{   \Lambda L_W}{2}   \| w_{k} - w^* \|^2 + 2 \Lambda\left( \left(\sfrac{(N-|S_{k}^H|)}{N}\right)(\kappa_2 + \mu_{2} \kappa_H) 
 +  \theta^t L_{\nabla \mathcal{L}} \right) \| w_{k} - w^* \|.
\end{align*}
\end{lemma}
\begin{proof} 
The proof is identical to \cite[Lemma 3.1]{bollapragada2019exact} if we remove the expectation, view the Lagrangian function as the objective function, and view the primal-dual variables as the decision variables. Thus, for brevity we omit the proof.
\end{proof}

Under the proximity and unit step size assumptions, we prove a linear rate of convergence for the proposed practical inexact matrix-free SQP method. 


\begin{theorem}\label{theorem.inexact.linear}
 Suppose Assumptions \ref{ass.function}, \ref{ass.M.Lipschitz.det1}, 
 and \ref{ass.M.inverse.bound.subsampled} hold. 
For $k \in \mathbb{N}$, let $|S_k^g| = N$ and $|S_k^H| \ge \max \left \{ \tfrac{12\Lambda ( \kappa_2 + \mu_{2} \kappa_H) -1}{12\Lambda ( \kappa_2 + \mu_{2} \kappa_H)}N, 1 \right\} $. If 
$t \ge \frac{\log \left( 12 \Lambda L_{\nabla \mathcal{L}}  \right) }{\log \left( \frac{1}{\theta}  \right) }$ 
and 
$\left\|\bar d_0\right\| \leq \bar \gamma_d$ ($\bar \gamma_d$ defined in Theorem~\ref{theorem.superlinear.convergence.deterministic}), 
then 
$\|w_{k+1} - w^* \| \le \frac{1}{2}\|w_{k} - w^* \|$.
\end{theorem}
{\allowdisplaybreaks
\begin{proof} By the conditions given in the theorem, it follows that 
$\left(\sfrac{(N-|S_{k}^H|)}{N}\right)(\kappa_2 + \mu_{2} \kappa_H) \le \tfrac{1}{12}$ and $\Lambda L_{\nabla \mathcal{L}} \theta^t \le \tfrac{1}{12}$. Moreover, as proved in Theorem~\ref{theorem.superlinear.convergence.deterministic}, $\| w_0-w^* \| \le \tfrac{1}{3\Lambda L_W }$. 
The remainder of the proof is similar to \cite[Theorem 3.2]{bollapragada2019exact}, with $r = \tfrac{1}{3\Lambda L_W }$ in Lemma~\ref{lemma.inexact.linear.intermediate}.
\end{proof}}


\begin{remark} 
    Theorem \ref{theorem.inexact.linear} shows that under the proximity and unit step size assumptions, if a sufficient number of MINRES iterations are employed, then the primal and dual iterates converge at a linear rate. 
    The linear rate of convergence 
    is a fixed constant 
    and the number of required MINRES iterations has only a logarithmic dependence on the condition number. 
    In the inexact gradient setting, one can prove a similar result if the gradient approximation is sufficiently small as compared to the proximity measure. 
    We note that the analysis presented in this section can be integrated into the 
    adaptive algorithm of Section~\ref{sec.analysis}. Finally, in the stochastic setting 
    \eqref{problem.expected_risk}, one requires that the gradient approximation error is sufficiently small 
    and Assumption \ref{ass.second.moment.bound} to derive a linear rate of convergence in expectation. 
\end{remark}



\section{Numerical Experiments}
\label{sec.numerical}

In this section, we demonstrate the empirical performance of our proposed algorithm in both the deterministic and stochastic settings using a MATLAB implementation. We first compare the performance of  Algorithm~\ref{alg.SubsampledSQP.practical} to the classical line search SQP method and other adaptations (discussed in Section~\ref{sec.gap}) on a subset of the 
nonlinear equality constrained problems from the CUTEst collection \cite{bongartz1995cute}. In the stochastic setting, we compare with the adaptive step size SQP method proposed in \cite{berahas2021sequential} on  equality constrained logistic regression problems. Finally, we compare the exact and inexact variants of Algorithm~\ref{alg.SubsampledSQP.practical}.

\subsection{Deterministic Setting}
\label{sec.numerical.det}
The goal of this section is to demonstrate the robustness and efficiency of our proposed modified line search SQP method in the deterministic setting. 
We compare our proposed algorithm to the classical line search SQP method~\cite{han1977globally} and other adaptations (discussed in Section~\ref{sec.gap}) on the CUTEst collection of test problems. Among 123 equality constrained test problems
, we select 41 problems with 
the following characteristics: 
(1) $f$ is not a constant function; (2) $n+m \le 1000$; (3) the LICQ holds at all iterates in all runs of all algorithms (4 LICQ failures); 
and, (4) at least one method was able to solve the problem 
(12 unsolved). 

We consider 5 methods in this subsection: 
\textbf{SQP-L1} \cite{han1977globally}, \textbf{2nd-corr} \cite{fukushima1986successive}, \textbf{Watchdog} \cite{chamberlain1982watchdog}, \textbf{SQP-AugLag} \cite{schittkowski1982nonlinear} and \textbf{Our method} (Algorithm~\ref{alg.SubsampledSQP.practical}).
Common parameters, used across methods, were set as: $\tau_{-1} = 1$, $\nu_{\alpha} = 0.5$, $\eta = 10^{-4}$ as in \cite{berahas2021sequential}. Method-specific parameters were set as: 
\begin{itemize}[leftmargin=0.75cm]
    \item \textbf{Our Method} (Algorithm \ref{alg.SubsampledSQP.practical}): 
    $\gamma_0 = 0.999\| d_0\|$, $\nu_{\gamma} = 0.7$; 
   \item \textbf{Watchdog \cite{chamberlain1982watchdog}}: 5 iterations for the relaxed step before a restart, as recommended in \cite[Section 15.6]{NoceWrig06};
   \item \textbf{SQP-AugLag \cite{schittkowski1982nonlinear}}:  $r_{-1} = 10^6$ (tuned over $r_{-1} = \{10,10^2,\cdots,10^8\}$). 
\end{itemize}
Although there are two additional hyperparameters  $\gamma_0$ and $\nu_{\gamma}$ in  Algorithm  \ref{alg.SubsampledSQP.practical}, the performance of the method is robust with respect to these hyperparameters, and the choices stated above work well across most problems tested. 
For \textbf{SQP-AugLag}, we found that the performance is sensitive to the initial penalty parameter $r_{-1}$, and tuned this value. 
A run terminated with a message of success if 
for $k \leq 100$,  
\begin{align*}
    \left\|g_k+J_k^T y_k\right\|_{\infty} \leq 10^{-6} \max \left\{1,\left\|g_0+J_0^T y_0\right\|_{\infty}\right\} \text { and }\left\|c_k\right\|_{\infty} \leq 10^{-6} \max \left\{1,\left\|c_0\right\|_{\infty}\right\}.
\end{align*}
Otherwise, the run was considered a failure.

Figure \ref{fig.DM_v6} presents the performance of the methods in terms of iterations and function evaluations using Dolan-Mor\'e performance profiles \cite{dolan2002benchmarking}. As is clear, given the budget of iterations, our proposed method is robust and efficient, and, in fact, slightly outperforms the other methods on this test set. Moreover, we investigate the 7 instances for which the performance of the \textbf{SQP-L1} method differs from that of our proposed line search method indicating that the modified line search condition was triggered (i.e., Line \ref{line.else10}, Algorithm \ref{alg.SubsampledSQP.practical}). For these problems, the number of iterations for which the modified line search is employed is not insignificant; see Table~\ref{table.modified.percentage}. We present a comparison of the two methods in the form of Morales profiles \cite{morales2002numerical} (Figure \ref{fig.Mo}). While the overall performance is comparable, the results suggest that when our proposed method outperforms \textbf{SQP-L1} method, it does so by a significant margin. 

\begin{figure}
    \centering \includegraphics[width=0.45\textwidth,height=0.28\linewidth,clip=true,trim=10 180 50 200]{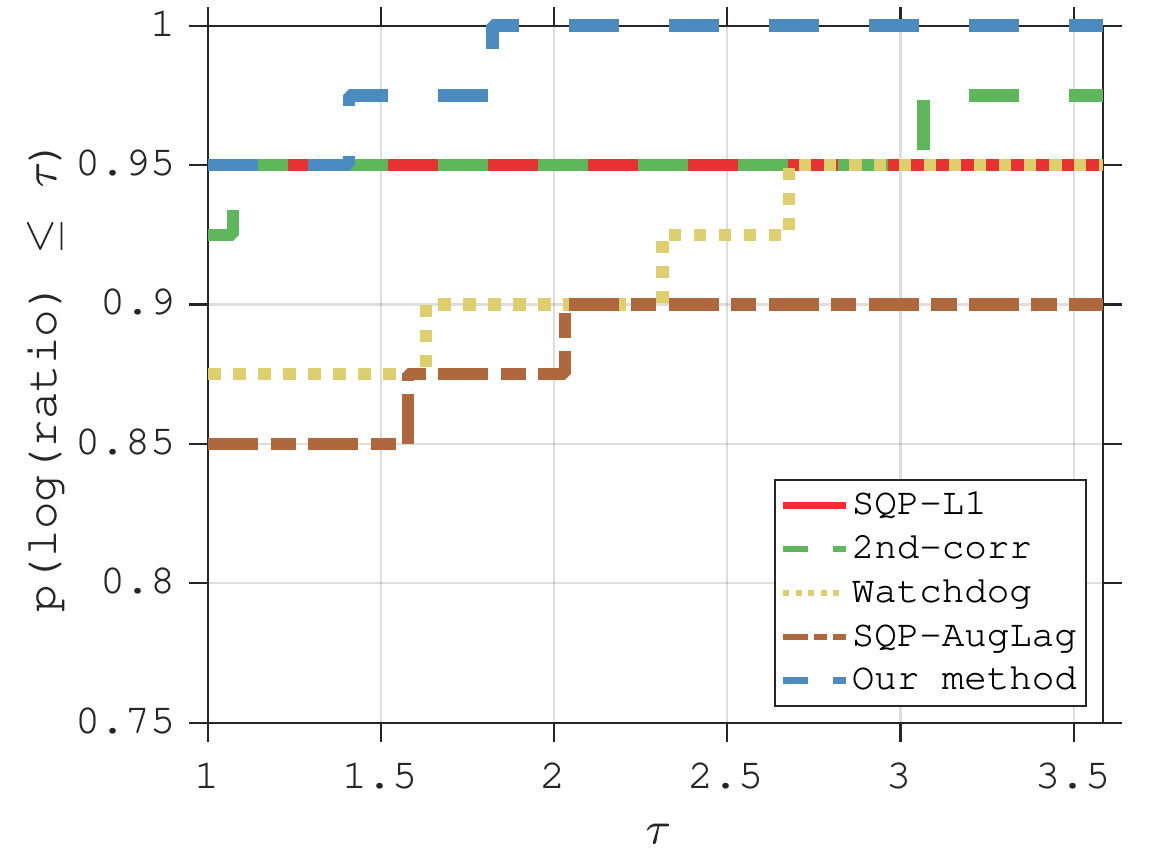}
    \includegraphics[width=0.45\textwidth,height=0.28\linewidth,clip=true,trim=10 180 50 200]{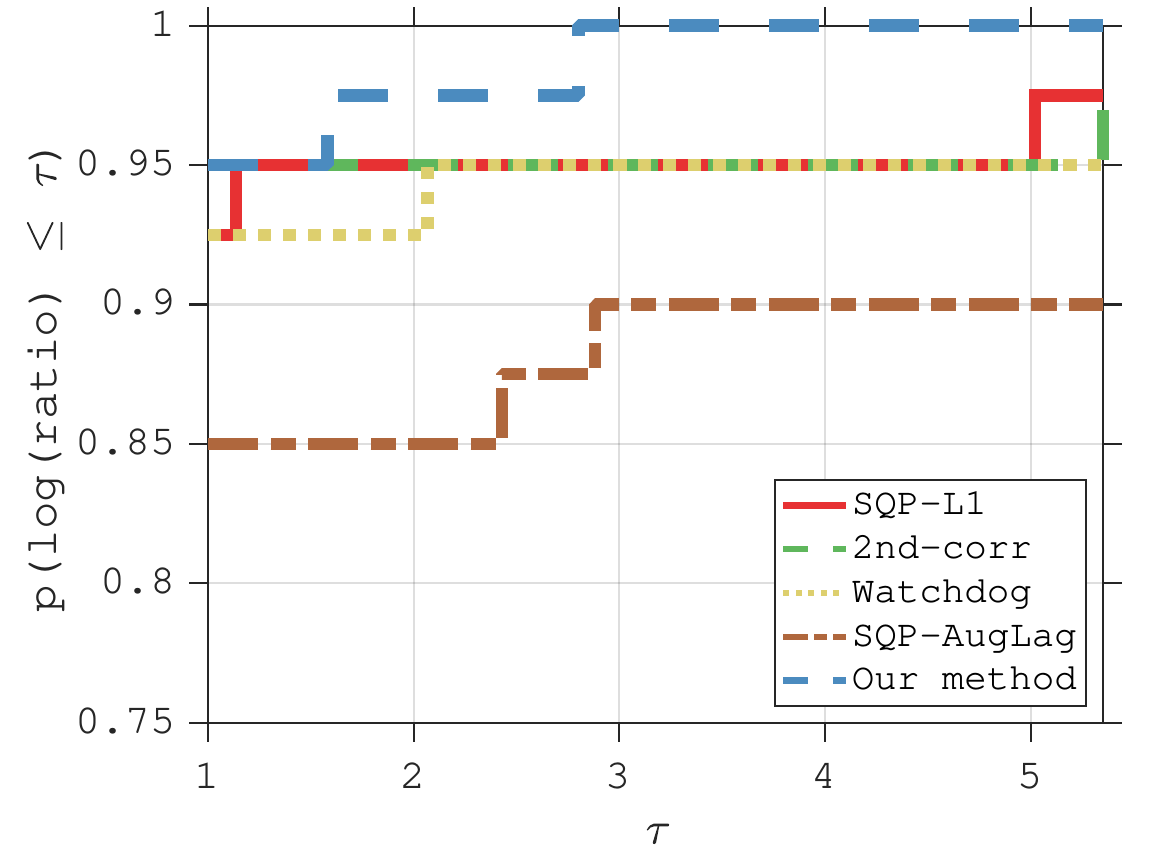}
    \caption{Dolan-Mor\'e  performance profiles \cite{dolan2002benchmarking} comparing methods on CUTEst collection of test problems in terms of iterations (left) and function evaluations (right).
    } \label{fig.DM_v6}
\end{figure}

\begin{table}[H]
\caption{Percentage of iterations for which modified line search condition is employed.}
  \centering
  {\footnotesize
\begin{tabular}
{cccccccc}\toprule
\textbf{Problem} & hs027    & biggs3  & hs006 & hs100lnp & bt7 & hs046 & orthregb \\ \midrule
\textbf{Percentage} &  $8\%$       & $23\%$     & $33\%$ & $18\%$ & $14\%$ & $16\%$ & $71\%$   \\      
\bottomrule                                                 \end{tabular}}
\label{table.modified.percentage}
\end{table}

\begin{figure}
    \centering \includegraphics[width=0.45\textwidth,height=0.3\linewidth,clip=true,trim=10 180 50 200]{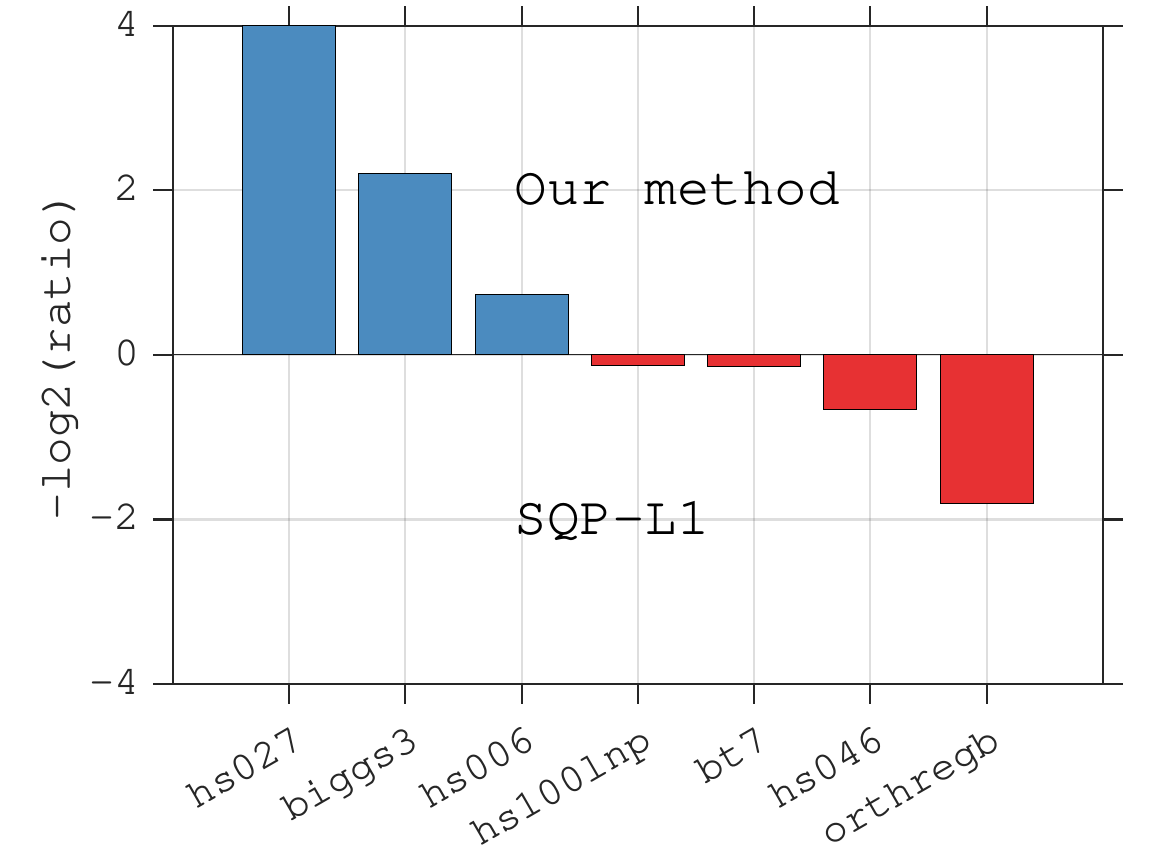}
    \includegraphics[width=0.45\textwidth,height=0.3\linewidth,clip=true,trim=10 180 50 200]{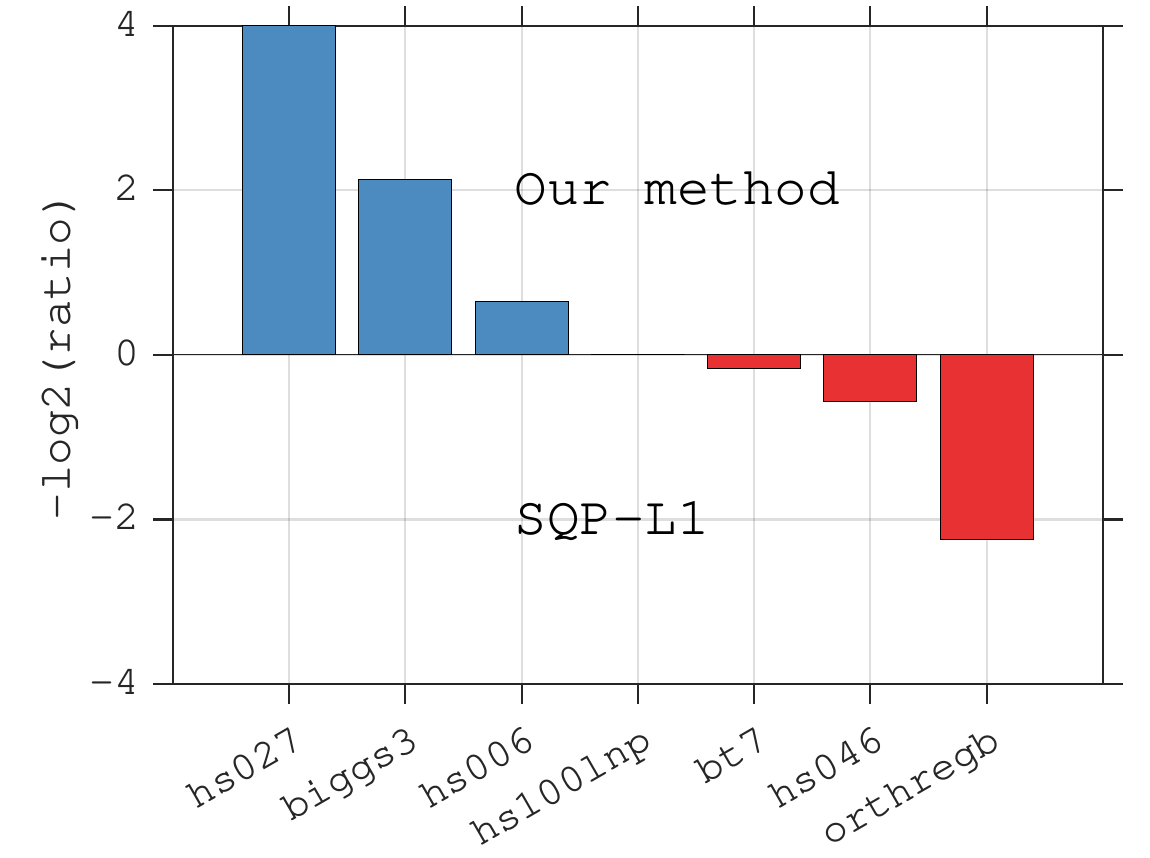}
    \caption{Morales performance profiles \cite{morales2002numerical} comparing the classical line search SQP method (\textbf{SQP-L1}) and our proposed method on seven test problems in terms of iterations (left) and function evaluations (right).} \label{fig.Mo}
\end{figure}

\subsection{Stochastic Setting}
\label{sec.numerical.sto}
In this section, we consider the following constrained binary classification problem,
\begin{align}
  &\min_{x\in\mathbb{R}^n}\ f(x) = \frac{1}{N} \sum_{i=1}^N \log \left( 1 + e^{-y_i(X_ix)}\right)\ \text{ s.t. }\ A_1 x=a_1, \quad x^T A_2 x = a_2, \label{eq.log_lin}
\end{align}
where $X \in \mathbb{R}^{N \times n}$ is the data matrix (feature data for $N$ data points), $X_i \in \mathbb{R}^{1 \times n}$ is the $i$th row of $X$, $y \in \{-1,1\}^N$ are the labels (for each data point), and $A_1 \in \mathbb{R}^{m \times n}$ (full row rank), $a_1 \in \mathbb{R}^m$, $A_2 \in \mathbb{R}^{n \times n}$ (positive definite) and $a_2 \in \mathbb{R}$. In these experiments, we consider datasets from the LIBSVM collection \cite{chang2011libsvm}. For brevity, we present results only on two datasets: \texttt{australian}: $n = 14$, $N = 621$; \texttt{mushroom}: $n = 112$, $N = 5500$. 
With regards to the constraints, $m=5$, the matrix and vector $A_1$ and $a_1$ were generated from a normal distribution, the matrix $A_2$ was a positive definite matrix with eigenvalues evenly distributed between 1 and 10, and $a_2 =5$.

In this set of experiments, the methods have access to exact function and gradient information of the objective and constraint functions, but inexact Hessian information of the objective function. The goal of these experiments is to investigate the effect of the Hessian approximation, the step size scheme and search direction (linear system solution) quality. We consider two batch sizes, $b=5\% N$ (small batch) and $b=50\% N$ (large batch) for the Hessian, and also the exact Hessian setting ($b=100\% N$, full batch). We compare the adaptive step size method proposed in \cite{berahas2021sequential} and Algorithm~\ref{alg.SubsampledSQP.practical}. For each step size scheme we consider first-order ($W_k = I$) and second-order ($W_k$ approximation of the Hessian of the Lagrangian) variants. We note that the first-order variants of the methods are essentially deterministic algorithms. A budget of 50 iterations (50$N$ gradient evaluations), 100$N$ function evaluations, and 50$N$ Hessian evaluations was used for all methods. 
We use the same parameter values for Algorithm~\ref{alg.SubsampledSQP.practical} as described in Section~\ref{sec.numerical.det}. For the adaptive step size method, we use the default values given in \cite[Section 4.2]{berahas2021sequential}. We should note that this adaptive step size scheme requires at least an estimate of the Lipschitz constants of the objective and constraint gradients. 
For all problems and algorithms, the initial primal iterate ($x_0$) was set to a normal random vector scaled to have norm $0.1$, and the multipliers were initialized as $y_0 = \arg\min_{y\in\mathbb{R}^m}\ \|g_0 + J_0^Ty\|_2^2$. Finally, for the small and large batch instances, we ran 10 replications with different random seeds for each problem, dataset, and algorithm, and report the average performance.

Figure~\ref{fig.sto.first.second} presents the evolution of the feasibility and stationarity errors and the step size with respect to iterations and number of Hessian evaluations. Overall, the second-order variants  of the methods (even small batch) converge faster than the first-order variants in terms of iterations. In the small batch setting ($5\%$), Algorithm~\ref{alg.SubsampledSQP.practical} yields smaller step sizes as compared to the adaptive step size scheme counterpart. This observation corroborates Lemma \ref{lemma.dist.small.imply.alpha.1}, which posits that a small Hessian approximation error is requisite. When more accurate Hessian approximations are employed, our results suggest that in terms of iterations the line search variant (Algorithm~\ref{alg.SubsampledSQP.practical}) outperforms the adaptive step size variant. This, of course, comes at the cost of function evaluations that are not required by the adaptive step size method. Finally, as expected, the small batch variants of the two methods outperform the large and full batch variants when the metric of comparison is the number of Hessian evaluations.

\begin{figure}
    \centering \includegraphics[width=0.19\textwidth,clip=true,trim=10 180 50 180]{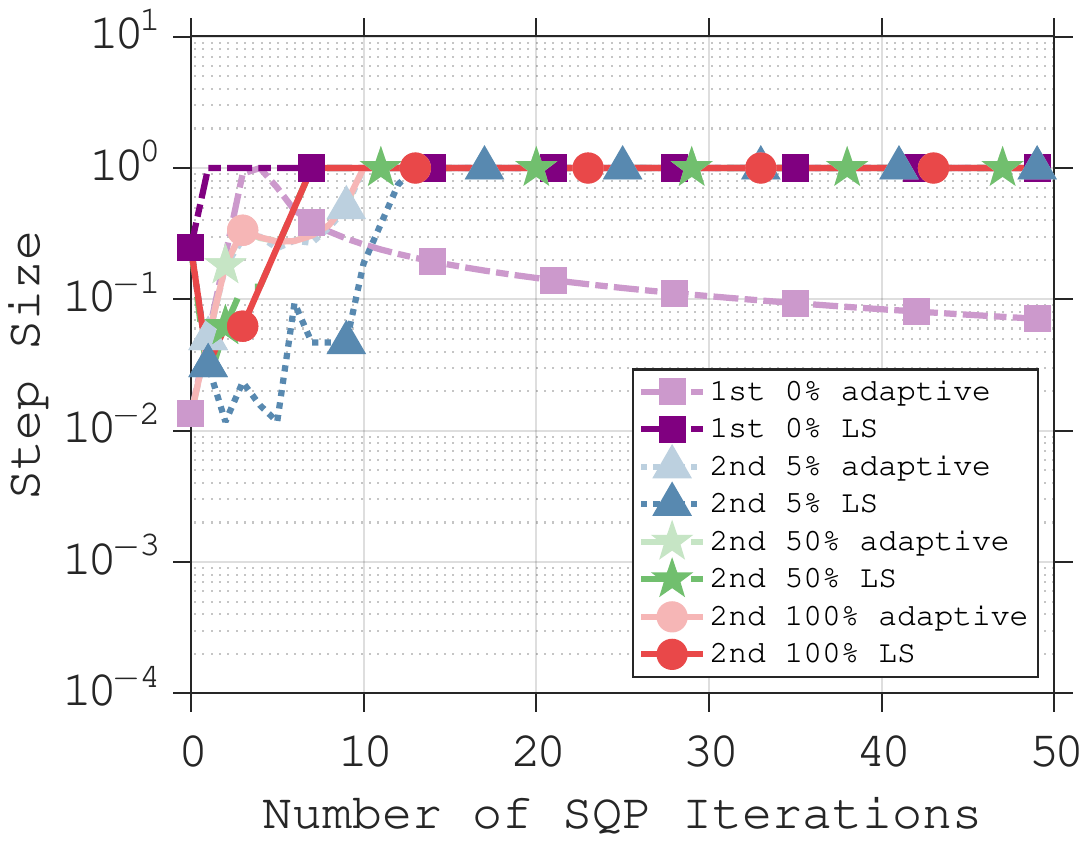}
    \includegraphics[width=0.19\textwidth,clip=true,trim=10 180 50 180]{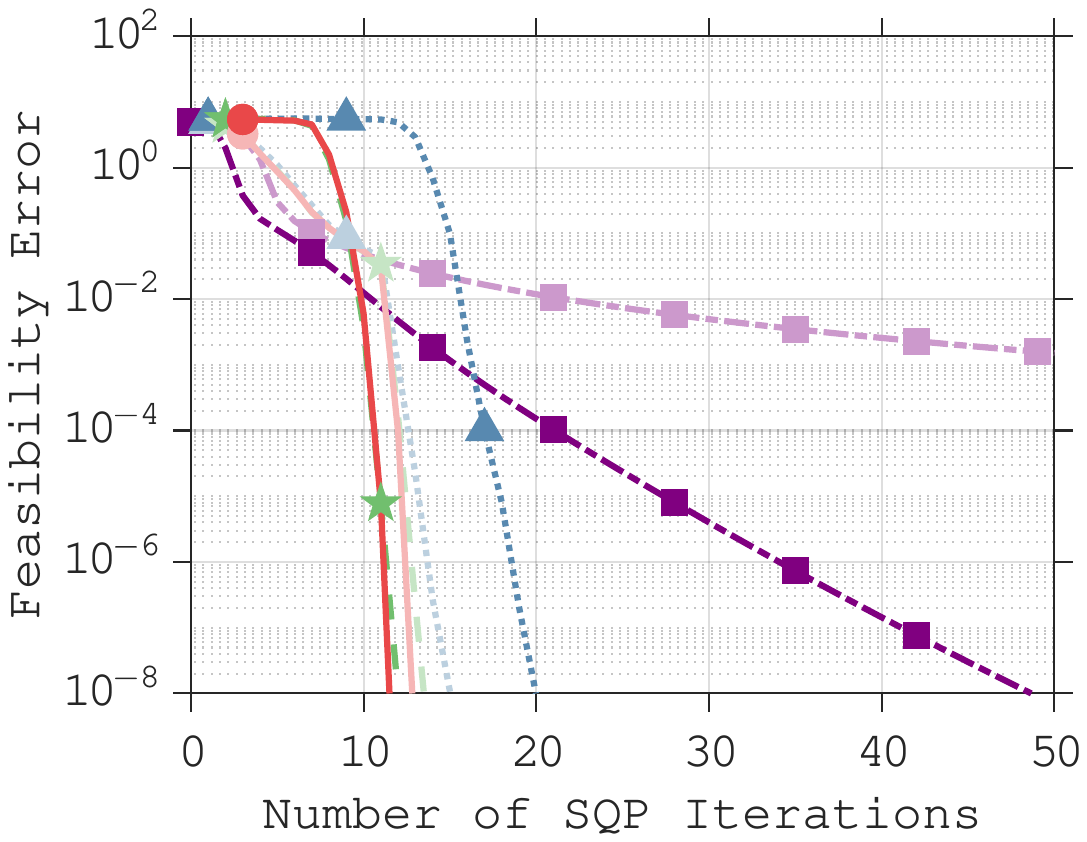}  \includegraphics[width=0.19\textwidth,clip=true,trim=10 180 50 180]{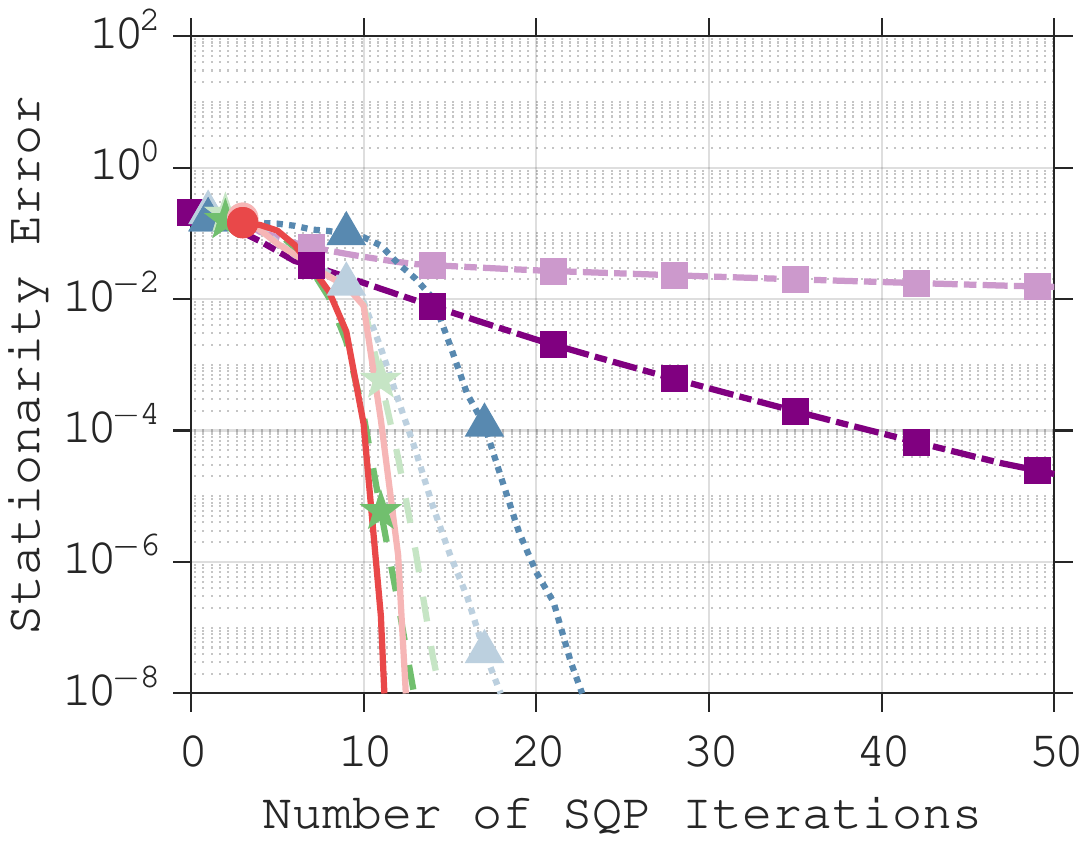}
\includegraphics[width=0.19\textwidth,clip=true,trim=10 180 50 180]{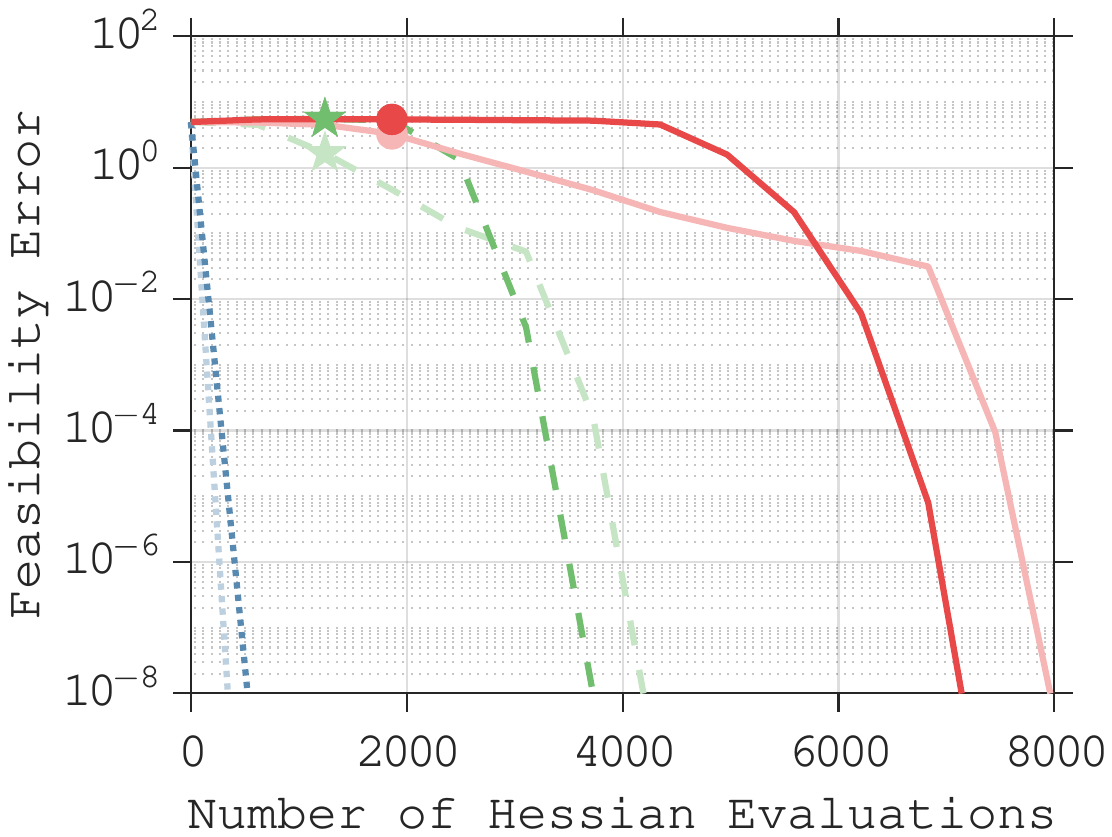} 
   \includegraphics[width=0.19\textwidth,clip=true,trim=10 180 50 180]{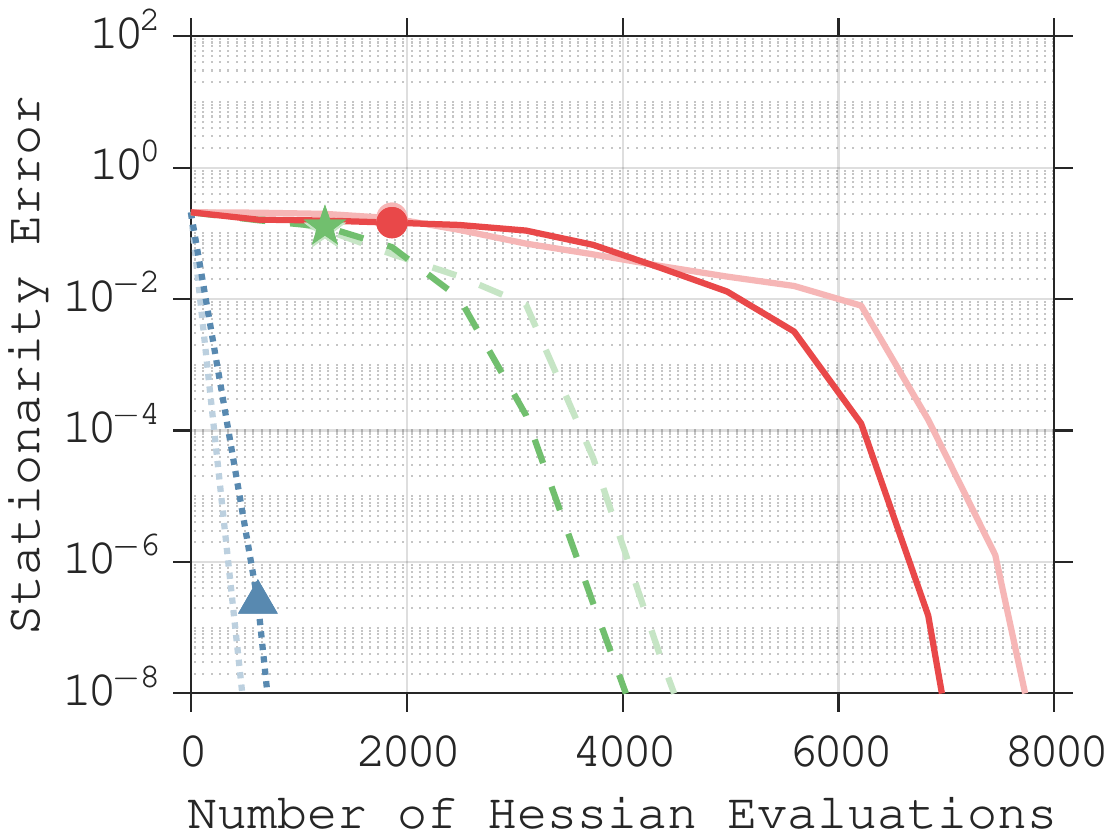} 
   \includegraphics[width=0.19\textwidth,clip=true,trim=10 180 50 180]{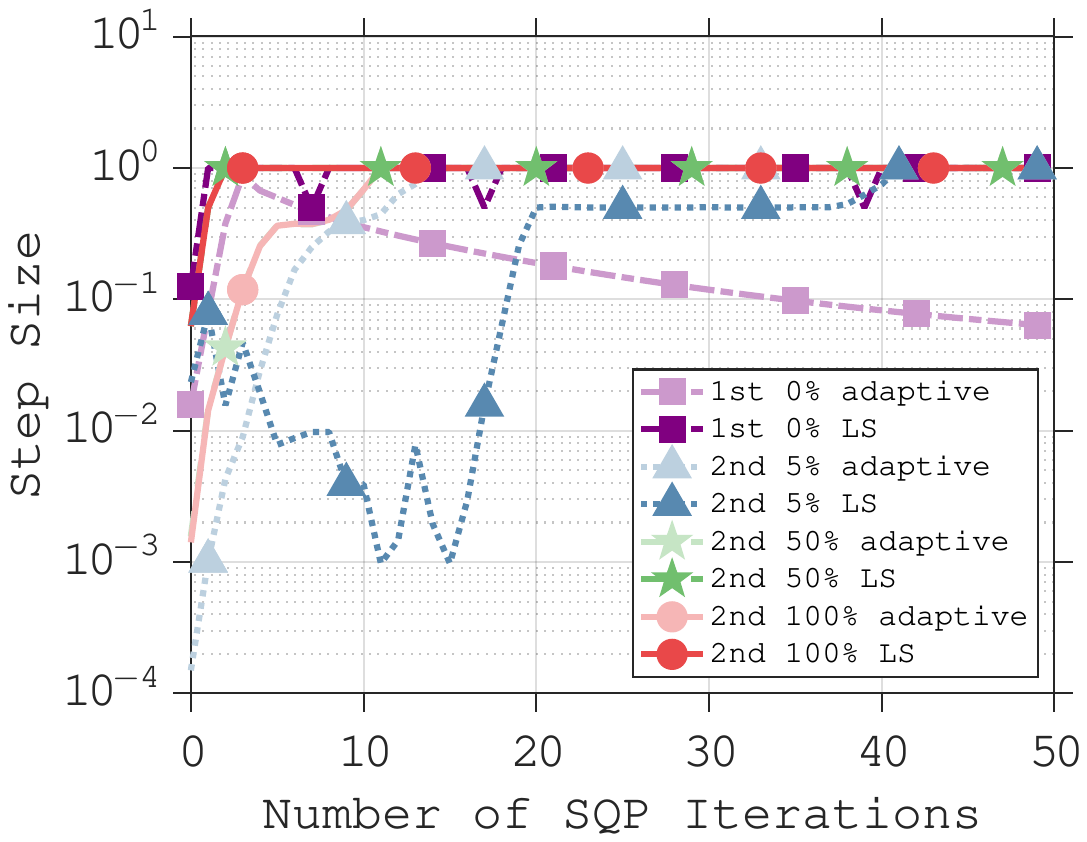}
    \includegraphics[width=0.19\textwidth,clip=true,trim=10 180 50 180]{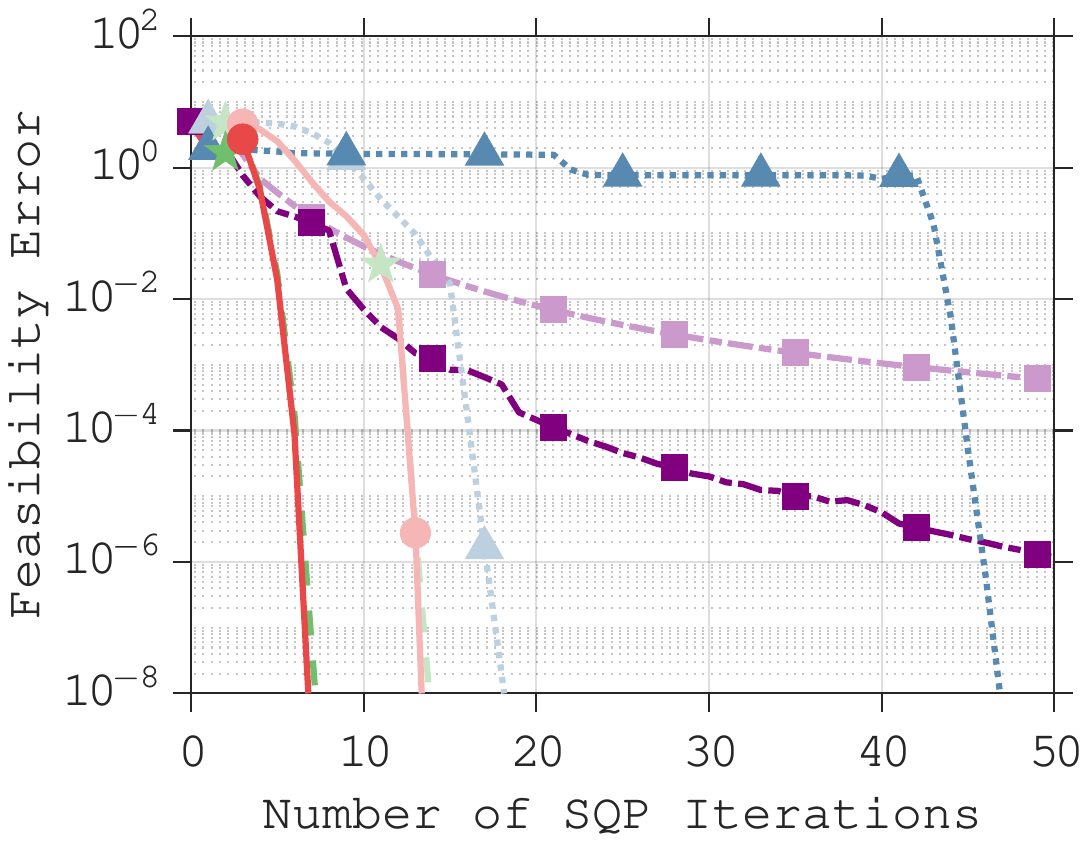}  \includegraphics[width=0.19\textwidth,clip=true,trim=10 180 50 180]{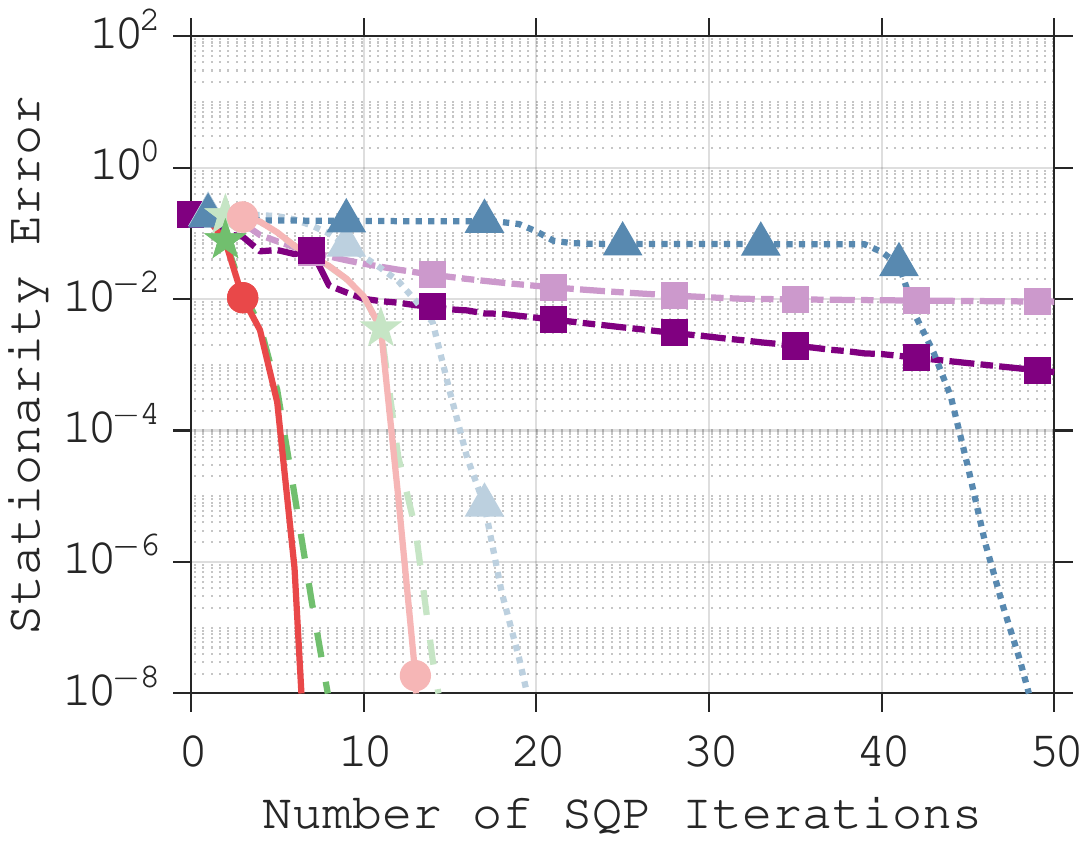}
\includegraphics[width=0.19\textwidth,clip=true,trim=10 180 50 180]{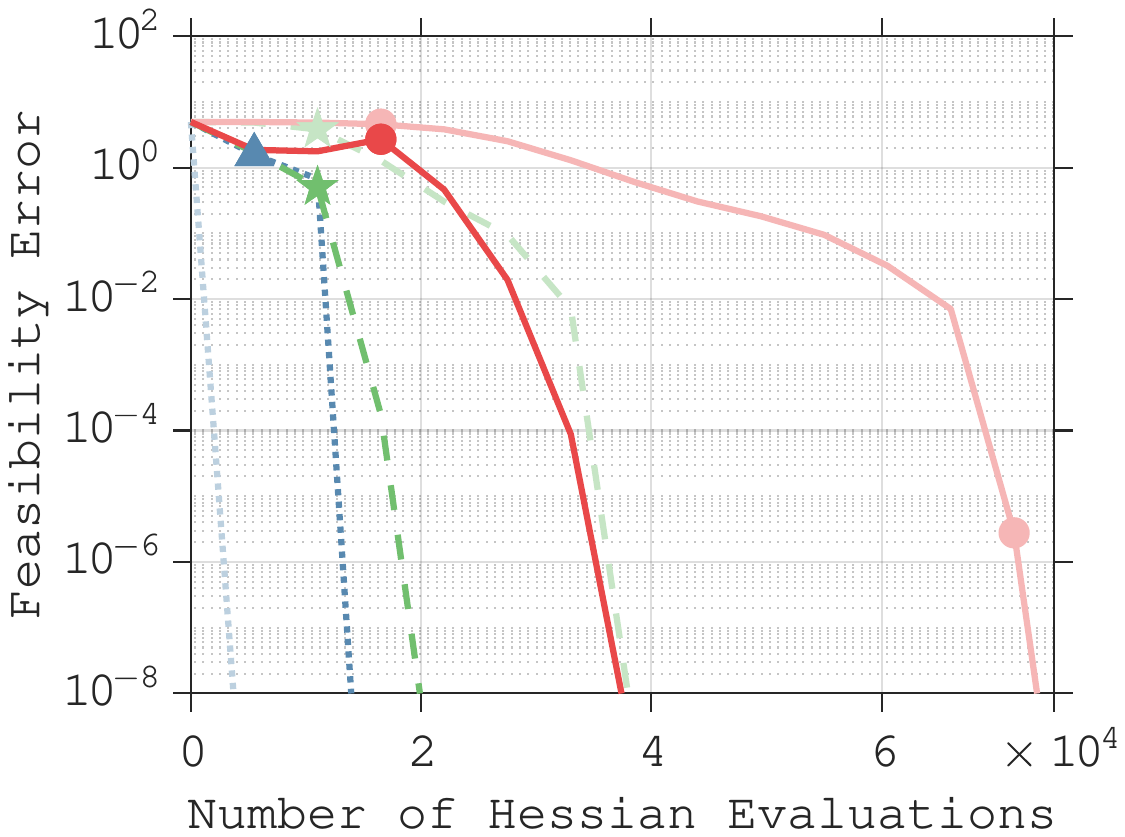} 
   \includegraphics[width=0.19\textwidth,clip=true,trim=10 180 50 180]{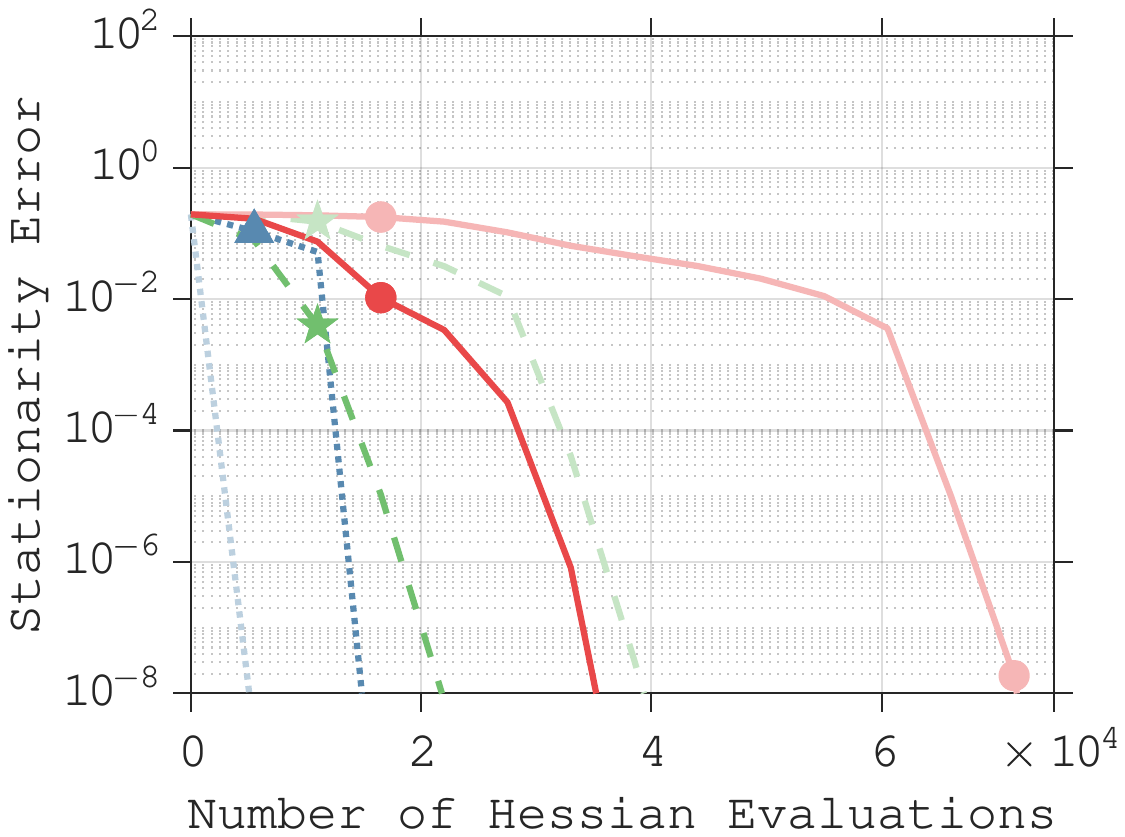} 
        \caption{Comparison of first-order (\textbf{1st}) and second-order (\textbf{2nd}) SQP methods with adaptive step size (\textbf{adaptive}) and line search (\textbf{LS}) schemes on the \texttt{australian} (top) and \texttt{mushroom} (below) datasets~\cite{chang2011libsvm}. True function and gradient information and inexact Hessian information (batch sizes \textbf{5\%},   \textbf{50\%}, \textbf{100\%} of total sample size).
    } \label{fig.sto.first.second}
\end{figure}

Next, we compare exact and inexact (
solution of the linear system \eqref{eq.SQP.det}) variants of Algorithm~\ref{alg.SubsampledSQP.practical} in Figure~\ref{fig.inexact_paper}. In these experiments, we utilize the MINRES method \cite{choi2011minres} to solve the linear system \eqref{eq.SQP.det} to a prescribed accuracy. Specifically, the MINRES method is terminated when the residual vectors $\rho_k^{(t)}$ and $r_k^{(t)}$ satisfy
 \begin{equation}
    \left\|\left[\begin{array}{c}
\rho_k^{(t)} \\
r_k^{(t)}
\end{array}\right]\right\|_{\infty} \leq \max \left\{\kappa\left\|\left[\begin{array}{c}
 {g}_k + J_k^T  y_k  \\
c_k
\end{array}\right]\right\|_{\infty}, 10^{-12}\right\}
\end{equation}
 where $\kappa = 10^{-12}$ (exact) and $\kappa = 10^{-1}$ (inexact) is the accuracy parameter. 
 In addition to the budgets given earlier in this subsection, a budget of $(m+n)N$ total MINRES iterations was imposed. 
 We compare exact and inexact variants of Algorithm \ref{alg.SubsampledSQP.practical} with $b=5\% N$ (small batch), $b=50\% N$ (large batch) and $b=100\% N$ (full batch), and 
 an adaptive Hessian batch size scheme, 
$b_k = \min \left\{ \lfloor (1- 0.95^{(k+2)/2}) N \rfloor , N\right\}$, which guarantees that the initial Hessian estimate uses $5\%$ of the total samples. 
In Figure \ref{fig.inexact_paper}, we observe that while exact methods generally yield more precise solutions in terms of iterations, our proposed inexact approach is competitive and requires fewer MINRES iterations to achieve convergence. Moreover, the adaptive methods achieve a good balance across all metrics making them a viable option in practice.

\begin{figure}
    \centering \includegraphics[width=0.19\textwidth,clip=true,trim=10 180 50 180]{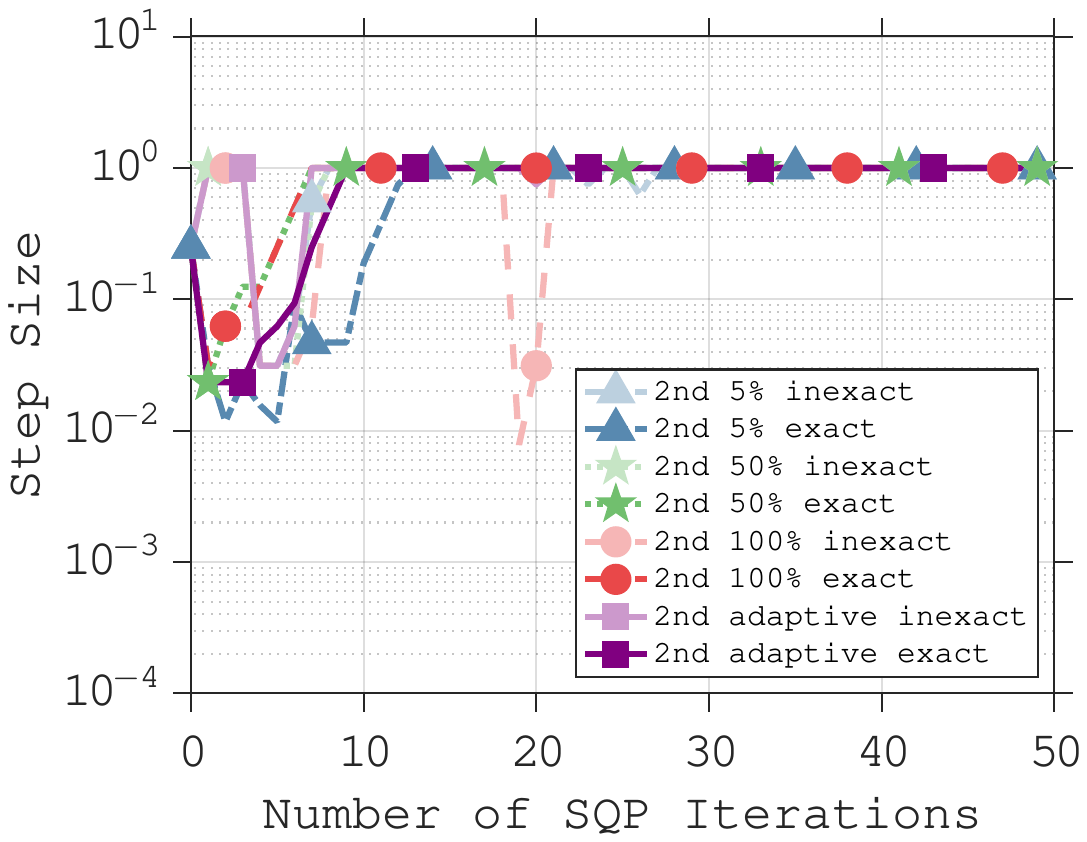}
    \includegraphics[width=0.19\textwidth,clip=true,trim=10 180 50 180]{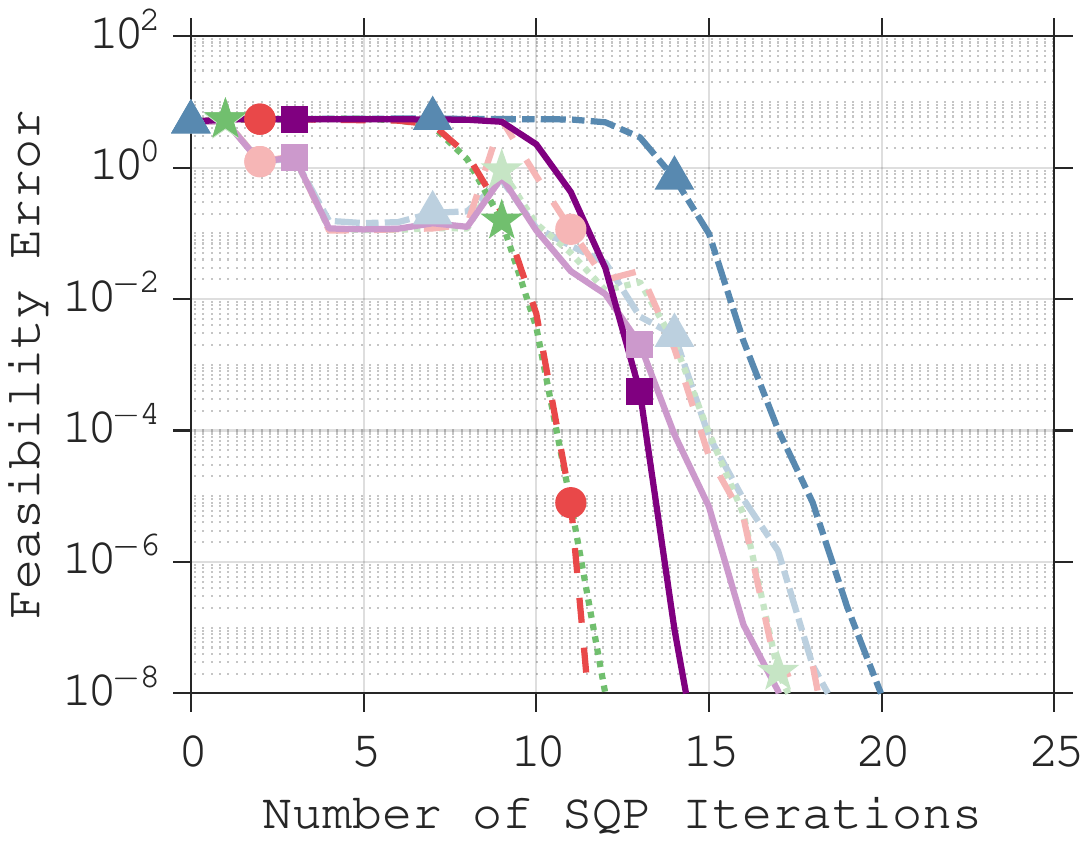}
    \includegraphics[width=0.19\textwidth,clip=true,trim=10 180 50 180]{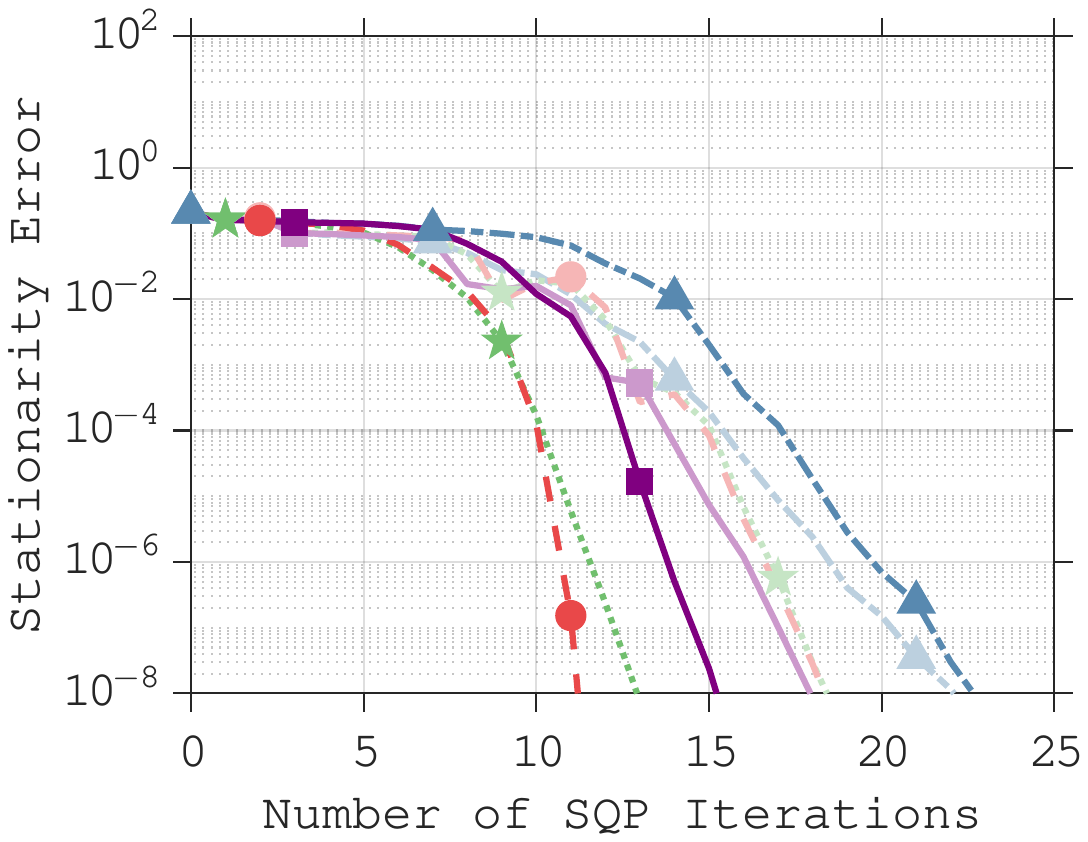}
       \includegraphics[width=0.19\textwidth,clip=true,trim=10 180 50 180]{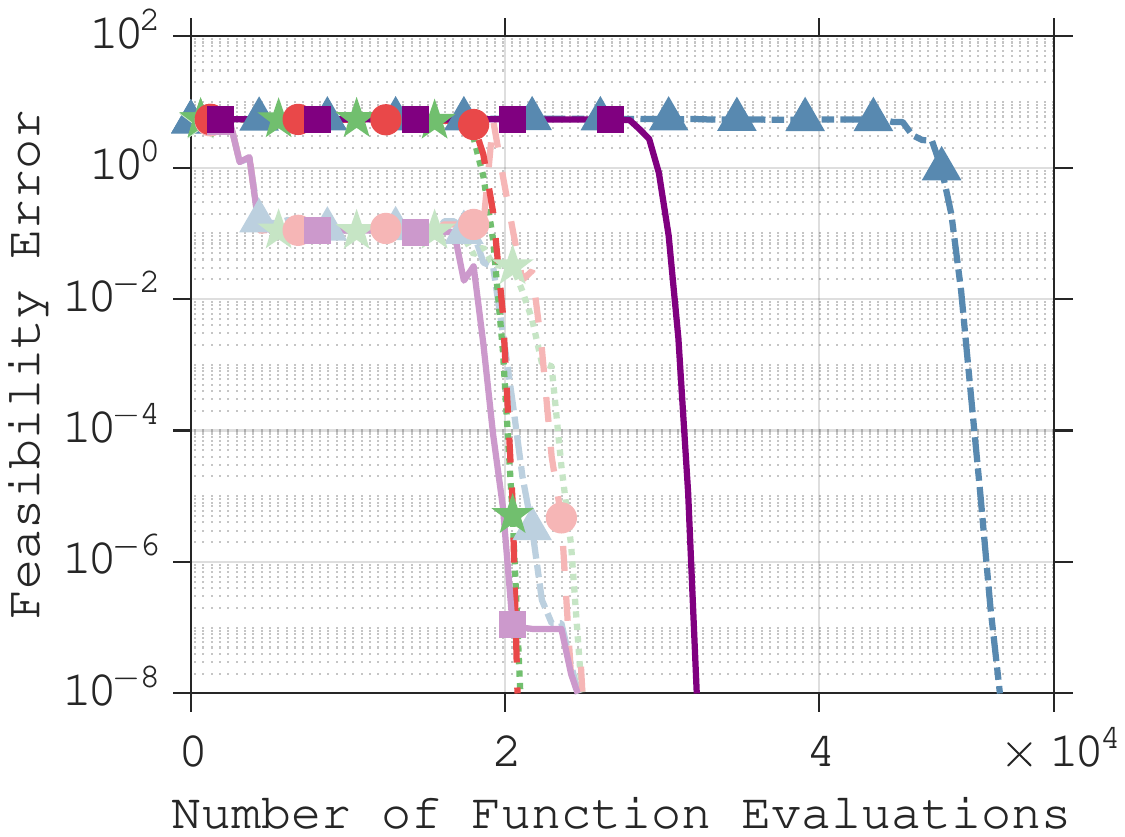}    \includegraphics[width=0.19\textwidth,clip=true,trim=10 180 50 180]{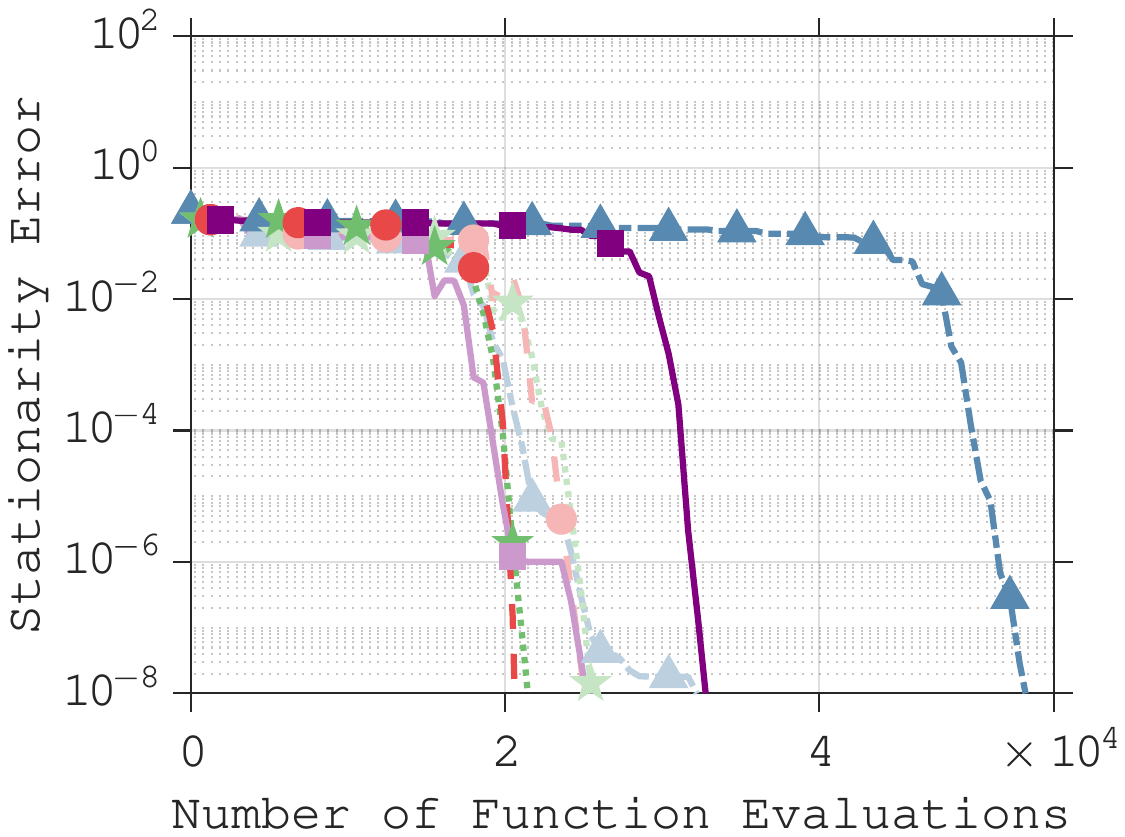}
        \includegraphics[width=0.19\textwidth,clip=true,trim=10 180 50 180]{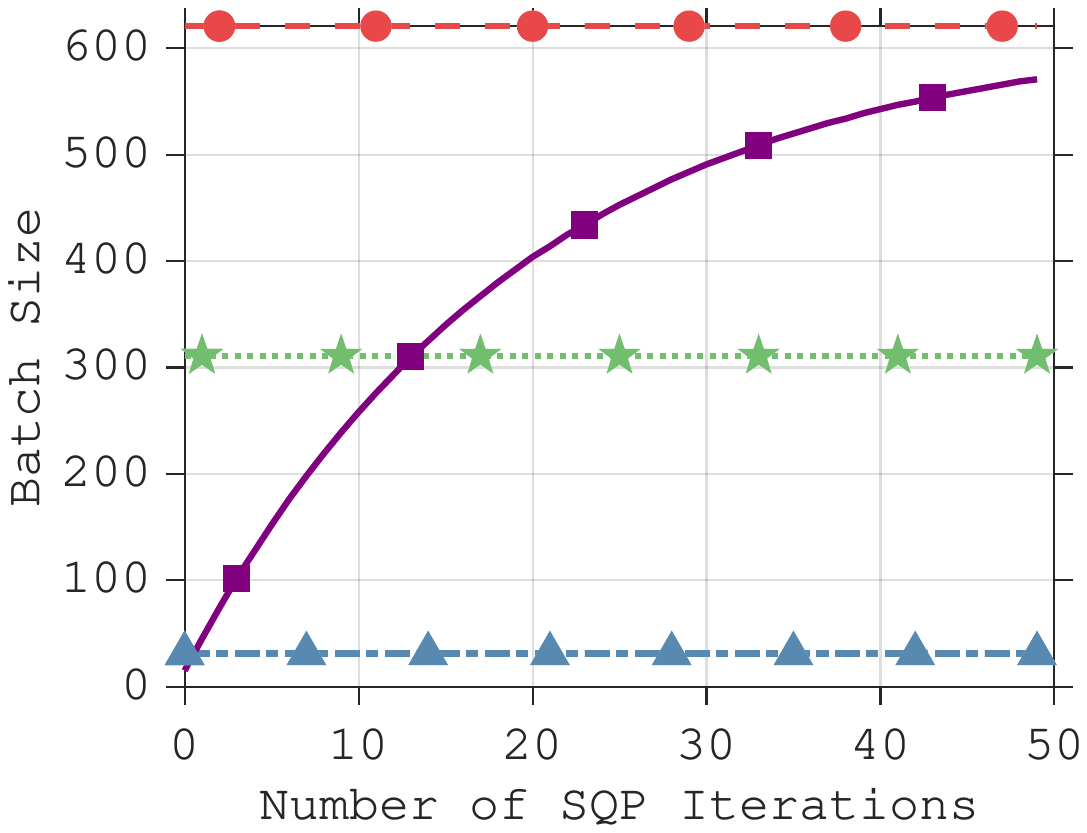} 
    \includegraphics[width=0.19\textwidth,clip=true,trim=10 180 50 180]{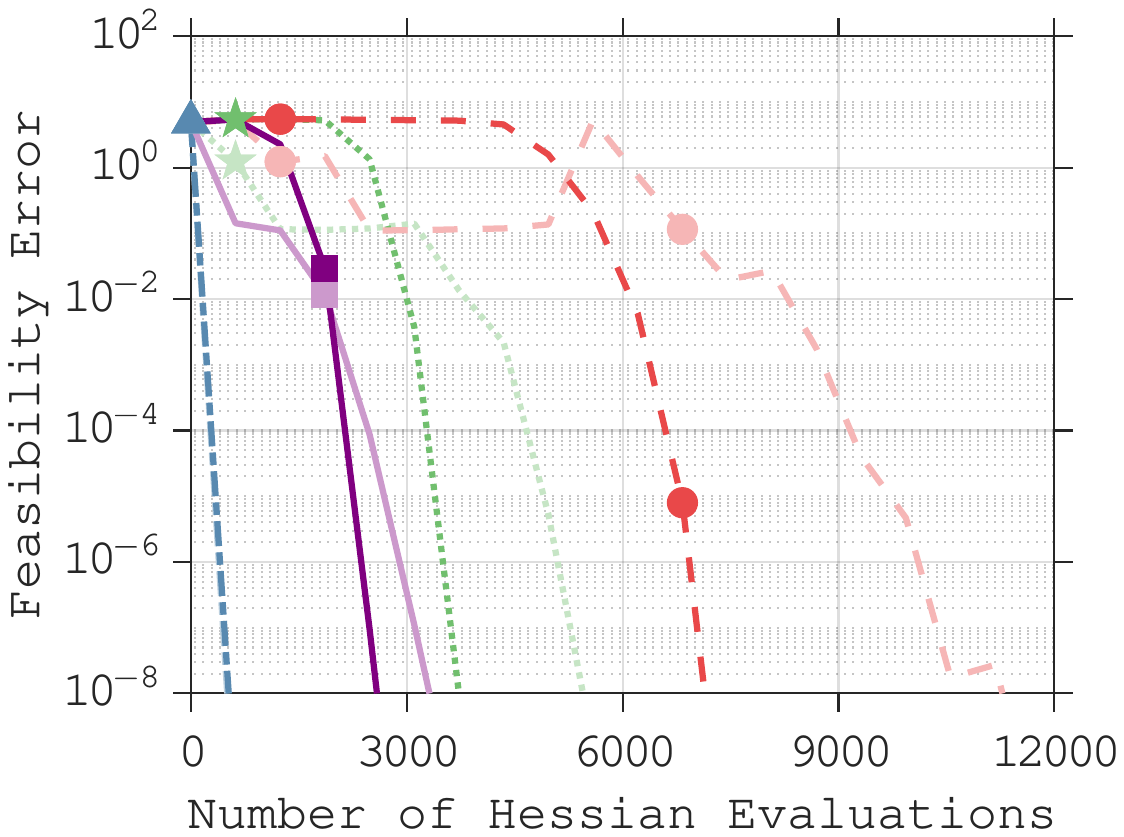}
    \includegraphics[width=0.19\textwidth,clip=true,trim=10 180 50 180]{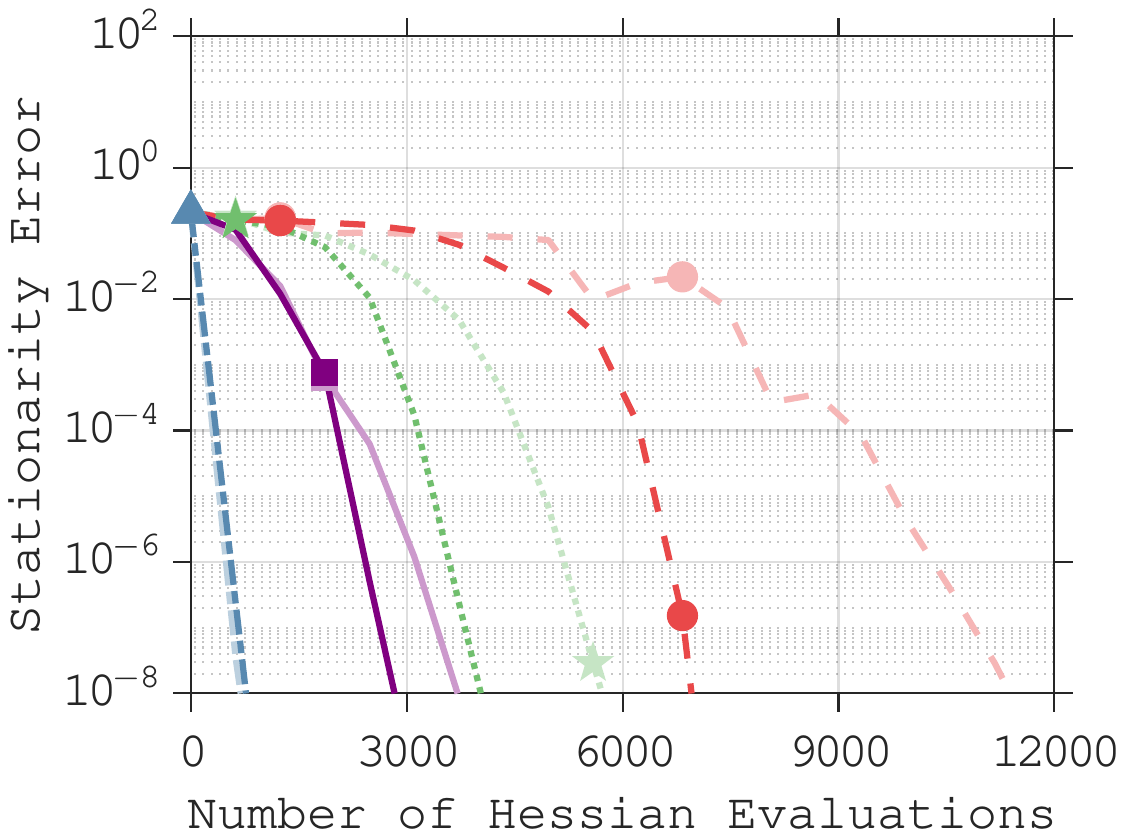}   
   \includegraphics[width=0.19\textwidth,clip=true,trim=10 180 50 180]{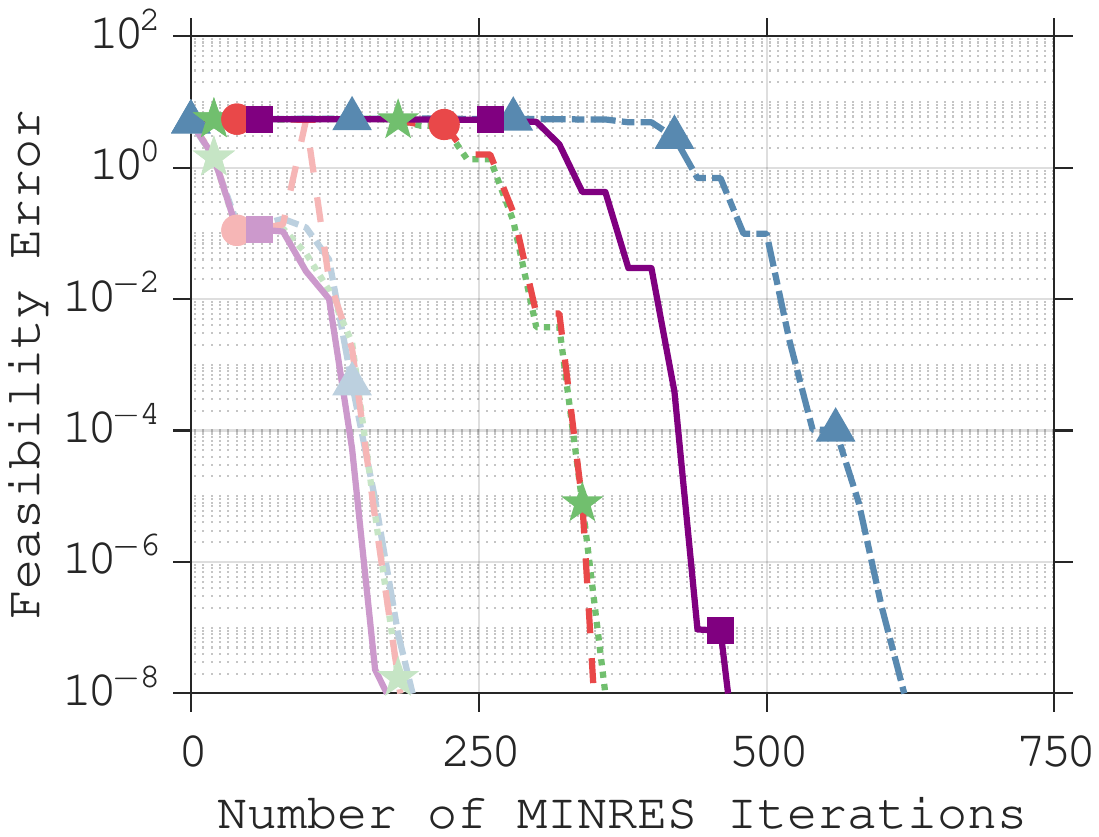}
    \includegraphics[width=0.19\textwidth,clip=true,trim=10 180 50 180]{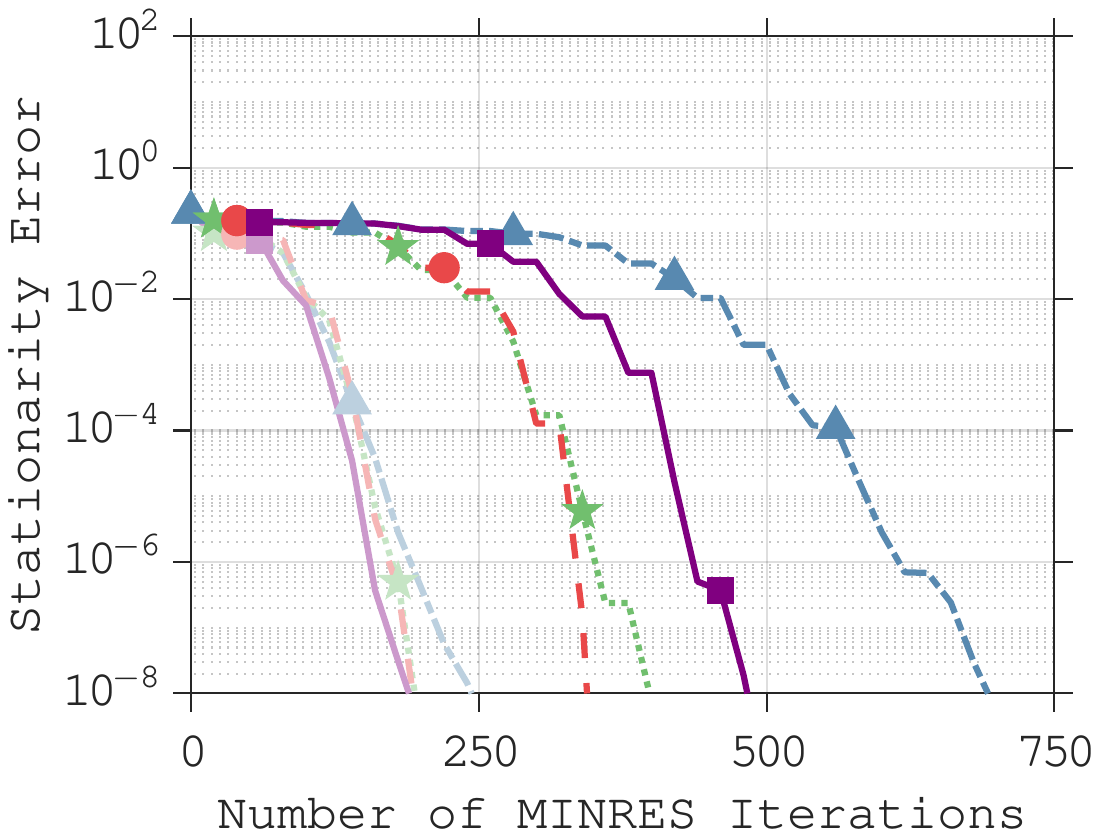}      
        \caption{
        Comparison of \textbf{exact} and \textbf{inexact} second-order (\textbf{2nd}) line search SQP methods with adaptive step size (\textbf{adaptive}) and line search (\textbf{LS}) schemes on the \texttt{australian} dataset~\cite{chang2011libsvm}. True function and gradient information and inexact Hessian information (batch sizes \textbf{5\%},   \textbf{50\%}, \textbf{100\%} of total sample size, and \textbf{adaptive}).} \label{fig.inexact_paper}
\end{figure}

\section{Final Remarks}
\label{sec.remark}

This paper introduces a novel second-order line search sequential quadratic programming (SQP) method designed for solving general nonlinear equality constrained optimization problems in both deterministic and stochastic settings. Our proposed methodology not only capitalizes on the global convergence properties of classical line search SQP methods, but also alleviates the adverse Maratos effect by selectively employing a carefully designed modified line search, and 
is endowed with local superlinear convergence guarantees without  the need for a priori knowledge of problem-specific constants. For large-scale problems, we have developed an inexact matrix-free variant, which integrates the MINRES method, offering a compromise between efficiency and accuracy. Numerical experiments conducted on the CUTEst test set and logistic regression problems demonstrate the efficacy, efficiency and robustness of the proposed method. Finally, we highlight the adaptivity of our algorithm that makes our approach not only theoretically appealing but also  practical. 



\bibliographystyle{plain}
\bibliography{reference}

\appendix

\section{Appendix: Proof of Lemma~\ref{lemma.diff.g.g.bar}}\label{app.lemma_2}
\begin{proof}  
It follows by Assumptions~\ref{ass.function} and \ref{ass.individual.f.subsampled}, the triangular inequality that 
\begin{align*}
\left\|g_k -\bar g_k\right\| & =\left\|\left(\tfrac{N-\left|S_{k}^g\right|}{N\left|S_{k}^g\right|}\right) \sum_{i \in S_{k}^g} \nabla f_i\left(x_k\right)-\tfrac{1}{N} \sum_{i \in [N]\setminus S_{k}^g} \nabla f_i\left(x_k\right)\right\| \\
& \leq\left(\tfrac{N-\left|S_{k}^g\right|}{N\left|S_{k}^g\right|}\right) \sum_{i \in S_{k}^g}\left\|\nabla f_i\left(x_k\right)\right\|+\tfrac{1}{N} \sum_{i \in [N]\setminus S_{k}^g}\left\|\nabla f_i\left(x_k\right)\right\| \\
& \leq 2 \left(\tfrac{N-\left|S_{k}^g\right|}{N}\right)(\kappa_1 + \beta_1 \| \nabla f(x) \|) \le  2 \left(\tfrac{N-\left|S_{k}^g\right|}{N}\right)(\kappa_1 + \beta_1 \kappa_g ).
\end{align*}

Note that the inequalities still hold for any types of norm. Similarly, 
\begin{align*}
    \left\|H_k -\bar H_k\right\|  &\leq 2 \left(\tfrac{N-\left|S_{k}^H\right|}{N}\right)(\kappa_2 + \beta_2 \| \nabla^2 f(x) \|)  \le 2 \left(\tfrac{N-\left|S_{k}^H\right|}{N}\right)(\kappa_2 + \beta_2 \kappa_H). \\ |f_k -\bar f_k|  &\leq 2 \left(\tfrac{N-\left|S_{k}^f\right|}{N}\right)(\kappa_0 + \beta_0 | f(x) |)  \le 2 \left(\tfrac{N-\left|S_{k}^f\right|}{N}\right)(\kappa_0 + \beta_0\kappa_f). 
\end{align*} 
\end{proof}

\section{Appendix: Proof of Theorem \ref{theorem.quadratic.det}}
\label{appendix.quadratic.theorem}

\begin{proof}
Replacing stochastic quantities $\bar d_k, \bar \delta_k$ with deterministic quantities $d_k, \delta_k$, it follows by \eqref{eq.delta.ub.d}  that $\| \delta_k \| \le 2\kappa_{J^{\dagger}} \kappa_W \| d_k \|$. 
When $\| d_k\| \le \kappa_d$, it then follows by assumptions \ref{ass.d.bound.imply.difference.det} that 
\begin{align}
    \|w_k - w^*\|& \le  2\mu_M^{-1} \left\|\begin{bmatrix}
     g_k + J_k^T y_k \\ c_k
\end{bmatrix}\right\| \le 2\mu_M^{-1} \left\|  M_k  M_k^{-1}\begin{bmatrix}
  g_k + J_k^T y_k \\ c_k
\end{bmatrix}  \right\| \notag \\ & \le  2\mu_M^{-1}  \kappa_M \left\|\begin{bmatrix}
 d_k  \\  \delta_k
\end{bmatrix}\right\|   \le  2\mu_M^{-1}  \kappa_M (1+   2\kappa_{J^{\dagger}} \kappa_W ) \|  d_k\| \label{eq.relation.w_k.d_k.det}.
\end{align}
If $\|d_0\| \le \tfrac{1}{ 2\mu_M^{-1}  \kappa_M (1+   2\kappa_{J^{\dagger}} \kappa_W ) \Lambda L_W }$, we have $\|w_0 - w^* \| \le \tfrac{1}{\Lambda L_W}$ by \eqref{eq.relation.w_k.d_k.det}.
It then follows by Cauchy-Schwartz inequality and assumption \ref{ass.M.Lipschitz.det1} that when $r = \tfrac{1}{\Lambda L_W}$, 
\begin{align}
\label{eq.w_k+1.det}
  &\  \| w_{k+1} - w^* \| 
=\left\| w_{k} - w^* -  M_k^{-1} \begin{bmatrix}
    g_k  + J_k^T y_k \\  c_k
\end{bmatrix} \right\|  \notag \\ 
  \le &\  \left\| M_k^{-1} \right\| \left\|  M_k( w_{k} - w^*) - \begin{bmatrix}
   g_k  + J_k^T  y_k \\  c_k
\end{bmatrix} \right\| \le \tfrac{\Lambda L_W}{2}   \| w_{k} - w^* \|^2  .  
\end{align}

We then use induction to prove $\|w_k - w^* \| \le \tfrac{1}{\Lambda L_W}$. The base case holds trivially, suppose that the statement holds for iteration $k$, then it follows from \eqref{eq.w_k+1.det} that 
\begin{equation} 
    \|  w_{k+1} - w^*\| \le  \tfrac{\Lambda L_W}{2}   \| w_{k} - w^* \|^2 \le  \tfrac{\Lambda L_W}{2} \| w_{k} - w^* \|  \| w_{k} - w^* \| \le \tfrac12  \| w_{k} - w^* \|.
\end{equation}
Hence, $\| w_{k} - w^* \| \to 0$ at a Q-quadratic rate.
\end{proof}

\section{Appendix: Proof of Theorem  \ref{theorem.superlinear.convergence.deterministic}} \label{app.superlinear.theorem}
\begin{proof}
 We use induction to show that for all $k$,
 \begin{align}
 \left\|w_k-w^*\right\|\leq \tfrac{1}{3 \Lambda L_W} \xi_k.
     \label{eq.desired.distance.induction}
 \end{align}
 By condition $(\romannumeral1)$, the initial gradient approximation satisfies 
$\|g_0 - \bar g_0\| \le  \min \left\{ \tfrac{\mu_M }{12 \Lambda L_W}, \tfrac{\kappa_d}{2\Lambda} \right\}$. The base case, $k=0$, is therefore satisfied.
Moreover, by  Lemma \ref{lemma.d.bound.imply.difference}, it follows that $\left\|w_0-w^*\right\| \le \tfrac{1}{3 \Lambda L_W}$.
Define sequences $\{ \xi_k \}_{k=0}^{\infty}$ and $\{ \rho_k \}_{k=0}^{\infty}$ as follows:
 \begin{align}
  \label{def.rho.tau}
    \xi_{k+1}=\max \left\{\xi_k \rho_k, \beta_{k+1}^{(k+1) / 4}\right\}, \quad \xi_0=1, \quad \rho_k=\tfrac{\xi_k}{6}+\tfrac{1}{4} \tfrac{N-|S_k^H|}{N-|S_0^H|}+\tfrac{\beta_{k}^{k/4}}{2 } .
 \end{align}
Assume 
\eqref{eq.desired.distance.induction} is true for iteration $k$, and consider iteration $k+1$,
\begin{align*}
    &\ \left\|w_{k+1}-w^*\right\| 
\\ \leq  &\  \tfrac{\Lambda L_W }{2}\left\|w_k-w^*\right\|^2 + \tfrac{2(\kappa_2 + \mu_{2} \kappa_H ) \Lambda  (N-|S_k^H|)}{N}  \left\|w_k-w^*\right\| + \tfrac{ 2(\kappa_1 + \mu_{1} \kappa_g ) \beta_{k}^{k/2}}{12 \Lambda L_W (\kappa_1 + \mu_{1} \kappa_g )} \\
 \leq&\ \tfrac{\xi_k}{3 \Lambda L_W} \left(\tfrac{\xi_k}{6}\right)+\tfrac{\xi_k}{3 \Lambda L_W} \left(\tfrac{1}{4} \tfrac{N-|S_k^H|}{N-|S_0^H|}\right)+\tfrac{1}{3 \Lambda L_W}\left(\tfrac{\beta_{k}^{k/2}}{2}\right) \\
 =&\ \tfrac{\xi_k}{3 \Lambda L_W} \left( \tfrac{\xi_k}{6}+\tfrac{1}{4} \tfrac{N-|S_k^H|}{N-|S_0^H|}+\tfrac{\beta_{k}^{k/2}}{2 \xi_k}\right) \\
 \leq&\ \tfrac{\xi_k}{3 \Lambda L_W}\left( \tfrac{\xi_k}{6}+\tfrac{1}{4} \tfrac{N-|S_k^H|}{N-|S_0^H|}+\tfrac{\beta_{k}^{k/4}}{2 }\right) =\tfrac{\xi_k}{3 \Lambda L_W} \rho_k  \leq \tfrac{\xi_{k+1} }{3 \Lambda L_W}. 
\end{align*}
 \eqref{eq.desired.distance.induction} is therefore satisfied on iteration $k+1$. We then use induction to verify that $\xi_k<1$ for $k\ge 1$. Considering the base case where $k = 1$, we find that
 \begin{align*}
\rho_0=\tfrac{\xi_0}{6}+\tfrac{1}{4}+\tfrac{1}{2}=
\tfrac{11}{12}<1, \;\; \text{ and } \;\;
\xi_1=\max \left\{\xi_0 \rho_0, \beta_{1}^{1 / 4}\right\}=\max \left\{\rho_0, \beta_{1}^{1 / 4}\right\}<1 .
\end{align*}
Let us assume that $\xi_k<1$ for some $k>1$. By the fact that $\left\{\left|S_k^H\right|\right\}$ is increasing and $\left\{\beta_{k}\right\}$ is decreasing, we have
 \begin{align*}
\xi_{k+1} \leq \max \left\{\rho_k, \beta_{1}^{(k+1) / 4}\right\}<1, \;\  \rho_k=\tfrac{\xi_k}{6}+\tfrac{1}{4}  \tfrac{N-|S_k^H|}{N-|S_0^H|}+\tfrac{\beta_{k}^{k/4}}{2 }\leq \tfrac{1}{6}+\tfrac{1}{4}+\tfrac{1}{2}=\tfrac{11}{12}<1.
\end{align*}
This substantiates that $\rho_k \leq 11 / 12$, which implies $\left\{\xi_k\right\} \to 0$ along with the fact that $\beta_{k+1} < \beta_k \le 1$. 
Consequently, it follows by \eqref{def.rho.tau} that $\left\{\rho_k\right\} \rightarrow 0$.
Based on these findings, we deduce that
 \begin{align*}
\lim _{k \rightarrow \infty} \tfrac{\xi_{k+1}}{\xi_k} 
 =\lim _{k \rightarrow \infty} \max \left\{\rho_k, \tfrac{\beta_{k+1}^{(k+1) / 4}}{\xi_k}\right\} & \leq \lim _{k \rightarrow \infty} \max \left\{\rho_k,\left(\tfrac{\beta_{k+1}}{\beta_{k}}\right)^{k / 4} \beta_{k+1}^{1 / 4}\right\} = 0. 
\end{align*}
\end{proof}

\section{Appendix: Proof of Lemma  \ref{lemma.MINRES.res.convergence}} \label{app.minres.res.prove}
\begin{proof}
Since $\bar M_k$ is nonsingular by assumption \ref{ass.function}, \ref{ass.M.inverse.bound.subsampled}, and \cite[Lemma 16.1]{NoceWrig06}. Let $\Theta(\bar M_k)$ denote the set of eigenvalues of $\bar M_k$ and  $\mathcal{E}_{\lambda}(\bar M_k)$ denote the eigenspace corresponding to an eigenvalue $\lambda \in \Theta(\bar M_k)$. Suppose that $\Theta_{+,\max}, \Theta_{+,\min}$ are the maximum and the minimum of positive element in  $\Theta(\bar M_k) \backslash   \Theta_{\perp}(\bar M_k)  $, $\Theta_{-,\max}, \Theta_{-,\min}$ are the negative element in  maximum and the minimum in  $\Theta(\bar M_k) \backslash   \Theta_{\perp}(\bar M_k) $. Let
$\kappa_i^+ := \tfrac{\Theta_{+,\max}}{\Theta_{+,\min}}$,  $\kappa_i^- := \tfrac{\Theta_{-,\min}}{\Theta_{-,\max}}$, and $\theta = \max\left\{ \tfrac{\sqrt{\kappa_i^+} - 1 }{\sqrt{\kappa_i^+} + 1 }, \tfrac{\sqrt{\kappa_i^-} - 1 }{\sqrt{\kappa_i^-} + 1 } \right\}$. The proof then follows directly from \cite[Lemma 4]{liu2022newton} by the fact that $\bar M_k$ does not have zero eigenvalues and taking $i = \psi_+$, $j = \psi_+ + \psi_0 + 1$ in its context.  
\end{proof}

\end{document}